\documentclass[letterpaper]{amsart}
\pdfoutput=1

\usepackage{amssymb}
\usepackage{amsfonts}
\usepackage{amsmath}
\usepackage{amsthm}
\usepackage{mathtools}
\usepackage{graphicx}
\usepackage{mathrsfs}
\usepackage{dsfont}
\usepackage{amscd}
\usepackage{multirow}

\usepackage[all]{xy}
\normalfont
\usepackage[T1]{fontenc}
\usepackage{calligra}
\usepackage{verbatim}
\usepackage{fourier}
\usepackage[usenames,dvipsnames]{color}
\usepackage[colorlinks=true,linkcolor=Blue,citecolor=Violet]{hyperref}
\usepackage{enumerate}
\usepackage{slashed}
\usepackage{nicematrix}
\usepackage{microtype}
\usepackage{tikz}\usetikzlibrary{snakes,math}
\usepackage{caption}
\usepackage{subcaption}
\usepackage{faktor}

\newcommand{\N}{{\mathds{N}}}
\newcommand{\Z}{{\mathds{Z}}}
\newcommand{\Q}{{\mathds{Q}}}
\newcommand{\R}{{\mathds{R}}}
\newcommand{\C}{{\mathds{C}}}
\newcommand{\T}{{\mathds{T}}}

\newcommand{\D}{{\mathfrak{D}}}
\newcommand{\A}{{\mathfrak{A}}}
\newcommand{\B}{{\mathfrak{B}}}

\newcommand{\Nbar}{\overline{\N}}
\newcommand{\Lip}{{\mathsf{L}}}
\newcommand{\TLip}{{\mathsf{T}}}
\newcommand{\SLip}{{\mathsf{S}}}
\newcommand{\QLip}{{\mathsf{Q}}}
\newcommand{\Hilbert}{{\mathscr{H}}}

\newcommand{\dpropinquity}[1]{{\mathsf{\Lambda}^\ast_{#1}}}

\newcommand{\dmetpropinquity}[1]{{\mathsf{\Lambda}^{\ast\mathsf{met}}_{#1}}}

\newcommand{\spectralpropinquity}[1]{{\mathsf{\Lambda}^{\mathsf{spec}}_{#1}}}

\newcommand{\LipschitzD}{{\mathsf{LipD}}}

\newcommand{\Kantorovich}[1]{{\mathsf{mk}_{#1}}}

\newcommand{\Haus}[1]{{\mathsf{Haus}\!\left[{#1}\right]\,}}

\newcommand{\StateSpace}{{\mathscr{S}}}

\newcommand{\MongeKant}{{Mon\-ge-Kan\-to\-ro\-vich metric}}

\newcommand{\gMVB}{metrical C*-correspondence}

\newcommand{\mcc}[3]{{\mathrm{metCor}\left({#1},{#2},{#3}\right)}}

\newcommand{\qcms}{quantum compact metric space}

\newcommand{\unit}{1}

\newcommand{\sa}[1]{{\mathfrak{sa}\left({#1}\right)}}

\newcommand{\inner}[3]{{\left<{#1},{#2}\right>_{#3}}}

\newcommand{\dom}[1]{{\operatorname*{dom}\left({#1}\right)}}
\newcommand{\codom}[1]{{\operatorname*{codom}\left({#1}\right)}}
\newcommand{\diam}[2]{{\mathrm{diam}\left({#1},{#2}\right)}}
\newcommand{\qdiam}[2]{{\,\mathrm{qdiam}\left({#1},{#2}\right)}}

\newcommand{\norm}[2]{\left\|{#1}\right\|_{#2}}

\newcommand{\targetsettunnel}[3]{{\mathfrak{t}_{#1}\left({#2}\middle\vert{#3}\right)}}

\newcommand{\CDN}{{\mathsf{DN}}}
\newcommand{\TDN}{{\mathsf{TN}}}

\newcommand{\worknote}[1]{}
\newcommand{\opnorm}[3]{{\left|\mkern-1.5mu\left|\mkern-1.5mu\left| {#1} \right|\mkern-1.5mu\right|\mkern-1.5mu\right|_{#3}^{#2}}}

\newcommand{\tunnellength}[1]{{\lambda\left({#1}\right)}}
\newcommand{\tunnelmagnitude}[2]{{\mu\left({#1}\middle\vert{#2}\right)}}

\newcommand{\tunnelextent}[1]{{\chi\left({#1}\right)}}

\newcommand{\alg}[1]{{\mathfrak{#1}}}

\newcommand{\module}[1]{{\mathscr{#1}}}

\newcommand{\ModState}[1]{\widehat{\StateSpace}}

\newcommand{\closure}[1]{\mathrm{cl}\left({#1}\right)}

\renewcommand{\geq}{\geqslant}
\renewcommand{\leq}{\leqslant}

\newcommand{\Siep}{{Sierpi{\'n}ski}}

\newcommand{\Dirac}{{\slashed{D}}}

\newcommand{\dualZmod}[1]{\widehat{\faktor{\Z}{#1}}}

\theoremstyle{plain}
\newtheorem{theorem}{Theorem}[section]
\newtheorem*{theorem*}{Theorem}

\newtheorem{corollary}[theorem]{Corollary}

\newtheorem{lemma}[theorem]{Lemma}
\newtheorem{proposition}[theorem]{Proposition}

\newtheorem{theorem-definition}[theorem]{Theorem-Definition}

\theoremstyle{definition}
\newtheorem{definition}[theorem]{Definition}
\newtheorem*{definition*}{Definition}

\newtheorem{notation}[theorem]{Notation}

\newtheorem{convention}[theorem]{Convention}

\theoremstyle{remark}
\newtheorem*{acknowledgement}{Acknowledgements}

\newtheorem{example}[theorem]{Example}

\newtheorem{remark}[theorem]{Remark}

\newtheorem{claim}[theorem]{Claim}

\numberwithin{equation}{section}

\newenvironment{claimproof}[1]{\noindent \emph{Proof of Claim (\ref{#1}).}}{This proves our claim. \hfill {\sc Q.E.D.}}

\begin{document}

\title[]{Convergence of inductive sequences of spectral triples for the spectral propinquity}

\author{Carla Farsi}
\email{carla.farsi@colorado.edu}
\address{Department of Mathematics \\ University of Colorado at Boulder \\ Boulder CO 80309-0395}

\author{Fr\'{e}d\'{e}ric Latr\'{e}moli\`{e}re}
\email{frederic@math.du.edu}
\urladdr{http://www.math.du.edu/\symbol{126}frederic}
\address{Department of Mathematics \\ University of Denver \\ Denver CO 80208}

\author{Judith Packer}
\email{judith.jesudason@colorado.edu}
\address{Department of Mathematics \\ University of Colorado at Boulder \\ Boulder CO 80309-0395}

\date{\today}
\subjclass[2000]{Primary:  46L89, 46L30, 58B34.}
\keywords{Spectral triples, Noncommutative metric geometry, quantum Gromov-Hausdorff distance, Monge-Kantorovich distance, Quantum Metric Spaces, Quantum Tori, Noncommutative solenoids, Bunce-Deddens algebras.}

\begin{abstract}
  In the context of metric geometry, we introduce a new necessary and sufficient condition for the convergence of an inductive sequence of quantum compact metric spaces for the Gromov-Hausdorff propinquity, which is a noncommutative analogue of the Gromov-Hausdorff distance for compact metric spaces. This condition is easy to verify in many examples, such as quantum compact metric spaces associated to AF algebras or certain twisted convolution C*-algebras of discrete inductive limit groups. Our condition also implies the convergence of an inductive sequence of spectral triples in the sense of the spectral propinquity, a generalization of the Gromov-Hausdorff propinquity on quantum compact metric spaces to the space of metric spectral triples. In particular we show the convergence of the state spaces of the underlying C*-algebras as quantum compact metric spaces, and also the convergence of the quantum dynamics induced by the Dirac operators in the spectral triples. We apply these results to new classes of inductive limit of even spectral triples on noncommutative solenoids and Bunce-Deddens C*-algebras. Our construction, which involves length functions with bounded doubling, adds geometric information and highlights the structure of these twisted C*-algebras as inductive limits.
\end{abstract}
\maketitle
\tableofcontents


\section{Introduction}

Spectral triples, introduced by Connes in 1985 as a noncommutative generalization of Dirac operators acting on bundles over manifolds \cite{Connes89,Connes}, have emerged as a powerful means to encode geometric information over noncommutative operator algebras. Motivated in part by ideas from mathematical physics, and by the recurrent usefulness of various notions of limits of C*-algebras, the second author introduced in \cite{Latremoliere18g} a distance on metric spectral triples, up to an obvious notion of unitary equivalence, thus enabling the discussion of approximations of certain spectral triples by others, in a geometric sense. This distance is named the spectral propinquity, and is built from a noncommutative analogue of the Gromov-Hausdorff distance for noncommutative geometry, called the Gromov-Hausdorff propinquity \cite{Latremoliere13b,Latremoliere13,Latremoliere15,Latremoliere14}. Thus, convergence of spectral triples is defined as part of a larger framework for convergence of {\qcms s}, which are noncommutative analogues of algebras of Lipschitz functions over compact metric spaces. Within this framework, the propinquity was extended to certain modules over {\qcms s} \cite{Latremoliere16}, and even C*-correspondences \cite{Latremoliere18d} with additional metric data inspired by metric connections. The propinquity also was extended to various dynamical systems \cite{Latremoliere18c,Latremoliere18b}. These extensions have been used by the second author to define the spectral propinquity over metric spectral triples.

The spectral propinquity $\spectralpropinquity{}$ has been applied to approximations of spectral triples on fractals \cite{Latremoliere20a} and on quantum tori \cite{Latremoliere21a}, with the latter example rooted in matrix models in physics and the problem of their convergence. Indeed, the spectral propinquity endows the space of all metric spectral triples with its own geometry, and it allows to capture some geometric intuition within the well understood framework of a topology. For instance, while quantum tori are not inductive limits of finite dimensional C*-algebras, spectral triples over quantum tori can now be approximated by spectral triples over full matrix algebras to arbitrary precision using the spectral propinquity --- a common heuristics in mathematical physics, now formalized. Convergence for the spectral propinquity implies convergence of the state spaces of the underlying algebras for a form of Gromov-Hausdorff distance, convergence of the quantum dynamics obtained by exponentiating the Dirac operators, and implies convergence of the spectra and the bounded continuous functional calculus for the Dirac operators, with implications for the convergence of physically important quantities such as the spectral actions \cite{Latremoliere22}.

\medskip

In this paper, we consider the question of when an \emph{inductive sequence of metric spectral triples} \cite{Floricel19} converges, in the sense of the spectral propinquity, to its inductive limit. To illustrate the power of our result, besides the class of AF algebras, we construct even metric spectral triples on noncommutative solenoids \cite{Latremoliere11c} and on some Bunce-Deddens algebras \cite{BunceDeddens75,Davidson} and show that they are limits of metric spectral triples on, respectively, quantum tori and bundles of full matrix algebras over the circle, in the sense of the spectral propinquity $\spectralpropinquity{}$. In this way, we provide a noncommutative geometric version of the fact that solenoid groups can be seen as metric limits of tori, and Bunce-Deddens algebras are metric limits of algebras of matrix valued functions over the circle.

\medskip

A spectral triple $(\A,\Hilbert,\Dirac)$ is given by a unital C*-algebra $\A$ acting on a Hilbert space $\Hilbert$ and a (usually unbounded) self-adjoint operator $\Dirac$ on $\Hilbert$, which has bounded commutator with the elements of a dense $\ast$-subalgebra of $\A$, and has compact resolvent (see Definition (\ref{spectral-triples-def})). Spectral triples contain much geometric information, including metric data. Indeed, Connes noted in \cite{Connes} that spectral triples define a canonical extended pseudo-distance on the state space of their underlying C*-algebras, which, in particular, recovers the geodesic distance when working with the usual spectral triple given by the Dirac operator acting on the square integrable sections of the spinor bundle of a compact connected Riemannian spin manifold without boundary.

\medskip

Rieffel in \cite{Rieffel98a,Rieffel99} then cast this metric aspect of noncommutative geometry under a new light, starting from the observation that Connes' distance induced by a spectral triple is a noncommutative analogue of the {\MongeKant} \cite{Kantorovich40,Kantorovich58}; it was thus natural to define a {\qcms} as an ordered pair $(\A,\Lip)$ of a unital C*-algebra $\A$ and a noncommutative analogue of a Lipschitz seminorm $\Lip$ such that, in particular, if we set, for any two states $\varphi,\psi$ of $\A$,
\begin{equation*}
  \Kantorovich{\Lip}(\varphi,\psi)\coloneqq\sup\left\{ |\varphi(a)-\psi(a)| : \Lip(a) \leq 1 \right\}
\end{equation*}
then $\Kantorovich{\Lip}$ is a distance inducing the weak-$^\ast$ topology on the state space of $\A$. The exact list of requirements on the seminorm $\Lip$ have evolved as the study of noncommutative metric geometry matured, and we will use the definition of a \emph{\qcms} given in \cite{Latremoliere13,Latremoliere15} and recalled in Definition (\ref{qcms-def}). Indeed, a spectral triple whose Connes' metric induces the weak-$^\ast$ topology on the state space of its underlying C*-algebra then automatically gives a {\qcms}; such a spectral triple is called a \emph{metric spectral triple}.

\medskip

Metric spectral triples may thus be studied within the context of noncommutative metric geometry. As a result, the second author introduced a distance on the space of metric spectral triples. The first step in defining this distance, called the spectral propinquity, is the construction of a noncommutative geometric analogue of the \emph{Gromov-Hausdorff distance} \cite{Edwards75,Gromov81,Gromov} between {\qcms s}, which we will recall in subsection (\ref{GH-subsection}). The first such analogue was introduced by Rieffel \cite{Rieffel00}, motivated by the possibility of formalizing certain convergence results found in the mathematical physics literature. While several such analogues have been offered, we will work with the \emph{Gromov-Hausdorff propinquity} $\dpropinquity{}$, introduced by the second author in \cite{Latremoliere13b,Latremoliere13,Latremoliere15,Latremoliere14} precisely to be well adapted to C*-algebras theory and the type of seminorms given by spectral triples. The propinquity in general is designed precisely to enable distance computations between {\qcms s} defined on unrelated C*-algebras, such as between matrix algebra and quantum tori. However, in this work, we investigate what additional properties of the propinquity we can derive when we work with inductive limits of C*-algebras.

\medskip

We begin this work by establishing a characterization of convergence of inductive limits of {\qcms s} to their inductive limit, in terms of \emph{bridge builders}, a type of $\ast$-automorphism with a natural relation to quantum metrics.

\begin{definition*}[Definition (\ref{bridge-builder-def})]
  For each $n\in\N\cup\{\infty\}$, let $(\A_n,\Lip_n)$ be a {\qcms}, such that $\A_\infty = \closure{\bigcup_{n\in\N} \A_n}$, where $(\A_n)_{n\in\N}$ is an increasing (for $\subseteq$) sequence of C*-subalgebras of $\A_\infty$, with the unit of $\A_\infty$ in $\A_0$. 
  
  A $\ast$-automorphism $\pi :\A_\infty \rightarrow\A_\infty$ is a \emph{bridge builder} for $((\A_n,\Lip_n)_{n\in\N},(\A_\infty,\Lip_\infty))$ when, for all $\varepsilon > 0$, there exists $N\in\N$ such that if $n\geq N$, then
  \begin{equation*}
    \forall a \in \dom{\Lip_\infty} \quad \exists b \in \dom{\Lip_n}: \quad \Lip_n(b) \leq \Lip_\infty(a) \text{ and }\norm{\pi(a) - b}{\A_\infty} < \varepsilon\Lip_\infty(a) \\
  \end{equation*}
  and
  \begin{equation*}
    \forall b \in \dom{\Lip_n} \quad \exists a \in \dom{\Lip_\infty}: \quad \Lip_\infty(a)\leq\Lip_n(b) \text{ and }\norm{\pi(a)-b}{\A_\infty}<\varepsilon\Lip_n(b) \text,
  \end{equation*}
  where $\norm{\cdot}{\A_\infty}$ is the C*-norm on $\A_\infty$.
\end{definition*}

Bridge builders are powerful means to prove metric convergence for the propinquity and notable because it is usually very difficult to find necessary conditions for metric convergence in the sense of the propinquity (besides the trivial convergence for the diameters). Thus, this theorem is of independent interest from our study of spectral triples, and addresses the relationship between inductive limits and limits in a metric sense as in \cite{Latremoliere18g, Latremoliere13b}. Our first main result is therefore the following theorem about convergence for the propinquity $\dpropinquity{}$ of certain inductive sequences.

\begin{theorem*}[Theorem (\ref{main-thm})]
  For each $n\in\N\cup\{\infty\}$, let $(\A_n,\Lip_n)$ be a {\qcms}, where $(\A_n)_{n\in\N}$ is an increasing (for $\subseteq$) sequence of C*-subalgebras of $\A_\infty$ such that $\A_\infty = \closure{\bigcup_{n\in\N} \A_n}$, with the unit of $\A_\infty$ in $\A_0$. We assume that there exists  $\exists M > 0 $ such that for all $n \in \N$:
  \begin{equation*}
    \quad \frac{1}{M} \Lip_n \leq \Lip_\infty \leq M \cdot \Lip_n \text{ on $\dom{\Lip_n}$.}
  \end{equation*}
  Then 
  \begin{equation*}
    \lim_{n\rightarrow\infty} \dpropinquity{}\left((\A_n,\Lip_n),(\A_\infty,\Lip_\infty)\right) = 0\text,
  \end{equation*}
  if, and only if, for any subsequence $(\A_{g(n)},\Lip_{g(n)})_{n\in\N}$ of $(\A_n,\Lip_n)_{n\in\N}$, there exists a strictly increasing function $f : \N \rightarrow \N$ and a \emph{bridge builder} $\pi$ for $((\A_{g\circ f(n)},\Lip_{g\circ f(n)})_{n\in\N},(\A_\infty,\Lip_\infty))$.
\end{theorem*}

\medskip

The second step in the construction of the spectral propinquity $\spectralpropinquity{}$ on the space of metric spectral triples is the extension of the Gromov-Hausdorff propinquity to a distance on the class of C*-correspondences over {\qcms s} endowed with a form of quantum metric, and with a compatible action of some monoid. The C*-correspondence associated with a metric spectral triple $(\A,\Hilbert,\Dirac)$ is the Hilbert space $\Hilbert$, seen as a $\A$-$\C$-C*-correspondence, with the quantum metric given by the graph norm of $\Dirac$, and with the action of $[0,\infty)$ on $\Hilbert$ given by $t \in [0,\infty) \mapsto \exp(i t \Dirac)$. Convergence for the spectral propinquity, by design, implies the convergence of the underlying {\qcms s}, but the converse does not hold in general. These matters will be recalled in detail in Subsection (\ref{SpProp-sec}).

\medskip

We then turn to the more specific context of inductive sequences of metric spectral triples. Inductive sequences of spectral triples were introduced in \cite{Floricel19}, and are a natural source of spectral triples; our interest is in the convergence of such sequences for the spectral propinquity, i.e. in the sense of an actual metric. We establish in the present work, as our second main result, that an inductive sequence of metric spectral triples converges for the spectral propinquity when there exists a fully quantum isometric bridge builder for the underlying sequence of {\qcms s}. Again, it is a surprising result that a mild strengthening of convergence for the Gromov-Hausdorff propinquity implies the much stronger convergence for the spectral propinquity, a fact which does not hold for arbitrary sequences of metric spectral triples, but holds thanks to the structure of inductive limits. Our second main theorem is given as follows.

\begin{theorem*}[Theorem (\ref{main-spec-thm})]
  Let $(\A_\infty,\Hilbert_\infty,\Dirac_\infty)$ be a metric spectral triple which is the inductive limit of a sequence $(\A_n,\Hilbert_n,\Dirac_n)_{n\in\N}$ of metric spectral triples, in the sense of Definition (\ref{ind-limit-spectral-triple-def}). For each $n\in\N\cup\{\infty\}$, let
  \begin{equation*}
    \dom{\Lip_n} \coloneqq \left\{a \in \A_n : a=a^\ast, a\,\dom{\Dirac_n}\subseteq\dom{\Dirac_n} \text{ and }[\Dirac_n,a]\text{ is bounded} \right\}\text,
  \end{equation*}
  and, for all $a \in \dom{\Lip_n}$, let $\Lip_n(a)$ be the operator norm of $[\Dirac_n,a]$.

  If there exists a bridge builder $\pi : (\A_\infty,\Lip_\infty) \rightarrow (\A_\infty,\Lip_\infty)$ for $((\A_n,\Lip_n)_{n\in\N},(\A_\infty,\Lip_\infty))$  which is a full quantum isometry of $(\A_\infty,\Lip_\infty)$, i.e. such that $\pi(\dom{\Lip_\infty})\subseteq\dom{\Lip_\infty}$ and $\Lip_\infty\circ\pi=\Lip_\infty$ on $\dom{\Lip_\infty}$, then
  \begin{equation*}
    \lim_{n\rightarrow\infty} \spectralpropinquity{}((\A_n,\Hilbert_n,\Dirac_n),(\A_\infty,\Hilbert_\infty,\Dirac_\infty)) = 0 \text.
  \end{equation*}
\end{theorem*}

\medskip

We conclude our paper with the construction of new even spectral triples on certain twisted group C*-algebras $C^\ast(G,\sigma)$ where the discrete group $G=\bigcup_{n\in\N}G_n$ is the union of a strictly increasing sequence of subgroups $G_n$ of $G$. These examples include noncommutative solenoids \cite{Latremoliere11c} and certain Bunce-Deddens algebras \cite{BunceDeddens75}. Our construction is motivated by the desire to see our new spectral triples over $C^\ast(G,\sigma)$ as limits, for the spectral propinquity, of an inductive sequence of metric spectral triples constructed over the inductive sequence $\left(C^\ast(G_n,\sigma)\right)_{n\in\N}$. This metric aspect distinguishes our spectral triples from other spectral triples on noncommutative solenoids \cite{Guido17,Luef21} or Bunce-Deddens algebras \cite{Hawkins13-2}, and is applicable, in principle, to many other examples. Moreover, noncommutative solenoids were shown in \cite{Latremoliere16} to be limits, for the propinquity, of quantum tori, for a different family of quantum metrics which did not come from a spectral triple.

\medskip

In general, it is difficult to prove that a given spectral triple is metric. Examples of metric spectral triples can be found over certain manifolds, quantum tori \cite{Connes,Sitarz13,Sitarz15,Latremoliere15c,Latremoliere21a}, or more generally, over unital C*-algebras endowed with ergodic actions of compact Lie groups \cite{Gabriel16,Rieffel98a}, over certain C*-crossed-products \cite{Hawkins13-2}, over quantum groups \cite{Dabrowski05}, over Podle\'s spheres \cite{Kaad18}, over AF algebras \cite{Antonescu04}, over certain fractals \cite{Lapidus08,Lapidus14}, and more. We note that there are known examples of spectral triples which are not metric \cite{Putnam16}.

\medskip

It is therefore quite interesting to obtain new examples of metric spectral triples, and moreover, to prove that they are interesting limits of spectral triples for the spectral propinquity. We thus establish the following third main result of this paper, which draws on the first two in its proof.
\begin{theorem*}[Simplified form of Theorem (\ref{main-Abelian-group-thm})]
  Let $G = \bigcup_{n\in\N} G_n$ be an Abelian discrete group, with $(G_n)_{n\in\N}$ a strictly increasing sequence of subgroups of $G$. Let $\sigma$ be a $2$-cocycle of $G$, with values in $\T\coloneqq\{z\in\C:|z|=1\}$.
  
  Let $\mathds{L}_H$ be a length function over $G$ whose restriction to $G_n$ is proper for all $n\in\N$, such that the sequence $(G_n)_{n\in\N}$ converges to $G$ for the Hausdorff  distance induced on the closed subsets of $G$ by $\mathds{L}_H$. Let
  \begin{equation*}
    \mathds{F} : g \in G \longmapsto \mathrm{scale}(\min\{n\in\N : g \in G_n\})\text,
  \end{equation*}
  where $\mathrm{scale} : \N \rightarrow [0,\infty)$ is a strictly increasing function.
  
  If the proper length function $\mathds{L} \coloneqq \max\{ \mathds{L}_H, \mathds{F} \}$ satisfies that, for some $\theta>1$, there exists $c>0$ such that  for all $ r \geq 1$:
  \begin{equation*}
    \quad \left|\left\{g\in G : \mathds{L}(g) \leq \theta \cdot r\right\}\right| \leq c \left| \left\{ g \in G : \mathds{L}(g)\leq r \right\} \right| \text,
  \end{equation*}
  then
  \begin{equation*}
    \lim_{n\rightarrow\infty} \spectralpropinquity{}((C^\ast(G,\sigma),\ell^2(G)\otimes \C^2,\Dirac),(C^\ast(G_n,\sigma),\ell^2(G_n)\otimes \C^2,\Dirac_n)) = 0 \text,
  \end{equation*}
  where for all $n\in\N\cup\{\infty\}$ and for all $(\xi_1,\xi_2)$ in
  \begin{equation*}
    \left\{ \xi \in \ell^2(G_n)\otimes\C^2 : \sum_{g\in G_n} (\mathds{L}_H(g)^2 + \mathds{F}(g)^2)\norm{\xi(g)}{\C^2}^2 < \infty \right\}\text,
  \end{equation*}
  we set
  \begin{equation*}
    \Dirac\xi : g \in G \longmapsto \begin{pmatrix} \mathds{F}(g)\xi_2(g) + \mathds{L}_H(g) \xi_1(g) \\ \mathds{F}(g)\xi_2(g) - \mathds{L}_H(g)\xi_1(g) \end{pmatrix}\text.
  \end{equation*}
  In the above spectral triples, $C^\ast(G,\sigma)$ and $C^\ast(G_n,\sigma)$ act via their left regular $\sigma$-projective representations.
\end{theorem*}

We then apply this theorem to construct metric spectral triples on noncommutative solenoids, i.e. the twisted group C*-algebras $C^\ast\left(\left(\Z\left[\frac{1}{p}\right]\right)^2,\sigma\right)$ where
\begin{equation*}
  \Z\left[\frac{1}{p}\right]\coloneqq\left\{\frac{k}{p^n}:k\in\Z,n\in\N \right\}\text,
  \end{equation*}
  with $p$ a prime natural number, and where $\sigma$ is a $2$-cocycle of $\left(\Z\left[\frac{1}{p}\right]\right)^2$. In this case, using the notation of the above theorem, we choose $\mathds{L}_H$ to be the restriction to $\left(\Z\left[\frac{1}{p}\right]\right)^2$ of any norm on $\R^2$, while $\mathds{F}$ can be chosen by setting $\mathds{F}(g)\coloneqq p^{\min\left\{n\in\N : g\in\left(\frac{1}{p^n}\Z\right)^2\right\}}$ for all $g = (g_1,g_2) \in \left(\Z\left[\frac{1}{p}\right]\right)^2$. Alternatively, following the ideas of \cite{FarsiPacker22}, which motivated the present work, we can choose $\mathds{F}(g_1,g_2)\coloneqq \max\{|g_1|_p , |g_2|_p\}$ for all $g_1,g_2 \in \Z\left[\frac{1}{p}\right]$, where $|\cdot|_p$ is the $p$-adic absolute value.

\medskip
  
Similarly, we can apply \cite{Packer89} to see that the Bunce-Deddens algebras are given as the twisted group C*-algebra $C^\ast\left(\Z(\alpha)\times \Z,\sigma\right)$ for an appropriate choice of a $2$-cocycle $\sigma$ and a sequence $\alpha=(\alpha_n)_{n\in\N}$ of nonzero natural numbers such that $\frac{\alpha_{n+1}}{\alpha_n}$ is a prime number for all $n\in\N$, where the group $\Z(\alpha)$ is the subgroup of the circle group $\T$ given by all roots of unity of order $\alpha_n$ for $n$ ranging over $\N$. We endow $\Z(\alpha)$ with the discrete topology. The supernatural number number describing the $\ast$-isomorphism class of the Bunce-Deddens algebra thus obtained is $\left(p^{|\{n\in\N:\frac{\alpha_{n+1}}{\alpha_n}=p\}|}\right)_{p \text{ prime }}$. For our purpose, we will work with sequences $\alpha$ for which $\left(\frac{\alpha_{n+1}}{\alpha_n}\right)_{n\in\N}$ is bounded. In this case, we will choose $\mathds{L}_H$ to be the sum or the max (or one of many other choices) of the restriction of a length function over $\T$ to $\Z(\alpha)$, and the absolute value on $\Z$. Observing that
\begin{equation*}
  \Z(\alpha) = \bigcup_{n\in\N} \dualZmod{\alpha_n}\text,
\end{equation*}
where $\dualZmod{m}$ is the group of all $m$-th roots of unity, we then set $\mathds{F}(\zeta,z)\coloneqq \min\{ \alpha_n : \zeta \in \dualZmod{\alpha_n}\}$ for all $(\zeta,z)\in \Z(\alpha)\times\Z$. This provides a new way to look at Bunce-Deddens algebras as limits of algebras of continuous sections of bundles of matrix algebras over circles in a geometric sense, as an echo of the topological fact that they are A$\T$ algebras. This work thus provides an approach to endowing Bunce-Deddens algebras with a different quantum metric from \cite{Latremoliere20a}, with the advantage that our quantum metrics are induced by spectral triples --- solving the main difficulty in \cite{Latremoliere20a}, at least for these Bunce-Deddens algebras to which our present work applies.

\begin{acknowledgement}
	This work was partially supported by the Simons Foundation (Simons Foundation collaboration grant \#523991 [C. Farsi] and \# 31698 [J. Packer].)
\end{acknowledgement}

\section{A Characterization of Convergence in the Propinquity for Inductive Sequences}

We introduce in this section the notion of \emph{bridge builders} associated with inductive sequences of {\qcms s}, which can be used to characterize the convergence of such sequences to their inductive limits in the sense of the Gromov-Hausdorff propinquity. We begin with a review of the notions of {\qcms s} and propinquity, and then we prove our main theorem, which underlies all the rest of our work.

\subsection{Preliminaries: the Gromov-Hausdorff Propinquity}\label{GH-subsection}

Our work is concerned with {\qcms s}, which are noncommutative analogues of the algebras of Lipschitz functions over a compact metric space. Our definition is the result of a natural evolution from the notion of compact quantum metric spaces introduced in \cite{Rieffel98a} by Rieffel, designed as the natural context for the construction of the propinquity. This subsection will also set some of the basic notation which we will use throughout this paper.

\begin{notation}
  By default, we denote the norm of a normed vector space $E$ by $\norm{\cdot}{E},$ and for us, the set $\N$ of natural numbers always contains zero.
\end{notation}

\begin{notation}
  If $\A$ is a unital C*-algebra, then the unit of $\A$ will simply be denoted by $1$. The state space of the C*-algebra $\A$ is denoted by $\StateSpace(\A)$. For any $a\in \A$, we write $\Re a = \frac{a+a^\ast}{2}$ and $\Im a = \frac{a-a^\ast}{2i}$. The space $\{ a \in \A : a=a^\ast \}$ is denoted by $\sa{\A}$ and is closed under the Jordan product $a,b\in \sa{\A} \mapsto \Re(ab)$ and the Lie product $a,b\in\sa{\A}\mapsto \Im(ab)$, making $\sa{\A}$ a Jordan-Lie algebra.
\end{notation}

\begin{definition}[{\cite{Connes89,Latremoliere13,Latremoliere15,Rieffel98a,Rieffel00,Rieffel10c}}]\label{qcms-def}
  Fix $\Omega\geq 1$ and $\Omega'\geq 0$. An {$(\Omega,\Omega')$-\qcms} $(\A,\Lip)$ is given by a unital C*-algebra $\A$ and a seminorm $\Lip$ defined on a dense Jordan-Lie subalgebra $\dom{\Lip}$ of $\sa{\A}$ such that:
  \begin{enumerate}
  \item $\{ a \in \dom{\Lip} : \Lip(a) = 0 \} = \R \unit$,
  \item the {\MongeKant} $\Kantorovich{\Lip}$, defined on the state space $\StateSpace(\A)$ of $\A$, by, for all
   $\varphi,\psi \in \StateSpace(\A) $:
    \begin{equation*}
     \quad \Kantorovich{\Lip}(\varphi,\psi) \coloneqq \sup\left\{ |\varphi(a) - \psi(a)| : a\in\dom{\Lip}, \Lip(a) \leq 1 \right\} 
    \end{equation*}
    is a metric which induces the weak-$^\ast$ topology on $\StateSpace(\A)$,
  \item for all $a,b \in \sa{\A}$,
    \begin{equation*}
      \max\left\{\Lip(\Re(ab)),\Lip(\Im(ab))\right\} \leq \Omega\left( \norm{a}{\A}\Lip(b) + \Lip(a)\norm{b}{\A} \right) + \Omega' \Lip(a)\Lip(b)\text;
    \end{equation*}
    this inequality being referred to as the $(\Omega,\Omega')$-Leibniz inequality,
  \item the set $\{ a \in \dom{\Lip} : \Lip(a) \leq 1 \}$ is closed in $\A$.
  \end{enumerate}

  Any such a seminorm $\Lip$ is called a \emph{Lipschitz seminorm} on $\A$.
\end{definition}

\begin{convention}
  By convention, if $\Lip$ is a Lipschitz seminorm on some unital C*-algebra $\A$, we will write $\Lip(a) = \infty$ whenever $a \notin \dom{\Lip}$, with the convention that $0\infty=0$ and $\infty+x=x+\infty=\infty$ for all $x\in [0,\infty]$. With this convention, $\Lip$ is lower semicontinuous over $\sa{\A}$ as a $[0,\infty]$-valued function (not just on $\dom{\Lip}$ but on the entire space $\sa{\A}$).
\end{convention}

\begin{convention}
  Throughout this paper, we fix $\Omega\geq 1$ and $\Omega' \geq 0$. These parameters will be implicit in our notation; when working with spectral triples, one may always assume $\Omega=1$ and $\Omega'=0$.
\end{convention}

\begin{remark}\label{qdiam-remark}
  If $(\A,\Lip)$ is a {\qcms}, then we record the following fact which we shall use repeatedly: if $a\in\dom{\Lip}$, then $\Lip(a+t 1) = \Lip(a)$ for all $t\in\R$, since
  \begin{equation*}
    \Lip(a) = \Lip(a+t1 - t1) \leq \Lip(a+t1) +\Lip(t1) = \Lip(a+t1) + t\underbracket[1pt]{\Lip(1)}_{=0} = \Lip(a+t1) \leq \Lip(a) + t\Lip(1) = \Lip(a) \text.
  \end{equation*}
\end{remark}

Since the state space of a {\qcms} is a compact metric space for the {\MongeKant}, it has bounded diameter. Moreover, its diameter can used to obtain a natural bound on the norm of some self-adjoint elements, which is a simple but very useful result, which we now recall.

\begin{notation}
  The diameter of a metric space $(E,d)$ is denoted by $\diam{E}{d}$. If $(\A,\Lip)$ is a {\qcms}, then we will write $\qdiam{A}{\Lip}$ for $\diam{\StateSpace(A)}{\Kantorovich{\Lip}}$. If $E$ is actually a normed vector space, then we simply write $\diam{A}{E}$ for the diameter of any subset $A$ of $E$ for the norm $\norm{\cdot}{E}$ of $E$.
\end{notation}

We recall the following fact, which we will use repeatedly.
\begin{theorem}[{\cite[Propostion 1.6]{Rieffel98a}}]\label{norm-Lip-bound-thm}
  If $(\A,\Lip)$ is a {\qcms}, and if $\mu \in \StateSpace(\A)$, then $\norm{a-\mu(a)1}{\A} \leq \Lip(a) \, \qdiam{\A}{\Lip}$.
\end{theorem}

\begin{proof}
  For all $\varphi\in\StateSpace(\A)$, we note that $|\varphi(a-\mu(a)1)| = |\varphi(a)-\mu(a)| \leq \Lip(a) \qdiam{\A}{\Lip}$. Since $a-\mu(a)1$ is self-adjoint, we conclude that $\norm{a-\mu(a)1}{\A} \leq \Lip(a)\qdiam{\A}{\Lip}$.
\end{proof}

The property difficult to establish when working with {\qcms s} is, of course, that the {\MongeKant} induces the weak-$^\ast$ topology. Rieffel provided various characterizations; we will find the following helpful in this paper:
\begin{theorem}[{\cite{Ozawa05}}]\label{tt-thm}
  Let $\Lip$ be a seminorm defined on some dense subspace $\dom{\Lip}$ of $\sa{\A}$ for some unital C*-algebra $\A$ such that $\{a\in\dom{\Lip}:\Lip(a)=0\}=\R1$. If we set $\Kantorovich{\Lip}(\varphi,\psi) = \sup\left\{|\varphi(a)-\psi(a)| : a\in\dom{\Lip},\Lip(a) \leq 1\right\}$, for all $\varphi,\psi \in \StateSpace(\A)$, then the following assertions are equivalent:
  \begin{itemize}
  \item $\Kantorovich{\Lip}$ is a metric on the state space $\StateSpace(\A)$ of $\A$ inducing the weak-$^\ast$ topology,
  \item there exists a state $\mu \in \StateSpace(\A)$ such that $\{a\in\dom{\Lip}:\Lip(a)\leq 1,\mu(a) = 0\}$ is totally bounded in $\sa{\A}$,
  \item for all states $\mu \in \StateSpace(\A)$, the set $\{a\in\dom{\Lip}:\Lip(a)\leq 1,\mu(a) = 0\}$ is totally bounded in $\sa{\A}$.
  \end{itemize}
\end{theorem}

We record the following helpful result, which we will also use often.
\begin{corollary}[{\cite{Rieffel98a}}]\label{boundedL-cor}
  If $(\A,\Lip)$ is a {\qcms}, $\mu \in \StateSpace(\A)$, and if $K > 0$, then the set
  \begin{equation*}
    \left\{ a \in \dom{\Lip} : \Lip(a)\leq 1, |\mu(a)|\leq K \right\}
  \end{equation*}
  is compact in $\A$.
\end{corollary}

\begin{proof}
  We first note that the set $\left\{ a \in \dom{\Lip} : \Lip(a)\leq 1, |\mu(a)|\leq K \right\}$ is closed since $\Lip$ is lower semicontinuous and $\mu$ is continuous.
  
  Let $(a_n)_{n\in\N}$ be a sequence in $\dom{\Lip}$ such that $\Lip(a_n)\leq 1$ and $|\mu(a_n)|\leq K$ for all $n\in\N$. Since $(|\mu(a_n)|)_{n\in\N}$ is bounded in $\R$, it has a convergent subsequence $(|\mu(a_{f(n)})|)_{n\in\N}$.

  On the other hand, $(a_{f(n)}-\mu(a_{f(n)})1)_{n\in\N}$ has a convergent subsequence $(a_{f(g(n))}-\mu(a_{f(g(n))}))_{n\in\N}$ by Theorem (\ref{tt-thm}). It now follows that $(a_{f(g(n))})_{n\in\N}$ is a convergent sequence.
\end{proof}

\medskip

Quantum compact metric spaces are the points of a (pseudo-)metric space, where the metric is the Gromov-Hausdorff propinquity, an analogue of the Gromov-Hausdorff distance in noncommutative geometry. The construction of the propinquity thus relies on an appropriate notion of \emph{quantum isometries}.

\begin{definition}\label{quantum-iso-def}
  Let $(\A_1,\Lip_1)$ and $(\A_2,\Lip_2)$ be two {\qcms s}. A \emph{Lipschitz morphism} $\pi:(\A_1,\Lip_1)\rightarrow(\A_2,\Lip_2)$ from $(\A_1,\Lip_1)$ to $(\A_2,\Lip_2)$ is a surjective $\ast$-morphism $\pi$ from $\A_1$ to $\A_2$ such that $\pi(\dom{\Lip_1})\subseteq\dom{\Lip_2}$. Moreover, if, for all $b\in\dom{\Lip_2}$:
  \begin{equation*}
    \Lip_2(b) = \inf\left\{ \Lip_1(a) : \pi(a) = b \right\} \text,
  \end{equation*}
  then $\pi$ is called a \emph{quantum isometry}. If $\pi$ is a quantum isometry and a bijection whose inverse is also a quantum isometry, then $\pi$ is called a \emph{full quantum isometry}; in this case $\pi$ is a $\ast$-isomorphism such that for all $ a \in \sa{\A_1} $:
  \begin{equation*}
    \quad \Lip_2\circ\pi(a) = \Lip_1(a) \text.
  \end{equation*}
\end{definition}

The propinquity is a metric computed by isometrically ``embedding'' two {\qcms s} into an arbitrary third one, which in the contravariant picture of noncommutative geometry, leads us to the following definition for a \emph{tunnel}. Crucially, a non-negative number can be associated to a tunnel using the Hausdorff distance.

\begin{notation}
  The Hausdorff distance induced by the distance function of a metric space $(X,d)$ on the hyperspace of closed subsets of $X$ is denoted by $\Haus{d}$. If $N$ is a norm on a vector space, we denote by $\Haus{N}$ the Hausdorff distance induced by the metric given by the norm $N$. By default, if $E$ is a normed vector space, we simplify our notation and simply write $\Haus{E}$ for the Hausdorff distance induced by the distance defined by the norm $\norm{\cdot}{E}$ of $E$.
\end{notation}

\begin{notation}
  If $\pi : \A \rightarrow \B$ is a unital $\ast$-morphism, then we define
  \begin{equation*}
    \pi^\ast : \varphi \in \StateSpace(\B) \longmapsto \varphi\circ\pi \in \StateSpace(\A) \text.
  \end{equation*}
\end{notation}

\begin{definition}[{\cite[Definition 3.1]{Latremoliere13b},\cite[Definition 2.11,Definition 3.6]{Latremoliere14}}]\label{tunnel-def}
  Let $(\A_1,\Lip_1)$ and $(\A_2,\Lip_2)$ be two {\qcms s}. A \emph{tunnel} $\tau = (\D,\Lip_\D,\pi_1,\pi_2)$ is given by a {\qcms} $(\D,\Lip_\D)$ and two quantum isometries $\pi_1 : (\D,\Lip_\D) \rightarrow (\A_1,\Lip_1)$ and $\pi_2 : (\D,\Lip_\D) \rightarrow (\A_2,\Lip_2)$. The \emph{domain} $\dom{\tau}$ of $\tau$ is $(\A_1,\Lip_1)$ and the \emph{codomain} $\codom{\tau}$ of $\tau$ is $(\A_2,\Lip_2)$.

  The \emph{extent} $\tunnelextent{\tau}$ of $\tau$ is the non-negative number:
  \begin{equation*}
    \tunnelextent{\tau} \coloneqq \max_{j \in \{1,2\}} \Haus{\Kantorovich{\Lip_\D}}\left(\pi_j^\ast(\StateSpace(\A_j)), \StateSpace(\D)\right) \text.
  \end{equation*}
\end{definition}

\begin{remark}
  We emphasize that all {\qcms s} involved in our tunnels in this paper must satisfy the same $(\Omega,\Omega')$-Leibniz inequality for our \emph{fixed} $\Omega,\Omega'$.
\end{remark}

There always exists a tunnel between any two {\qcms s}, and the extent of a tunnel is always finite. We thus define:

\begin{definition}\label{propinquity-def}
  The \emph{(dual) Gromov-Hausdorff propinquity} $\dpropinquity{}((\A,\Lip_\A),(\B,\Lip_\B))$ between any two {\qcms s} $(\A,\Lip_\A)$ to $(\B,\Lip_\B)$ is defined by:
  \begin{equation*}
    \dpropinquity{}((\A,\Lip_\A),(\B,\Lip_\B)) \coloneqq 
    \inf\left\{ \tunnelextent{\tau} : \tau \text{ tunnel from $(\A,\Lip_\A)$ to $(\B,\Lip_\B)$ } \right\} \text.
  \end{equation*}
\end{definition}

The (dual) propinquity is well-behaved, as summarized in the following theorem:
\begin{theorem}[{\cite{Latremoliere13,Latremoliere13b}}]
  The dual propinquity is a complete metric up to full quantum isometry. Moreover, if $(X_n,d_n)_{n\in\N}$ is a sequence of compact metric spaces, then $(X_n,d_n)_{n\in\N}$ converges to a compact metric space $(X,d)$ for the Gromov-Hausdorff distance if, and only if $\lim_{n\rightarrow\infty}\dpropinquity{}((C(X_n),\Lip_{d_n}),(C(X),\Lip_{d})) = 0$, where $\Lip_d$ denotes the Lipschitz seminorm induced by any metric $d$.
\end{theorem}

There are several interesting known examples of convergence for the propinquity, including approximations of quantum tori by fuzzy tori \cite{Latremoliere13c}, approximations of spheres by matrix algebras \cite{Rieffel15b}, continuity of quantum tori in their cocycle parameter \cite{Latremoliere13c}, continuity of UHF algebras with respect to the Baire space seen as their natural parameter space, continuity of the Effros-Shen algebras in their irrational parameters \cite{Latremoliere15d}, and more.

\subsection{Main result}

We begin with a simple sufficient condition to ensure that a seminorm is indeed a Lipschitz seminorm on an inductive limit of unital C*-algebras, when each of the C*-subalgebra in the inductive sequence is already equipped with a Lipschitz seminorm. This condition is quite natural and generalizes, for instance, the idea behind the construction of Lipschitz seminorms on AF algebras in \cite{Latremoliere15d}.

\begin{proposition}\label{sufficient-cqms-prop}
  Let $\A_\infty$ be a unital C*-algebra. For each $n\in\N$, let $(\A_n,\Lip_n)$ be a {\qcms}, where $(\A_n)_{n\in\N}$ is an increasing sequence of C*-subalgebras of $\A_\infty$ with the unit of $\A_\infty$ in $\A_0$. Assume moreover that $\A_\infty = \closure{\bigcup_{n\in\N}\A_n}$. Let $\Lip_\infty$ be a seminorm defined on a dense Jordan-Lie subalgebra $\dom{\Lip_\infty}$ of $\sa{\A_\infty}$, such that:
  \begin{enumerate}
  \item $\{a\in\dom{\Lip_\infty} :\Lip_\infty(a)=0\} = \R\unit$,
  \item the unit ball of $\Lip_\infty$ is closed in $\A_\infty$,
  \item $\Lip_\infty$ is $(\Omega,\Omega')$-Leibniz.
  \end{enumerate}

  If there exists a unital isometric positive linear map $\pi : \A_\infty \rightarrow \A_\infty$ such that, for all $\varepsilon > 0$, there exists $N\in \N$ with the property that:
  \begin{equation*}
    \forall a \in \dom{\Lip_\infty} \quad \exists b \in \dom{\Lip_N}: \quad \Lip_N(b)\leq \Lip_\infty(a) \text{ and }\norm{\pi(a) - b}{\A_\infty} < \varepsilon \Lip_\infty(a)\text,
  \end{equation*}
  then $(\A_\infty,\Lip_\infty)$ is a {\qcms}.
\end{proposition}

\begin{proof}
  Let $\mu \in\StateSpace(\A_\infty)$. By assumption, $\mu\in\StateSpace(\A_n)$ for all $n\in\N$ --- where we use the same symbol $\mu$ to denote the restriction of $\mu$ to $\A_n$. Let
  \begin{equation*}
    B_\infty \coloneqq \left\{ a \in \dom{\Lip_\infty} : \mu\circ\pi(a) = 0, \Lip_\infty(a) \leq 1 \right\}\text.
  \end{equation*}

  \medskip
  
  Now, let $\varepsilon > 0$ and let $n\in\N$. We set
  \begin{equation*}
    B_n \coloneqq \left\{ a \in \dom{\Lip_n} : |\mu(a)| < \frac{\varepsilon}{4}, \Lip_n(a) \leq 1 \right\} \text.
  \end{equation*}
  Let $a \in B_n$, and let $\varphi\in\StateSpace(\A_n)$. By Theorem (\ref{norm-Lip-bound-thm}), we have the following inclusion:
  \begin{equation*}
    B_n \subseteq \left\{ a \in \dom{\Lip_n} : \Lip_n(a)\leq 1,\norm{a}{\A_n}\leq \qdiam{\A_n}{\Lip_n} + \frac{\varepsilon}{4} \right\}
  \end{equation*}
  and the latter set is compact since $\Lip_n$ is a Lipschitz seminorm, by Corollary (\ref{boundedL-cor}). So $B_n$ is totally bounded. In fact, since $\Lip_n$ is lower semicontinuous and $\mu$ is continuous, the set $B_n$ is also closed in the complete space $\A_\infty$, so $B_n$ is compact.
  
  By assumption on $\pi$, there exists $N\in\N$ such that
  \begin{equation*}
    \forall a \in B_\infty \quad \exists b \in \dom{\Lip_N}: \quad \Lip_N(b)\leq 1 \text{ and }\norm{\pi(a) - b}{\A_\infty} < \frac{\varepsilon}{4}\text.
  \end{equation*}
  In particular, if $a\in B_\infty$ and $b \in \dom{\Lip_N}$ with $\Lip_N(b)\leq 1$ and $\norm{\pi(a)-b}{\A_\infty} < \frac{\varepsilon}{4}$, then $|\mu(b)| \leq \norm{b-\pi(a)}{\A_\infty}+|\mu(\pi(a))| < \frac{\varepsilon}{4}$, so $b \in B_N$.

  Since $B_N$ is compact in $\sa{\A_N}$ by Corollary (\ref{boundedL-cor}), there exists a $\frac{\varepsilon}{4}$-dense subset $F\subseteq B_N$ of $B_N$. So
  \begin{equation*}
    \Haus{\A_\infty}(\pi(B_\infty),F)\leq\Haus{\A_\infty}(\pi(B_\infty),B_N) + \Haus{\A_\infty}(B_N,F) <\frac{\varepsilon}{2}\text.
  \end{equation*}

  The domain $\dom{\Lip_\infty}$ is dense in $\sa{\A}$, so it is not empty and thus $\{ a\in\dom{\Lip_\infty} : \Lip_\infty(a)\leq 1\}$ is not empty, since $\Lip$ is a seminorm. Thus, by Remark (\ref{qdiam-remark}), the set $B_\infty$ is not empty as well. We thus obtain:
  \begin{equation*}
    \emptyset \neq B_\infty = \bigcup_{b \in F} \left\{ a \in B_\infty : \norm{\pi(a)-b}{\A_\infty} < \frac{\varepsilon}{2} \right\} \text.
  \end{equation*}
  Therefore, if we define
  \begin{equation*}
    G\coloneqq \left\{ b \in F : \left\{ a \in B_\infty : \norm{\pi(a)-b}{\A_\infty} < \frac{\varepsilon}{2} \right\} \neq \emptyset \right\} \text,
  \end{equation*}
  then $G\neq\emptyset$ and $B_\infty = \bigcup_{b \in G} \left\{ a \in B_\infty : \norm{\pi(a)-b}{\A_\infty} < \frac{\varepsilon}{2} \right\}$. For each $b \in G$, we pick $t(b) \in B_\infty$ such that $\norm{\pi(t(b))-b}{\A_\infty} < \frac{\varepsilon}{2}$. Let now $a \in B_\infty$. There exists $b \in G$ such that $\norm{\pi(a)-b}{\A_\infty} < \frac{\varepsilon}{2}$. Then
  \begin{align*}
    \norm{a-t(b)}{\A_\infty}
    &= \norm{\pi(a-t(b))}{\A_\infty} \\
    &\leq \norm{\pi(a)-b}{\A_\infty} + \norm{b-\pi(t(b))}{\A_\infty} \\
    &< \frac{\varepsilon}{2} + \frac{\varepsilon}{2} = \varepsilon \text.
  \end{align*}
  Thus, $t(G)$ is a $\varepsilon$-dense subset of $B_\infty$. So $B_\infty$ is totally bounded in $\A_\infty$. Therefore, noting that $\mu\circ\pi$ is a state of $\A_\infty$, we conclude by Theorem (\ref{tt-thm}) that $\Kantorovich{\Lip_\infty}$ induces the weak-$^\ast$ topology on $\StateSpace(\A_\infty)$. Since all other required properties are assumed, $\Lip_\infty$ is indeed a Lipschitz seminorm.
\end{proof}

The next natural question is to find a sufficient condition to strengthen Proposition (\ref{sufficient-cqms-prop}) and obtain convergence of the sequence $(\A_n,\Lip_n)_{n\in\N}$ to $(\A_\infty,\Lip_\infty)$ in the sense of the propinquity. To this end, we introduce the notion of a bridge builder --- a map which, among other things, satisfy the condition in Proposition (\ref{sufficient-cqms-prop}). In fact, we basically ``symmetrize'' the condition in Proposition (\ref{sufficient-cqms-prop}) and require that we work with $\ast$-morphism (which will allow us to construct seminorms with the Leibniz property), rather than just positive linear maps.

\begin{notation}
  We will write $\Nbar \coloneqq \N\cup\{\infty\}$ for the one point compactification of $\N$.
\end{notation}

\begin{definition}\label{bridge-builder-def}
  For each $n\in\N\cup\{\infty\}$, let $(\A_n,\Lip_n)$ be a {\qcms}, where $(\A_n)_{n\in\N}$ is an increasing (for $\subseteq$) sequence of C*-subalgebras of $\A_\infty$ such that $\A_\infty = \closure{\bigcup_{n\in\N} \A_n}$ and the unit of $\A_\infty$ is in $\A_0$.
  
  A $\ast$-automorphism $\pi :\A_\infty \rightarrow\A_\infty$ is a \emph{bridge builder} for $((\A_n,\Lip_n)_{n\in\N},(\A_\infty,\Lip_\infty))$ when, for all $\varepsilon > 0$, there exists $N\in\N$ such that if $n\geq N$, then
  \begin{equation*}
    \forall a \in \dom{\Lip_\infty} \quad \exists b \in \dom{\Lip_n}: \quad \Lip_n(b) \leq \Lip_\infty(a) \text{ and }\norm{\pi(a) - b}{\A_\infty} < \varepsilon\Lip_\infty(a) \\
  \end{equation*}
  and
  \begin{equation*}
    \forall b \in \dom{\Lip_n} \quad \exists a \in \dom{\Lip_\infty}: \quad \Lip_\infty(a)\leq\Lip_n(b) \text{ and }\norm{\pi(a)-b}{\A_\infty}<\varepsilon\Lip_n(b) \text.
  \end{equation*}
\end{definition}

\begin{proposition}\label{bridge-builder-prop}
  For each $n\in\N\cup\{\infty\}$, let $(\A_n,\Lip_n)$ be a {\qcms}, where $(\A_n)_{n\in\N}$ is an increasing (for $\subseteq$) sequence of C*-subalgebras of $\A_\infty$ such that $\A_\infty = \closure{\bigcup_{n\in\N} \A_n}$ and the unit of $\A_\infty$ is in $\A_0$.
  
  If there exists a bridge builder for $((\A_n,\Lip_n)_{n\in\N},(\A_\infty,\Lip_\infty))$, then
  \begin{equation*}
    \lim_{n\rightarrow\infty}\dpropinquity{}((\A_n,\Lip_n),(\A_\infty,\Lip_\infty)) = 0\text.
  \end{equation*}
\end{proposition}

\begin{proof}
  Let $\pi : \A_\infty\rightarrow\A_\infty$ be the given bridge builder. Let $\varepsilon > 0$. There exists $N\in\N$ such that if $n\geq N$, then
  \begin{itemize}
  \item $\forall a \in \dom{\Lip_\infty} \quad \exists b \in \dom{\Lip_n}: \quad \Lip_n(b) \leq \Lip_\infty(a) \wedge \norm{\pi(a)-b}{\A_\infty} < \varepsilon\Lip_\infty(a) \text,$
  \item $\forall b \in \dom{\Lip_n} \quad \exists a \in \dom{\Lip_\infty}: \quad \Lip_\infty(a) \leq \Lip_n(b) \wedge \norm{\pi(a)-b}{\A_\infty} < \varepsilon\Lip_n(b) \text.$
  \end{itemize}
  
  Fix $n\geq N$. We define, for all $a\in\dom{\Lip_\infty}$ and $b\in \dom{\Lip_n}$:
  \begin{equation*}
    \TLip_n(a,b) \coloneqq \max\left\{ \Lip_\infty(a), \Lip_n(b), \frac{1}{\varepsilon}\norm{\pi(a)-b}{\A_\infty} \right\} \text.
  \end{equation*}

  It is a standard argument that $\left(\A_\infty\oplus\A_n,\TLip_n\right)$ is a {\qcms}:
  \begin{enumerate}
  \item the domain $\dom{\TLip_n}  = \dom{\Lip_\infty}\oplus\dom{\Lip_n}$ of $\TLip_n$ is dense in $\sa{\A_\infty\oplus\A_n}$ since $\dom{\Lip_\infty}$ is dense in $\sa{\A_\infty}$ and $\dom{\Lip_n}$ is dense in $\sa{\A_n}$,
  \item if $\TLip_n(a,b) = 0$ for some $(a,b)\in\dom{\TLip_n}$, then $\Lip_\infty(a) = 0$ so $a=t \unit$ for some $t \in\R$, and $\Lip_n(b) = 0$ so $b = s \unit$ for some $s \in \R$ (it matters here that the unit is the same in $\A_\infty$ and $\A_n$), and $0 = \norm{\pi(a)-b}{\A_\infty} = |t-s|$ so $(a,b) = t(\unit,\unit)$;
  \item $\TLip_n$ is the maximum of two lower semicontinuous functions and one continuous function, so it is lower semicontinuous over $\sa{\A_\infty\oplus\A_n}$;
  \item a direct computation shows that $\TLip_n$ is $(\Omega,\Omega')$-Leibniz since $\Lip_\infty$ and $\Lip_n$ both are, and $\pi$ is a $\ast$-morphism;
  \item fixing any state $\mu$ of $\A_\infty$ and setting $\varphi :(a,b) \in \A_\infty\oplus\A_n\mapsto \mu(a)$, then $\varphi \in \StateSpace(\A_\infty\oplus\A_n)$, and
    \begin{multline*}
      \left\{ (a,b) \in \dom{\TLip_n} : \TLip_n(a,b) \leq 1, \varphi(a,b) = 0 \right\} \subseteq \\
      \left\{a\in\dom{\Lip_\infty}:\Lip_\infty(a),\mu(a) = 0 \right\} \times \left\{ b \in \dom{\Lip_n} : \Lip_n(b) \leq 1, |\mu\circ\pi^{-1}(b)| \leq \varepsilon \right\}
    \end{multline*}
    and, as seen in the proof of Proposition (\ref{sufficient-cqms-prop}), the set on the right hand side is a product of two compact set, and thus compact; thus the set on the left hand side is compact (closed in a compact set) and thus, $\TLip_n$ is indeed a Lipschitz seminorm, invoking Theorem (\ref{tt-thm}).
  \end{enumerate}
  
  We now check that $\tau_n \coloneqq (\A_\infty\oplus\A_n,\TLip_n,\psi_n,\theta_n)$, with $\psi_n : (a,b) \in \A_\infty\oplus\A_n \mapsto a \in \A_\infty$ and $\theta_n : (a,b)\in\A_\infty\oplus\A_n\mapsto b \in \A_n$, is a tunnel, in the sense of Definition (\ref{tunnel-def}).

    Let $a\in\dom{\Lip_\infty}$. By assumption, there exists $b \in \dom{\Lip_n}$ with $\Lip_n(b) \leq \Lip_\infty(a)$ and $\norm{\pi(a)-b}{\A_\infty} < \varepsilon\Lip_\infty(a)$. Therefore, $\TLip_n(a,b) = \Lip_\infty(a)$. Since by construction, $\TLip_n(a,c)\geq\Lip_\infty(a)$ for all $a\in\dom{\Lip_\infty}$ and $c\in\dom{\Lip_n}$, we have shown that $\psi_n$ is a quantum isometry by Definition (\ref{quantum-iso-def}).

    Let now $b \in \dom{\Lip_n}$.  Again by assumption on $\pi$, there exists $a\in \dom{\Lip_\infty}$ such that $\norm{\pi(a)-b}{\A_\infty}<\varepsilon\Lip_n(b)$ and $\Lip_\infty(a)\leq\Lip_n(b)$. Thus $\TLip_n(a,b)=\Lip_n(b)$. Once again, $\TLip_n(c,b)\geq\Lip_n(b)$ by construction for all $c\in\dom{\Lip_\infty}$, so $\theta_n$ is indeed a quantum isometry, so $\tau_n$ is a tunnel.

    \medskip
    
    We now compute the extent of $\tau_n$, in the sense of Definition (\ref{tunnel-def}). Let $\varphi \in \StateSpace(\A_\infty\oplus\A_n)$. Using Hahn-Banach theorem, we extend $\varphi$ to a state $\varphi'$ of $\A_\infty\oplus\A_\infty$. Let $\mu:a\in\A_\infty\mapsto \varphi'(a,\pi(a))$; since $\pi$ is a unital $\ast$-morphism, $\mu$ is a state of $\A_\infty$. By construction, if $\TLip_n(a,b)\leq 1$ then $\norm{\pi(a)-b}{\A_\infty} \leq \varepsilon$ and thus
    \begin{align*}
      |\varphi(a,b) - \mu\circ\psi_n(a,b)|
      &= |\varphi'(a,b) - \varphi'(a,\pi(a))| \\
      &\leq |\varphi'(0,b-\pi(a))| \\
      &\leq \norm{b-\pi(a)}{\A_\infty} \leq \varepsilon \text.
    \end{align*}
    Thus $\Haus{\Kantorovich{\TLip_n}}(\psi_n^\ast(\StateSpace(\A_\infty)),\StateSpace(\A_\infty\oplus\A_n)) \leq\varepsilon$.
    
    Let now $\mu' : b\in\A_n\mapsto\varphi(\pi^{-1}(b),b)$. Since $\pi$ is a $\ast$-automorphism of $\A_\infty$, the map $\mu'$ is a state of $\A_n$. Moreover:
    \begin{align*}
      |\varphi(a,b) - \mu'\circ\theta_n(a,b)|
      &= |\varphi(a,b) - \varphi(\pi^{-1}(b),b)|\\
      &= |\varphi(a-\pi^{-1}(b),0)| \\
      &\leq \norm{a-\pi^{-1}(b)}{\A_\infty} \\
      &= \norm{\pi(a)-b}{\A_\infty} \leq \varepsilon \text.
    \end{align*}
    
    Thus $\Haus{\Kantorovich{\TLip_n}}(\theta_n^\ast(\StateSpace(\A_n)),\StateSpace(\A_\infty)) \leq \varepsilon$.

    Hence, the extent $\tunnelextent{\tau_n}$ of $\tau_n$ is at most $\varepsilon$. By Definition (\ref{propinquity-def}), we thus have shown that for all $n\geq N$,
    \begin{equation}\label{bb-pf-eq}
      \dpropinquity{}((\A_n,\Lip_n),(\A_\infty,\Lip_\infty)) \leq \varepsilon \text,
    \end{equation}
    which concludes our proof.
\end{proof}

\medskip

Our main result in this section is the following theorem, which shows that the natural sufficient condition in Definition (\ref{bridge-builder-def}) and Proposition (\ref{bridge-builder-prop}) is, in fact, very close to necessary, under a mild and natural condition. This is notable because in general, it is difficult to exhibit nontrivial necessary conditions for convergence in the sense of the propinquity (besides, say, the fact that diameters must converge). It also shows that the existence of bridge builders is the natural setup for establishing convergence of inductive limits in the sense of the propinquity, thus providing a complete answer for the relationship between convergence of inductive sequences of {\qcms s} in the categorical sense and the propinquity sense, under a commonly met condition. 

\begin{theorem}\label{main-thm}
   For each $n\in\N\cup\{\infty\}$, let $(\A_n,\Lip_n)$ be a {\qcms}, where $(\A_n)_{n\in\N}$ is an increasing (for $\subseteq$) sequence of C*-subalgebras of $\A_\infty$ such that $\A_\infty = \closure{\bigcup_{n\in\N} \A_n}$ and the unit of $\A_\infty$ is in $\A_0$. We assume that there exists $ M > 0$ such that  for all  $n \in \N $:
  \begin{equation*}
    \quad \frac{1}{M} \Lip_n \leq \Lip_\infty \leq M \cdot \Lip_n \text{ on $\dom{\Lip_n}$.}
  \end{equation*}
  Then 
  \begin{equation*}
    \lim_{n\rightarrow\infty} \dpropinquity{}\left((\A_n,\Lip_n),(\A_\infty,\Lip_\infty)\right) = 0\text,
  \end{equation*}
  if, and only if, for any subsequence $(\A_{g(n)},\Lip_{g(n)})_{n\in\N}$ of $(\A_n,\Lip_n)_{n\in\N}$, there exists a strictly increasing function $f : \N \rightarrow \N$ and a \emph{bridge builder} $\pi$ for $((\A_{g\circ f(n)},\Lip_{g\circ f(n)})_{n\in\N},(\A_\infty,\Lip_\infty))$.
\end{theorem}

\begin{proof}
  First, assume that for any subsequence $(\A_{g(n)},\Lip_{g(n)})_{n\in\N}$, there exists a strictly increasing function $f : \N \rightarrow \N$ and a bridge builder $\pi$ for $((\A_{g\circ f(n)},\Lip_{g\circ f(n)})_{n\in\N},(\A_\infty,\Lip_\infty))$. By Proposition (\ref{bridge-builder-prop}), we conclude that every subsequence of $(\A_n,\Lip_n)_{n\in\N}$ has a subsequence converging to $(\A_\infty,\Lip_\infty)$. Therefore, $(\A_n,\Lip_n)_{n\in\N}$ converges to $(\A_\infty,\Lip_\infty)$ since the propinquity is, indeed, a metric (up to full quantum isometry).

  \medskip

    Let us now assume that $(\A_n,\Lip_n)_{n\in\N}$ converges to $(\A_\infty,\Lip_\infty)$ for the propinquity. Since any subsequence will converge as well, it is sufficient to prove our statement for $g$ being the identity, and this will simplify our notation.

    Since $(\A_n,\Lip_n)_{n\in\N}$ converges to $(\A_\infty,\Lip_\infty)$, there exists a sequence
    \begin{equation*}
      (\tau_n)_{n\in\N} \coloneqq (\D_n,\TLip_n, \psi_n,\theta_n)_{n\in\N}
    \end{equation*}
    of tunnels, as in Definition (\ref{tunnel-def}),  with $\lim_{n\rightarrow\infty}\tunnelextent{\tau_n} = 0$, while, for each $n\in\N$, we have $\dom{\tau_n} = (\A_\infty,\Lip_\infty)$ and $\codom{\tau_n} = (\A_n,\Lip_n)$. To ease notation, the target set of $a\in\dom{\Lip_\infty}$ with $l\geq \Lip_\infty(a)$ defined by $\tau_n$ will be denoted by $\targetsettunnel{n}{a}{l}$, rather than $\targetsettunnel{\tau_n}{a}{l}$; we recall from \cite{Latremoliere13b,Latremoliere13} that:
    \begin{equation*}
      \targetsettunnel{n}{a}{l} = \left\{ \theta_n(d) : d \in \psi_n^{-1}(\{a\}), \TLip_n(d) \leq l \right\} \text.
    \end{equation*}
    This proof heavily relies on the properties of target sets, as discussed in \cite{Latremoliere13b,Latremoliere13,Latremoliere15,Latremoliere14}. In \cite{Latremoliere13b}, various estimates which we will refer to in this proof are expressed using the \emph{length} $\tunnellength{\tau}$ of a tunnel $\tau$, rather than the extent $\tunnelextent{\tau}$; however as seen in \cite[Proposition 2.12]{Latremoliere14}, for any tunnel $\tau$, we have $\tunnellength{\tau}\leq\tunnelextent{\tau}\leq 2\tunnellength{\tau}$. We will use this inequality without further mention to express all our results here in terms of extents.

    \begin{claim}\label{singleton-claim}
      For all $a\in \dom{\Lip_\infty}$, there exists a strictly increasing function $f : \N\rightarrow \N$ and an element $\pi(a) \in \dom{\Lip_\infty}$ such that, for all $l\geq \Lip_\infty(a)$,
      \begin{equation*}
        \lim_{n\rightarrow\infty} \Haus{\A_\infty}\left( \targetsettunnel{f(n)}{a}{l}, \{\pi(a)\}  \right) = 0 \text,
      \end{equation*}
      and $\norm{\pi(a)}{\A_\infty} = \norm{a}{\A_\infty}$.
    \end{claim}
    
    \begin{claimproof}{singleton-claim}
    First, since the sequence $(\tunnelextent{\tau_n})_{n\in\N}$ converges (to $0$), it is bounded; let $K' > 0$ such that $\tunnelextent{\tau_n}\leq K'$ for all $n\in\N$.
    
    Let $a \in \dom{\Lip_\infty}$. Let $l = \Lip_\infty(a)$. For any $K > 0$, let
    \begin{equation*}
      \A_\infty[K] \coloneqq \left\{ b \in\dom{\Lip_\infty} : \Lip_\infty(b)\leq K, \norm{b}{\A_\infty}\leq \norm{a}{\A_\infty} + K K'  \right\}\text.
    \end{equation*}
    The set $\A_\infty[K]$ is compact in $\sa{\A_\infty}$ by Corollary (\ref{boundedL-cor}).    
    By \cite[Corollary 4.5]{Latremoliere13b} and since $\Lip_\infty\leq M \Lip_n$ on $\dom{\Lip_n}$, the sequence $(\targetsettunnel{n}{a}{l})_{n\in\N}$ is a sequence of compact subsets of $\A_\infty[M l]$, and
    \begin{equation*}
      \lim_{n\rightarrow\infty} \diam{\targetsettunnel{n}{a}{l}}{\A_\infty} = 0\text.
    \end{equation*}

    Since $\A_\infty[M l]$ is compact in $\A_\infty$, the Hausdorff distance $\Haus{\A_\infty}$ induced on the set of closed subsets of $\A_\infty[M l]$ by the norm $\norm{\cdot}{\A_\infty}$ of $\A_\infty$ gives a compact topology as well. Therefore, there exists a subsequence $(\targetsettunnel{f(n)}{a}{l})_{n\in\N}$ of $(\targetsettunnel{n}{a}{l})_{n\in\N}$ which converges, for $\Haus{\A_\infty}$, to a singleton $\{ \pi(a) \}$ of $\A_\infty[M l]$. In particular, $\Lip_\infty(\pi(a)) \leq M l = M \Lip_\infty(a)$.
  
  \medskip
  
  Let now $L \geq l$. By definition, $\targetsettunnel{f(n)}{a}{l} \subseteq \targetsettunnel{f(n)}{a}{L}$ for all $n\in\N$ and
  \begin{equation*}
    \lim_{n\rightarrow\infty} \diam{\targetsettunnel{f(n)}{a}{L}}{\A_\infty} = 0 \text,
  \end{equation*}
  so we conclude easily as well that
  \begin{equation*}
    \lim_{n\rightarrow\infty} \Haus{\A_\infty}\left(\targetsettunnel{f(n)}{a}{L},\{\pi(a)\} \right) = 0 \text.
  \end{equation*}
  
  By \cite[Proposition 4.4]{Latremoliere13b},  we also note that if $b_n \in \targetsettunnel{f(n)}{a}{l}$ for each $n\in\N$, then
  \begin{equation*}
      \norm{\pi(a)}{\A_\infty} = \lim_{n\rightarrow\infty}\norm{b_n}{\A_\infty} \leq \limsup_{n\rightarrow\infty} \left( \norm{a}{\A_\infty} + \tunnelextent{\tau_{f(n)}} l \right) = \norm{a}{\A_\infty} \text.
    \end{equation*}

    Similarly, since $a\in\targetsettunnel{\tau_{f(n)}^{-1}}{b_n}{l}$, we also have
    \begin{equation*}
      \norm{a}{\A_\infty} \leq \limsup_{n\rightarrow\infty} \left( \norm{b_n}{\A_\infty} + l \tunnelextent{\tau_{f(n)}} \right) = \norm{\pi(a)}{\A_\infty} \text.
    \end{equation*}

    So indeed, $\norm{\pi(a)}{\A_\infty} = \norm{a}{\A_\infty}$.
  \end{claimproof}
  
  \begin{claim}\label{JL-morphism-claim}
    There exists a unital $\ast$-endomorphism $\pi$ of $\A_\infty$ such that $\pi(\dom{\Lip_\infty})\subseteq \dom{\Lip_\infty}$, and a strictly increasing function $f : \N\rightarrow\N$ such that, for all $a\in\dom{\Lip_\infty}$, and for all $l\geq \Lip_\infty(a)$, 
    \begin{equation*}
      \lim_{n\rightarrow\infty} \Haus{\A_\infty}\left(\targetsettunnel{f(n)}{a}{l},\{\pi(a)\}\right) = 0 \text.
    \end{equation*}
    
  \end{claim}

  \begin{claimproof}{JL-morphism-claim}
    Since $\A_\infty$ is separable, there exists a countable dense subset $S_\infty$ of $\sa{\A_\infty}$ with $S_\infty\subseteq\dom{\Lip_\infty}$. Using Claim (\ref{singleton-claim}), a diagonal argument shows that there exists a strictly increasing sequence $f : \N \rightarrow \N$ such that, for all $a\in S_\infty$ and for all $l\geq \Lip_\infty(a)$, we have $\lim_{n\rightarrow\infty}\Haus{\A_\infty}\left(\targetsettunnel{f(n)}{a}{l},\{\pi(a)\}\right) = 0$.
    
    \medskip
    
    Let now $a\in\dom{\Lip_\infty}$, and let $l \geq \Lip_\infty(a)$. Let $\varepsilon > 0$. Since $S_\infty$ is dense in $\dom{\Lip_\infty}$, there exists $a_\varepsilon \in \dom{\Lip_\infty}$ such that $\norm{a-a_\varepsilon}{\A_\infty} < \frac{\varepsilon}{5}$. Note that $\Lip_\infty(a_\varepsilon) < \infty$ but in general, there is no relation between $\Lip_\infty(a_\varepsilon)$ and $\Lip_\infty(a)$.
    
    Let $l = \max\{\Lip_\infty(a),\Lip_\infty(a_\varepsilon)\}$. Since it is convergent for the Hausdorff distance $\Haus{\A_\infty}$, the sequence $\left(\targetsettunnel{f(n)}{a_\varepsilon}{l}\right)_{n\in\N}$ is Cauchy for $\Haus{\A_\infty}$.
    
    Therefore, there exists $N \in \N$ such that, for all $p,q \geq N$, we have
    \begin{equation*}
      \Haus{\A_\infty}\left(\targetsettunnel{f(p)}{a_\varepsilon}{l}, \targetsettunnel{f(q)}{a_\varepsilon}{l}\right) < \frac{\varepsilon}{5} \text.
    \end{equation*}
    
    Since $\lim_{n\rightarrow\infty} \tunnelextent{\tau_{f(n)}} = 0$, there exists $N' \in \N$ such that if $n\geq N'$ then $\tunnelextent{\tau_{f(n)}} < \frac{\varepsilon}{5(l+1)}$. Therefore, if $n\geq N'$, then by \cite[Corollary 4.5]{Latremoliere13b},
    \begin{equation*}
      \Haus{\A_\infty}\left(\targetsettunnel{f(n)}{a}{l},\targetsettunnel{f(n)}{a_\varepsilon}{l}\right) \leq \norm{a-a_\varepsilon}{\A_\infty} + l \tunnelextent{\tau_{f(n)}} < \frac{2 \varepsilon}{5} \text.
    \end{equation*}
    
    Let now $p,q \geq \max\{N,N'\}$. We compute:
    \begin{align*}
      \Haus{\A_\infty}\left(\targetsettunnel{f(p)}{a}{l},\targetsettunnel{f(q)}{a}{l}\right)
      &\leq \Haus{\A_\infty}\left(\targetsettunnel{f(p)}{a}{l},\targetsettunnel{f(p)}{a_\varepsilon}{l}\right) \\
      &\quad + \Haus{\A_\infty}\left(\targetsettunnel{f(p)}{a_\varepsilon}{l},\targetsettunnel{f(q)}{a_\varepsilon}{l}\right) \\
      &\quad + \Haus{\A_\infty}\left(\targetsettunnel{f(q)}{a_\varepsilon}{l},\targetsettunnel{f(q)}{a}{l}\right) \\
      &< \frac{2\varepsilon}{5} + \frac{\varepsilon}{5} + \frac{2\varepsilon}{5} = \varepsilon \text.
    \end{align*}
    
    Thus, $\left(\targetsettunnel{f(n)}{a}{l}\right)_{n\in\N}$ is Cauchy for $\Haus{\A_\infty}$. Since $\sa{\A_\infty}$ is complete, so is the set of all closed subsets of $\sa{\A_\infty}$ with the Hausdorff distance $\Haus{\A_\infty}$. Therefore, $\left(\targetsettunnel{f(n)}{a}{l}\right)_{n\in\N}$ converges to some compact subset in $\sa{\A_\infty}$. In fact, since
    \begin{equation*}
      \lim_{n\rightarrow\infty}\diam{\targetsettunnel{f(n)}{a}{l}}{\A_\infty} = 0
    \end{equation*}
    by \cite[Corollary 4.5]{Latremoliere13b}, the sequence $\left(\targetsettunnel{f(n)}{a}{l}\right)_{n\in\N}$ converges to some singleton. As observed in Claim (\ref{singleton-claim}), this limit does not depend on $l$; we denote it by $\{\pi(a)\}$. Again using the same argument, we also note that $\norm{\pi(a)}{\A_\infty} = \norm{a}{\A_\infty}$.
    
    Since $\Lip_\infty$ is lower semicontinuous over $\A_\infty$, and since by construction, $\pi(a)$ is the limit in $\A_\infty$ of any sequence $(b_n)_{n\in\N}$ with $b_n\in\targetsettunnel{f(n)}{a}{\Lip_\infty(a)}$ for all $n\in\N$, we also conclude that
    \begin{align*}
      \Lip_\infty(\pi(a))
      &\leq \liminf_{n\rightarrow\infty}\Lip_\infty(b_n) && \text{ by lower semicontinuity of $\Lip_\infty$,}\\
      &\leq \liminf_{n\rightarrow\infty} M\cdot\Lip_n(b) && \text{ since $\Lip_\infty\leq M\cdot\Lip_n$ for all $n\in\N$,} \\
      &\leq M\cdot\Lip_\infty(a) && \text{ since $\Lip_n(b) \leq\Lip_\infty(a)$, as $b\in\targetsettunnel{f(n)}{a}{\Lip_\infty(a)}$ } \text.
    \end{align*}
    \medskip
    
    Let $a,a' \in \dom{\Lip_\infty}$. Let $t \in \R$. Since $\targetsettunnel{f(n)}{a}{l} + t\cdot\targetsettunnel{f(n)}{a'}{l}\subseteq\targetsettunnel{f(n)}{a+ta'}{(1+|t|)l}$ for all $n\in\N$ by \cite[Corollary 4.5]{Latremoliere13b}, we immediately conclude that $\{\pi(a)\} + t\cdot\{\pi(a')\} \subseteq \{\pi(a+ta')\}$, i.e. $\pi$ is linear. A similar argument shows that $\pi$ is a Jordan-Lie morphism over $\dom{\Lip_\infty}$, using \cite[Proposition 4.8]{Latremoliere13b}.
    
    As a linear map $\pi$ with $\norm{\pi(a)}{\A_\infty} = \norm{a}{\A_\infty}$ for all $a\in\dom{\Lip_\infty}$, we can uniquely extend $\pi$ to $\sa{\A_\infty}$ as a uniformly continuous map over $\sa{\A_\infty}$; this map is of course again a Jordan-Lie morphism from $\sa{\A_\infty}$ to $\sa{\A_\infty}$ and an isometry.
    
    A straightforward argument shows that we can uniquely extent $\pi$ to a continuous Jordan-Lie algebra endomorphism of $\A_\infty$, and thus $\pi$ thus extended is a unital $\ast$-endo\-mor\-phism with $\Lip_\infty\circ\pi \leq \Lip_\infty$ over $\dom{\Lip_\infty}$.
    
    We already know that $\pi$ is an isometry on $\sa{\A_\infty}$ and a $\ast$-morphism, so it is injective on $\A_\infty$: if $\pi(a) = 0$ then $\pi(\Re a) = 0$ so $\Re a = 0$, and $\pi(\Im a) = 0$ so $\Im a = 0$, and thus $a=0$. In particular, since $\pi$ is now an injective $\ast$-morphism, it is an isometry on $\A_\infty$ (rather than just $\sa{\A_\infty}$).
  \end{claimproof}
 
  \begin{claim}\label{uniform-claim}
    For all $\varepsilon > 0$, there exists $N\in\N$ such that for all $n\geq N$, and for all $a\in\dom{\Lip_\infty}$ with $\Lip_\infty(a)\leq 1$, we have
    \begin{equation*}
      \Haus{\A_\infty}\left( \{\pi(a)\}, \targetsettunnel{f(n)}{a}{1} \right) < \varepsilon \text.
    \end{equation*}
  \end{claim}
  
  \begin{claimproof}{uniform-claim}   
    Let $\varepsilon > 0$. Fix $\mu \in \StateSpace(\A_\infty)$. The set
    \begin{equation*}
      B \coloneqq \left\{ a \in \dom{\Lip_\infty} : \Lip_\infty(a) \leq 1, \mu(a) = 0 \right\}
    \end{equation*}
    is compact in $\sa{\A_\infty}$ by Theorem (\ref{tt-thm}). Therefore, there exists a finite subset $F\subseteq B$ such that $\Haus{\A_\infty}(F,B) < \frac{\varepsilon}{4}$. Since $F$ is finite, by Claim (\ref{JL-morphism-claim}), there exists $N\in\N$ such that, for all $a\in F$ and for all $n\geq N$, we have $\Haus{\A_\infty}\left(\{\pi(a)\}, \targetsettunnel{f(n)}{a}{\Lip_\infty(a)}\right) < \frac{\varepsilon}{4}$. Moreover, there exists $N' \in \N$ such that, if $n\geq N'$, then $\tunnelextent{\tau_n} < \frac{\varepsilon}{4}$.
    
    Now, let $n\geq \max\{N,N'\}$, $a \in B$ and $b \in \targetsettunnel{f(n)}{a}{1}$. There exists $a' \in F$ such that $\norm{a-a'}{\A_\infty} < \frac{\varepsilon}{4}$. Let $b' \in \targetsettunnel{f(n)}{a'}{1}$. By \cite[Corollary 4.5]{Latremoliere13b}, we compute the following expression:
    \begin{align*}
      \norm{\pi(a)-b}{\A_\infty}
      &\leq \norm{\pi(a)-\pi(a')}{\A_\infty} + \norm{\pi(a') - b'}{\A_\infty} + \norm{b'-b}{\A_\infty} \\
      &\leq \underbracket[1pt]{\norm{\pi(a-a')}{\A_\infty}}_{\pi \text{ is linear}} + \underbracket[1pt]{\frac{\varepsilon}{4}}_{\text{by choice of }N} + \underbracket[1pt]{\norm{a-a'}{\A_\infty} + \tunnelextent{\tau_n} }_{\text{by \cite{Latremoliere13b}}} \\
      &\leq 2\norm{a-a'}{\A_\infty} + \frac{\varepsilon}{4} + \frac{\varepsilon}{4} \\
      &\leq \frac{\varepsilon}{2} + \frac{\varepsilon}{4} + \frac{\varepsilon}{4} = \varepsilon \text.
    \end{align*}
    
    We thus have proven our uniform convergence claim over $B$. Let now $a\in \dom{\Lip_\infty}$. Then of course, $a-\mu(a)\unit \in B$, since $\Lip_\infty(a-\mu(a)\unit) \leq \Lip_\infty(a) + \Lip_\infty(\mu(a)\unit) = \Lip_\infty(a) \leq 1$ (in fact, $\Lip_\infty(a)=\Lip_\infty(a-\mu(a)\unit)$). If $b \in \targetsettunnel{f(n)}{a}{1}$ then $b-\mu(a)\unit \in \targetsettunnel{f(n)}{a-\mu(a)\unit}{1}$ by construction, and thus $\norm{\pi(a)-b}{\A_\infty} = \norm{\pi(a-\mu(a)\unit)-(b-\mu(a)\unit)}{\A_\infty} < \varepsilon$.

    Thus, as claimed, $\Haus{\A_\infty}\left(\{\pi(a)\}, \targetsettunnel{f(n)}{a}{1}\right) < \varepsilon$ for all $n\geq\max\{N,N'\}$ and for all $a\in B$.
  \end{claimproof}
  
  \begin{claim}\label{main-property-claim}
    For all $\varepsilon > 0$, there exists $N\in\N$ such that, if $n\geq N$, then
    \begin{itemize}
    \item $\forall a \in \dom{\Lip_\infty} \quad \exists b \in \dom{\Lip_n}: \quad \Lip_n(b)\leq \Lip_\infty(a) \text{ and }\norm{\pi(a)-b}{\A_\infty} < \varepsilon\Lip_\infty(a)$,
    \item $\forall b \in \dom{\Lip_n} \quad \exists a \in \dom{\Lip_\infty}: \quad \Lip_\infty(a) \leq \Lip_n(b) \text{ and }\norm{\pi(a)-b}{\A_\infty} < \varepsilon\Lip_n(b)$.  
    \end{itemize}
  \end{claim}
  
  \begin{claimproof}{main-property-claim}
    Let $\varepsilon > 0$. Let $N\in\N$ be chosen as in Claim (\ref{uniform-claim}), so that for all $a\in\dom{\Lip_\infty}$ with $\Lip_\infty(a)\leq 1$, and for all $n\geq N$, we have $\Haus{\A_\infty}(\{\pi(a)\},\targetsettunnel{f(n)}{a}{1})<\varepsilon$. Let now $n\geq N$.
    
    If $a\in \dom{\Lip_\infty}\setminus\R\unit_{\A_\infty}$, and if $b \in \targetsettunnel{f(n)}{a}{\Lip_\infty(a)}$, then $\Lip_\infty(a) > 0$, $\Lip_n(b)\leq\Lip_\infty(a)$, and $\frac{b}{\Lip_\infty(a)} \in \targetsettunnel{f(n)}{\frac{a}{\Lip_\infty(a)}}{1}$ and thus $\norm{\pi\left(\frac{a}{\Lip_\infty(a)}\right)-\frac{b}{\Lip_\infty(a)}}{\A_\infty} < \varepsilon$. So $\norm{\pi(a)-b}{\A_\infty} < \varepsilon \Lip_\infty(a)$, as needed.
    
    Now, let $b \in \dom{\Lip_n}\setminus\R\unit_\A$. Let $b' = \frac{b}{\Lip_n(b)}$, so $\Lip_n(b') = 1$. Let $a' \in \targetsettunnel{\tau_{f(n)}^{-1}}{b'}{1}$, so in particular $\Lip_\infty(a')\leq 1$. By symmetry, $b' \in \targetsettunnel{f(n)}{a'}{1}$. Therefore, $\norm{\pi(a')-b'}{\A_\infty} < \varepsilon$. Hence, letting $a = \Lip_n(b)a'$, we conclude that $\norm{\pi(a)-b}{\A_\infty} \leq \Lip_n(b)\varepsilon$ and $\Lip_\infty(a)\leq \Lip_n(b)$, as desired.
    
    Last, it is immediate that since $\pi(1)=1$, our claim holds whenever $\Lip_\infty(a) = 0$ or $\Lip_n(b) = 0$, i.e., for any $a,b \in \R1$.    
  \end{claimproof}
  
  \begin{claim}\label{automorphism-claim}
    The map $\pi$ constructed in Claim (\ref{JL-morphism-claim}) is a $\ast$-automorphism.
  \end{claim}
  
  \begin{claimproof}{automorphism-claim}  
    The map isometry of $\A_\infty$, hence it is a $\ast$-monomorphism of $\A_\infty$, via Claim (\ref{JL-morphism-claim}),
    
    Now, let $b \in \bigcup_{n\in\N} \dom{\Lip_n}$, so $b \in \dom{\Lip_m}$ for some $m\in\N$. Thus $b \in \dom{\Lip_\infty}$ by assumption. Let $l = \Lip_m(b)$. By assumption, $\Lip_\infty(b) \leq M \Lip_m(b) = M l $. Let $\varepsilon > 0$ and let $N \in \N$ given by Claim (\ref{main-property-claim}). Since $\Lip_n(b)\leq M\Lip_\infty(b) \leq M^2 l$, for all $n\geq \max\{N,m\}$, and there exists $a_n\in \A_\infty$ with $\norm{\pi(a_n)-b}{\A_\infty} < \varepsilon M^2 l$ (and $\Lip_\infty(a)\leq \Lip_n(b)$, which we do not need for this claim). As $\varepsilon > 0$ was arbitrary, the element $b$ lies in the closure of the range of $\pi$. Since $\A_\infty$ is complete and $\pi$ is an isometry, the range of $\pi$ is closed, and we now have shown that the range of $\pi$ is a closed set containing the total subspace $\bigcup_{n\in\N}\dom{\Lip_n}$ of $\A_\infty$; consequently, $\pi$ is a surjection as well. Thus as claimed, $\pi$ is a $\ast$-automorphism of $\A_\infty$.
    
    Moreover, by construction, for all $a\in\dom{\Lip_\infty}$, as noted in Claim (\ref{singleton-claim}), we have $\Lip_\infty(\pi(a))\leq M \Lip_\infty(a)$ --- in particular, $\pi(a) \in \dom{\Lip_\infty}$. So $\pi(\dom{\Lip_\infty})\subseteq\dom{\Lip_\infty}$ and thus $\pi$ is a Lipschitz morphism.    
  \end{claimproof} 

  This concludes the proof of our theorem.
\end{proof} 

\begin{remark}
  Limits, for the propinquity, are unique up to full quantum isometry. Therefore, the appearance of some map $\pi$ in Theorem (\ref{main-thm}) is to be expected. However, the map $\pi$ in Theorem (\ref{main-thm}) is quite a bit more general than a full quantum isometry --- in fact, it need not be Lipschitz for us to use Proposition (\ref{bridge-builder-prop}) --- even though Theorem (\ref{main-thm}) shows that it can always be chosen to be so. The map $\pi$ is really used here as a tool to construct a special kind of bridge. In general, the function $\pi$ is not expected to be unique: if $\Lip_n$ is just the restriction to $\A_n$ of $\Lip_\infty$ for all $n\in\N$, and if $\theta$ is a full quantum isometry of $(\A_\infty,\Lip_\infty)$, then $\pi\circ\theta$ can be used in place of $\pi$, of course. The situation is more delicate when $\Lip_n$ varies, but there will usually be many maps $\pi$ if there is one.
\end{remark}

Theorem (\ref{main-thm}) characterizes the convergence of inductive sequences in the sense of the propinquity, under the condition of uniform equivalence of the Lipschitz seminorms on the sequence. The condition of uniform equivalence of Lipschitz seminorms is in essence our compatibility condition between the Lipschitz seminorms and the inductive limit structure in Theorem (\ref{main-thm}): using the notation of Theorem (\ref{main-thm}), as seen in \cite{Latremoliere16b}, under the hypothesis that $\dom{\Lip_n}=\A_n\cap\dom{\Lip_\infty}$, the Lipschitz seminorms $\Lip_n$ and $\Lip_\infty$ are equivalent for each $n\in\N$, and we require, in the assumptions of Theorem (\ref{main-thm}), that we want this equivalence be uniform. This leads us to several natural questions: does convergence of $(\A_n,\Lip_n)_{n\in\N}$ imply some uniform equivalence of the Lipschitz seminorms $\Lip_n$ ($n\in\Nbar$) (i.e. is our assumption redundant)? Does the existence of a bridge builder imply uniform equivalence of the Lipschitz seminorms? Does convergence of an inductive limit for the propinquity imply the existence of a bridge builder without the assumption of uniform equivalence of the Lipschitz seminorms? Moreover, does the convergence of $(\A_n,\Lip_n)_{n\in\N}$ to $(\A_\infty,\Lip_\infty)$ for the propinquity imply the convergence of $(\A_n,\Lip_k)_{k\geq n}$ to $(\A_n,\Lip_\infty)$ for a fixed $n\in\N$, for the propinquity?

We now will show with two examples that all of the above questions have negative answers, so there is no obvious generalization of Theorem (\ref{main-thm}). First, we see that it is possible to have convergence for the propinquity of an inductive sequence of {\qcms s}, using the identity as a bridge builder, and yet, not have uniform equivalence of the Lipschitz seminorms.

\begin{example}
  \begin{figure}    \begin{tikzpicture}
      \draw (-0.5,0)--(6.5,0);
      \draw (0,-0.1)--(0,0.1) node[below=6pt]{\tiny $0$};
      \draw[thick,blue] (0,0)--(4,0) node[above,midway]{\tiny $\mathsf{dist}$};
      \draw (6,-0.1)--(6,0.1) node[below=6pt]{\tiny $1$};
      \draw[thick,red] (4,0)--(6,0) node[above,midway]{\tiny $n\cdot\mathsf{dist}$};
      \draw (4,-0.1)--(4,0.1) node[below=6pt]{\tiny $1-n^{-2}$};

      \draw (-0.5,-1)--(6.5,-1);
      \draw (0,-1.1)--(0,-0.9) node[below=6pt]{\tiny $0$};
      \draw[thick,blue] (0,-1)--(4.5,-1) node[above,midway]{\tiny $\mathsf{dist}$};
      \draw (6,-1.1)--(6,-0.9) node[below=6pt]{\tiny $1$};
      \draw[thick,red] (4.5,-1)--(6,-1) node[above,midway]{\tiny $(n+1)\cdot\mathsf{dist}$};
      \draw (4.5,-1.1)--(4.5,-0.9) node[below=6pt]{\tiny $1-(n+1)^{-2}$};

      \draw[->] (2,-2) node[left=8pt]{\tiny $\cdots$} -- (4,-2) node[midway,above]{\tiny $n\rightarrow\infty$};
      
      \draw (-0.5,-3)--(6.5,-3);
      \draw (0,-3.1)--(0,-2.9) node[below=6pt]{\tiny $0$};
      \draw[thick,blue] (0,-3)--(6,-3) node[above,midway]{\tiny $\mathsf{dist}$};
      \draw (6,-3.1)--(6,-2.9) node[below=6pt]{\tiny $1$};
      \draw[snake=brace, blue] (6,-3.6)--(0,-3.6) node[below=3pt,midway]{\tiny $X$ with $\mathsf{dist}:x,y\in X\mapsto|x-y|$};
    \end{tikzpicture}
    \caption{Approximating $[0,1]$ with itself by modifying the metric on a small interval at the end (red)}
  \end{figure}
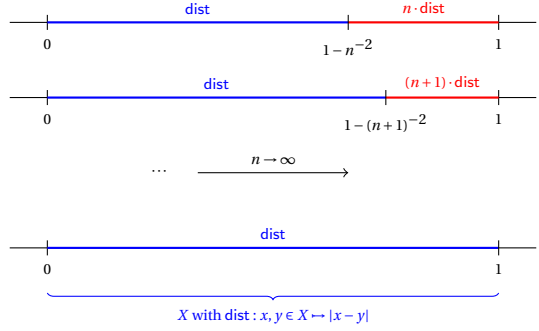
  
  Let $X = [0,1]$ with its usual metric. If $Y\subseteq X$ with at least two points, then we set $\Lip_Y(f) = \sup\left\{\frac{|f(x)-f(y)|}{|x-y|} : x\neq y, x,y \in Y \right\}$ for all $f \in C(X)$, allowing for $\infty$. For each $n\in\N$, and for all $f \in C(X)$, we set:
  \begin{equation*}
    \Lip_n(f) = \Lip_{\left[0,1-\frac{1}{n^2}\right]}(f) + \frac{1}{n}\Lip_{\left[1-\frac{1}{n^2},1\right]}(f),
  \end{equation*}
  allowing again for $\infty$.

  Let
  \begin{equation*}
    f_n : x \in [0,1] \longmapsto
    \begin{cases}
      0\quad  &\text{ if $x \leq 1-\frac{1}{n^2}$}, \\
      x-(1-\frac{1}{n^2})\quad &\text{ otherwise.}
    \end{cases}
  \end{equation*}

  By construction, $\Lip_{[0,1]}(f_n) = 1$ for all $n\in\N$. On the other hand, $\Lip_n(f) = 0 + \frac{1}{n}\cdot 1 = \frac{1}{n}$. So there does not exists $M > 0$ such that $\Lip_{[0,1]}\leq M \Lip_n$ on the common domain of these Lipschitz seminorms (the algebra of Lipschitz functions for the usual metric).

  We now prove that $(C(X),\Lip_n)_{n\in\N}$ converges for the propinquity to $(C(X),\Lip_{[0,1]})$ --- this could be done here just as easily by proving the convergence for the Gromov-Hausdorff distance of $X$ with a sequence of distances which agree with the usual distance on $[0,1-\frac{1}{n^2}]$ and is a dilation by a factor $n$ of the usual distance on $[1-\frac{1}{n^2},1]$, but we will keep with our functional analytic perspective here.

  We thus define, for all $n\in\N$, and for all $f, g$ Lipschitz functions over $[0,1]$ with its usual metric:
  \begin{equation*}
    \TLip_n(f,g) \coloneqq \max\left\{ \Lip_{[0,1]}(f), \Lip_n(g), (n+1) \norm{f-g}{C(X)} \right\} \text.
  \end{equation*}

  Let $f \in C(X)$ with $\Lip_{[0,1]}(f) = 1$. Then $\Lip_n(f) \leq 1+\frac{1}{n}$. From this, we see that
  \begin{equation*}
    \TLip_n \left(f,\frac{1}{1+\frac{1}{n}}f - \frac{1}{n+1}f(0)\right)  \leq \max\left\{1,\frac{n+1}{n+1}\norm{f-f(0)\unit}{C(X)} \right\} \leq 1 \text.
  \end{equation*}
  Let now $g \in C(X)$ with $\Lip_n(g) = 1$. Thus $\Lip_{\left[0,1-\frac{1}{n^2}\right]}(g) \leq 1$ and $\Lip_{\left[1-\frac{1}{n^2},1\right]}(g) \leq n$. In particular, for all $x \in [1-\frac{1}{n^2},1]$, we have $\left|g(x)-g\left(1-\frac{1}{n^2}\right)\right| < n |x-1+\frac{1}{n^2}| \leq \frac{1}{n}$.

  Let $h \in C(X)$ defined by $h(x)=g(x)$ if $x\in \left[0,1-\frac{1}{n^2}\right]$, and $h(x) = g\left(1-\frac{1}{n^2}\right)$ otherwise. By construction, $\Lip_{[0,1]}(h) \leq 1$ and $\norm{g-h}{C(X)} < \frac{1}{n}$. Thus $\TLip_n(h,g) = 1 = \Lip_n(g)$.

  Therefore, $(C(X)\oplus C(X),\TLip_n,p_1,p_2)$, with $p_1:(f,g)\in C(X)\oplus C(X) \mapsto f$ and $p_2:(f,g)\in C(X)\oplus C(X)\mapsto g$, is easily seen to be a tunnel whose extent is at most $\frac{1}{n}$ (the method is analogous to Proposition (\ref{bridge-builder-prop})).

  Hence $(C(X),\Lip_n)_{n\in\N}$ converges to $(C(X),\Lip_{[0,1]})_{n\in\N}$ for the propinquity. Moreover, the identity map satisfies Condition (2) of Theorem (\ref{main-thm}). Nonetheless, there is no $M > 0$ such that $\forall n \in \N \quad \Lip_{[0,1]}\leq M \Lip_n$ on the common domain of these seminorms. So convergence in the propinquity does not imply uniform equivalence of the Lipschitz seminorms, even when working with a fixed, Abelian C*-algebra.
\end{example}

Now, we can also ask whether convergence for the propinquity of an inductive sequence, implies the existence of a bridge builder, and as we shall see in the next example, this is not the case: once again, convergence occurs without uniform equivalence of Lipschitz seminorms (and we prove that we have neither uniform dominance or uniform domination using both examples). Moreover, we see that $(\A_n,\Lip_m)_{m\geq n}$ does not converge to $(\A_n,\Lip_\infty)$ in this case, for any $n\in\N$.

\begin{example}\label{Nbar-ex}
  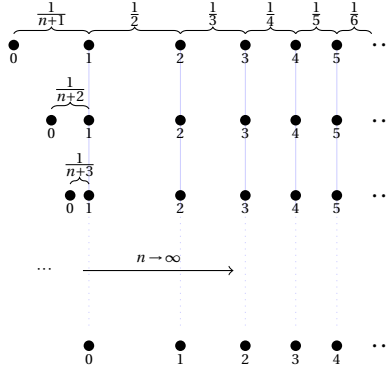
\begin{figure}
    \begin{tikzpicture}

      \foreach \x [count=\y from 2] in {1,2,...,5}
      {
        \draw[very thin, blue!20!white] ({3*ln(\x+1)},0)--({3*ln(\x+1)},-2.2);
        \draw[blue!20!white,dotted] ({3*ln(\x+1)},-2.2)--({3*ln(\x+1)},-3.8);
        \fill({3*ln(\x+1)},0) circle(2pt) node[below]{\tiny $\x$};
        \draw[snake=brace] ({3*ln(\x+1)},0.15)--({3*ln(\x+2)},0.15) node[above,midway]{\tiny $\frac{1}{{\y}}$};
      }
      \fill(1.08,0) circle(2pt) node[below]{\tiny $0$};
      \node at (6,0) {$\ldots$};
      \draw[snake=brace] (1.08,0.15)--(2.08,0.15) node[above,midway]{\tiny $\frac{1}{n+1}$};
      
      \foreach \x in {1,2,...,5}
      \fill({3*ln(\x+1)},-1) circle(2pt) node[below]{\tiny $\x$};
      \fill(1.58,-1) circle(2pt) node[below]{\tiny $0$};
      \node at (6,-1) {$\ldots$};
      \draw[snake=brace] (1.58,-0.85)--(2.08,-0.85) node[above,midway]{\tiny $\frac{1}{n+2}$};
      
      \foreach \x in {1,2,...,5}
      \fill({3*ln(\x+1)},-2) circle(2pt) node[below]{\tiny $\x$};
      \fill(1.83,-2) circle(2pt) node[below]{\tiny $0$};
      \node at (6,-2) {$\ldots$};
      \draw[snake=brace] (1.83,-1.85)--(2.08,-1.85) node[above,midway]{\tiny $\frac{1}{n+3}$};

      \draw[->] (2,-3) node[left=8pt]{\tiny $\cdots$} -- (4,-3) node[midway,above]{\tiny $n\rightarrow\infty$};

      \foreach \x in {0,1,...,4}
      \fill({3*ln(\x+2)},-4) circle(2pt) node[below]{\tiny $\x$};
      \node at (6,-4) {$\ldots$};
    \end{tikzpicture}
    \caption{Approximating $\Nbar$ by itself, by merging the first two points at $\infty$}
  \end{figure}
  
  Let $\A_\infty$ be the C*-algebra of convergent sequences with values in $\C$. For each $n \in \N$, let $\A_n = \left\{ (x_k)_{k\in\N} : (x_k)_{k\geq n} \text{ is constant } \right\}$, so $\A_n$ is a C*-subalgebra of $\A_\infty$ sharing the unit $(1)_{n\in\N}$ of $\A_\infty$.  

  For all $n \in \N$, and for all $(x_k)_{k\in\N} \in \A_n$, we set
  \begin{equation*}
    \Lip_n((x_k)_{k\in\N}) \coloneqq \sup\left\{ \frac{|x_p-x_q|}{|\varphi_n(p)-\varphi_n(q)|} : p,q \in \N,p\neq q \right\}
  \end{equation*}
  where:
  \begin{equation*}
    \varphi_n :m \in \N \mapsto
    \begin{cases}
      \frac{1}{m} \text{ if $m>0$,}\\
      1+\frac{1}{n} \text{ if $m = 0$.}
    \end{cases}
  \end{equation*}
  Of course, $\Lip_n$ is indeed a seminorm on the finite dimensional C*-subalgebra $\A_n$ of $\A_\infty$.
  
  We also set $\Lip_\infty((x_k)_{k\in\N}) = \sup\left\{ \frac{|x_p-x_q|}{\left|\frac{1}{p+1}-\frac{1}{q+1}\right|} : p,q \in \N, p\neq q\right\}$ for all $(x_k)_{k\in\N} \in \A_\infty$, allowing for the value $\infty$. Of course, $\bigcup_{n\in\N}\A_n\subseteq\dom{\Lip_\infty}$.
  
  Now, let
  \begin{equation*}
    x : n \in \N \mapsto
    \begin{cases}
      1 \text{ if $n = 0$,}\\
      0 \text{ otherwise.}
    \end{cases}
  \end{equation*}
  By construction, $\Lip_\infty(x) = 2$, yet $\Lip_n(x) = n$. So there is no $M > 0$ such that, for all $n\in\N$, the inequality $M\Lip_n \leq \Lip_\infty$ on $\dom{\Lip_n}$ holds.

  On the other hand, $\lim_{n\rightarrow\infty} \dpropinquity{}((\A_n,\Lip_n),(\A_\infty,\Lip_\infty)) = 0$. Indeed, let $\pi:(x_k)_{k\in\N}\mapsto (x_0,x_0,x_1,x_2,x_3,\ldots) \in \A_\infty$, $\B = \pi(\A_\infty)$, and let $\theta:(x_k)_{k\in\N}\in \B \mapsto (x_{k+1})_{k\in \N}\in\A_\infty$ --- of course, $\theta$ is a $\ast$-isomorphism from $\B$ onto $\A_\infty$ such that $\pi=\theta^{-1}$. We define $\Lip_\B(\pi(x)) = \Lip_\infty(x)$ for all $x\in\dom{\Lip_\infty}$. This way, $\pi$ is easily checked to be a full quantum isometry from $(\A_\infty,\Lip_\infty)$ to $(\B,\Lip_\B)$.

  Let $\varepsilon > 0$ and let $N\in\N$ be such that if $n\geq N$, then $\frac{1}{n+1}<\frac{\varepsilon}{2}$. If $x = (x_k)_{k\in \N}$ with $\Lip_\infty(x) \leq 1$, and if $l = \lim_{s\rightarrow\infty} x_s$, then by construction,
  \begin{equation*}
    \frac{|x_k-l|}{\frac{1}{k+1}} = \lim_{s\rightarrow\infty}\frac{|x_k-x_s|}{\frac{1}{k+1}-\frac{1}{s+1}} \leq 1
  \end{equation*}
  so $|x_k - l|\leq\frac{1}{k+1} < \frac{\varepsilon}{2}$ for all $k\geq N$. Therefore, if $k\geq N$ then $|x_k - x_N| < \varepsilon$.

	Now, let $n\geq N$. Let $\D_n = \A_n\oplus\B$, and for all $(a,b) \in \dom{\A_n}\oplus\dom{\B}$, we set:
  \begin{equation*}
    \TLip_n(a,b) \coloneqq \max\left\{ \Lip_n(a),\Lip_\B(b), \frac{1}{\varepsilon}\norm{\pi(a)-b}{\B} \right\} \text.
  \end{equation*}
  
  We also set $p_n : (a,b)\in\D_n\mapsto a\in\A_n$ and $q_n : (a,b)\in\D_n\mapsto \theta(b) \in \A_\infty$. We are now going to prove that $\tau_n \coloneqq(\D_n,\TDN_n,p_n,q_n)$ is indeed a tunnel from $(\A_n,\Lip_n)$ to $(\A_\infty,\Lip_\infty)$. 
  
  Let $a\coloneqq(x_k)_{k\in\N}\in\dom{\Lip_\infty}$ with $\Lip_\infty(a)=1$, and let
  \begin{equation*}
    a' \coloneqq (x_0,x_0,x_1,x_2,\ldots,x_{N-1},x_N,x_N,x_N\ldots) \in \A_n\text.
  \end{equation*}
  By construction, $\Lip_n(a') \leq 1$ and $\norm{\pi(a)-a'}{\A_\infty}<\varepsilon$ by our choice of $N$. Also by construction, $\Lip_\B(\pi(a))=\Lip_\infty(a)=1$. Thus $\TLip_n(a',\pi(a)) \leq \Lip_\infty(a)=1$. So, we have shown that, for any $a\in\dom{\Lip_\infty}$ with $\Lip_\infty(a)=1$, there exists an element $d\coloneqq(a',\pi(a)) \in \D_n$ such that $\TDN_n(d)=1=\Lip_\infty(a)$ and $q_n(d) = \theta(\pi(a)) = a$. Therefore, the map $q_n$ is indeed a quantum isometry from $(\D_n,\TLip_n)$ to $(\A_\infty,\Lip_\infty)$.

  Let now $a = (x_k)_{k\in\N} \in \dom{\Lip_n}$ with $\Lip_n(a) = 1$. By definition, $|x_1-x_0|\leq \frac{1}{n} < \varepsilon$. Let
  \begin{equation*}
    b = (x_1,x_1,x_2,x_3,x_4,\ldots)\text.
  \end{equation*}
  By construction, $b \in \dom{\Lip_\B}$ with $\Lip_\B(b)\leq \Lip_n(b)$, and $\norm{a-b}{\A_\infty} = |x_1-x_0| < \varepsilon$. Thus again $\TLip_n(a,b) = \Lip_n(b)$. So $p_n : (a,b)\in\D_n\mapsto a \in \A_n$ is a quantum isometry. Therefore, $(\D_n,\TLip_n,p_n,q_n)$ is indeed a tunnel from $(\A_n,\Lip_n)$ to $(\A_\infty,\Lip_\infty)$. We now compute an upper bound on its extent.

  Let $\varphi\in\StateSpace(\D_n)$ be a state of $\D_n$. If we set $\mu:a\in\A_n\mapsto \varphi(a,\pi(a))$, then $\mu\in\StateSpace(\A_n)$ is again a state of $\A_n$. If $(a,b)\in\dom{\TLip_n}$ with $\TLip_n(a,b)\leq 1$, then
  \begin{align*}
    |\varphi(a,b)-\mu\circ p_n(a,b)|
    &=|\varphi(a,b) - \varphi(a,\pi(a))| \\
    &=|\varphi(0,b-\pi(a))| \\
    &\leq \norm{b-\pi(a)}{\A_\infty} < \varepsilon \text,
  \end{align*}
  so indeed $\Haus{\Kantorovich{\TLip_n}}(\StateSpace(\D_n),p_n^\ast\StateSpace(\A_n)) < \varepsilon$.
  
  On the other hand, let $\nu : a \in \A_\infty \mapsto \varphi'(a,\pi(a))$ where $\varphi'$ is an extension of $\varphi$ to a state of $\A_\infty\oplus\B$ by the Hahn-Banach theorem. Once again, it is immediate that $\Kantorovich{\TLip_n}(\varphi,\nu\circ q_n) < \varepsilon$. So $\Haus{\Kantorovich{\TLip_n}}(\StateSpace(\D),q_n^\ast\StateSpace(\A_\infty)) < \varepsilon$.

  Thus, for all $n\geq N$, the extend of $\tunnelextent{\tau_n}$ is at most $\varepsilon$. We conclude:
  \begin{equation*}
    \lim_{n\rightarrow\infty}\dpropinquity{}((\A_n,\Lip_n),(\A_\infty,\Lip_\infty)) = 0 \text.
  \end{equation*}

  However, for any fixed $p\in\N$, it is easy to check, by a similar method, that 
  \begin{equation*}
  \lim_{n\rightarrow\infty} \dpropinquity{}((\A_p,\Lip_n),(\A_{p-1},\Lip_\infty)) = 0\text,
  \end{equation*}
   and since $\dim\A_{p-1}<\dim\A_p$, the sequence $(\A_p,\Lip_n)_{n\geq p}$ does not converge to $(\A_p,\Lip_\infty)$.

  The map $\pi$ we have used here is not surjective. In fact, there is no bridge builder in this case. Indeed, assume that we have a unital $\ast$-morphism $\pi : \A_\infty\rightarrow\A_\infty$ such that for all $\varepsilon > 0$, there exists $N_\pi(\varepsilon)\in\N$ with the property that if $n\geq N_\pi(\varepsilon)$, and if $a\in \dom{\Lip_\infty}$, then there exists $b \in \dom{\Lip_n}$ with $\Lip_n(b)\leq \Lip_\infty(a)$ and $\norm{\pi(a)-b}{\A_\infty} < \frac{\varepsilon}{2} \Lip_\infty(a)$. 
  
  Fix $a \in\dom{\Lip_\infty}$ with $\Lip_\infty(a) = 1$. Let $\varepsilon > 0$ and let $n \geq N_\pi(\varepsilon)$ such that $\frac{1}{n}<\frac{\varepsilon}{2}$. Define $(y_k)_{k\in\N}\coloneqq\pi(a)$. Then there exists $b\coloneqq(b_k)_{k\in\N} \in \dom{\Lip_n}$ such that $\Lip_n(b)\leq 1$ and $\norm{\pi(a)-b}{\A_\infty}<\frac{\varepsilon}{2}$. By definition of $\Lip_n$, we thus conclude that $|b_1 - b_0| \leq \frac{1}{n} < \frac{\varepsilon}{2}$.  Thus, $|y_1-y_0| < \varepsilon$. As $\varepsilon > 0$ is arbitrary, we conclude that $y_1 = y_0$. Thus $\pi$ can never be surjective --- in fact, it is valued in $\B$. So no bridge builder exists for this example.
\end{example}

\medskip

As seen in Example (\ref{Nbar-ex}), convergence of $(\A_n,\Lip_n)_{n\in\N}$ to $(\A_\infty,\Lip_\infty)$ for the propinquity does not imply the convergence of $(\A_n,\Lip_p)_{p\in\N}$ to $(\A_n,\Lip_\infty)$. We have the following immediate consequence of our work:
\begin{corollary}
  Let $\A_\infty$ be a unital separable C*-algebra, such that $\A_\infty = \closure{\bigcup_{n\in\N} \A_n}$, where $(\A_n)_{n\in\N}$ is an increasing (for $\subseteq$) sequence of C*-subalgebras of $\A_\infty$, with the unit of $\A_\infty$ in $\A_0$. For each $n\in\Nbar$, let $\Lip_n$ be a Lipschitz seminorm on $\A_n$. If there exists a bridge builder $\pi : \A_\infty\rightarrow\A_\infty$ for $((\A_n,\Lip_n)_{n\in\N},(\A_\infty,\Lip_\infty))$ such that $\pi(\A_n)\subseteq\A_n$ for each $n\in\N$, then for all $n \in \Nbar$,
  \begin{equation*}
    \lim_{\substack{p\rightarrow\infty \\ p \geq n}} \dpropinquity{}((\A_n,\Lip_p),(\A_n,\Lip_\infty)) = 0\text,
  \end{equation*}
  and $\lim_{n\rightarrow\infty}\spectralpropinquity{}((\A_n,\Lip_n),(\A_\infty,\Lip_\infty)) = 0$.
\end{corollary}

\begin{proof}
  This follows by observing that the restriction of $\pi$ to $\A_n$ is a bridge builder for $((\A_n,\Lip_p)_{p\geq n},(\A_n,\Lip_\infty))$. Our result then follows from Proposition (\ref{bridge-builder-prop}).
\end{proof}

\section{Convergence of Inductive Sequences of Metric Spectral Triples for the Spectral Propinquity}

We now study the convergence of certain families of metric spectral triples for the spectral propinquity \cite{Latremoliere18g}, whose construction we will recall below. We thus begin this section with the definition of a spectral triple, due to Connes, and the foundational concept for noncommutative Riemannian geometry.

\begin{definition}[{\cite{Connes,Connes89}}]\label{spectral-triples-def}
  A \emph{spectral triple} $(\A,\Hilbert,\Dirac)$ is given by a unital C*-algebra $\A$ of bounded linear operators on a Hilbert space $\Hilbert$, and a self-adjoint operator $\Dirac$ defined on some dense subspace $\dom{\Dirac}$ of $\Hilbert$, such that:
  \begin{enumerate}
  \item $\left\{ a \in \A : a\cdot \dom{\Dirac} \subseteq \dom{\Dirac}, [\Dirac,a] \text{ is bounded } \right\}$ is a dense $\ast$-algebra in $\A$,
  \item $\Dirac$ has compact resolvent.
  \end{enumerate}

  The operator $\Dirac$ is referred to as the Dirac operator of the spectral triple.
\end{definition}

\subsection{Preliminaries: The Spectral Propinquity}\label{SpProp-sec}

The spectral propinquity is a distance, up to unitary equivalence, on the class of metric spectral triples.
\begin{notation}
  If $T : D\subseteq E \rightarrow F$ is a linear operator defined from a dense subspace $D$ of a normed vector space $E$ to a normed vector space $F$, then we write:
  \begin{equation*}
    \opnorm{T}{F}{E} \coloneqq \sup\left\{ \norm{T \xi}{F} : \xi \in D, \norm{\xi}{E} \leq 1 \right\}
  \end{equation*}
  allowing for the value $\infty$. If $F=E$, then $\opnorm{T}{F}{E}$ is simply denoted by $\opnorm{T}{}{E}$.
\end{notation}

\begin{definition}\label{metric-spectral-triple-def}
  A {spectral triple} $(\A,\Hilbert,\Dirac)$ is \emph{metric} if the Connes extended pseudo-distance, defined on the state space $\StateSpace(\A)$ of $\A$ by:
  \begin{equation*}
    \Kantorovich{\Dirac} :\varphi,\psi \in \StateSpace(\A) \mapsto \sup\left\{|\varphi(a)-\psi(a)|:a\,\dom{\Dirac}\subseteq\dom{\Dirac} \text{ and }\opnorm{[\Dirac,a]}{}{\Hilbert} \leq 1 \right\}
  \end{equation*}
  is in fact a metric on $\StateSpace(\A)$, which  induces the weak-$^\ast$ topology on $\StateSpace(\A)$.
\end{definition}

As soon as a spectral triple is metric, it induces a structure of {\qcms} on its underlying C*-algebra in a natural manner.
\begin{proposition}[{\cite[Proposition 1.10]{Latremoliere18g}}]\label{spectral-cqms-prop}
  Let $(\A,\Hilbert,\Dirac)$ be a spectral triple. We set:
  \begin{equation*}
    \dom{\Lip_\Dirac} \coloneqq \left\{ a \in \sa{\A} : a\,\dom{\Dirac}\subseteq\dom{\Dirac} \text{ and }[\Dirac,a]\text{ is bounded } \right\}
  \end{equation*}
  and  for all  $a \in \dom{\Lip_\Dirac} $:
  \begin{equation*}
   \quad \Lip_\Dirac(a) \coloneqq \opnorm{[\Dirac,a]}{}{\Hilbert} \text.
  \end{equation*}

  The spectral triple $(\A,\Hilbert,\Dirac)$ is metric if, and only if, $(\A,\Lip_\Dirac)$ is a {\qcms}.
\end{proposition}

The construction of the spectral propinquity begins with the following observation. Recall from  \cite{Latremoliere18g} that if $(\A,\Hilbert,\Dirac)$ is a metric spectral triple, and if we set
\begin{itemize}
 \item for all $\xi \in \dom{\Dirac} $:
\begin{equation}\label{Dirac-D-norm-eq}
 \quad \CDN_\Dirac(\xi)\coloneqq\norm{\xi}{\Hilbert} + \norm{\Dirac\xi}{\Hilbert} \text,
\end{equation}
\item 	$\dom{\Lip_\Dirac} \coloneqq \left\{ a\in\sa{\A} : a\,\dom{\Dirac}\subseteq\dom{\Dirac}\text{, }[\Dirac,a] \text{ is bounded } \right\}$ 
\item  for all $a \in \dom{\Lip_\Dirac}$: 
\begin{equation*}
  \quad \Lip_\Dirac(a) \coloneqq \opnorm{[\Dirac,a]}{}{\Hilbert} \text,
\end{equation*}
\end{itemize}
then 
\begin{equation*}
  \mcc{\A}{\Hilbert}{\Dirac} \coloneqq \left(\Hilbert,\CDN_\Dirac,\A,\Lip_\Dirac,\C,0\right)
\end{equation*}
is an example of a \emph{metrical C*-correspondence}, in the following sense:

\begin{definition}
  An $\A$-$\B$-\emph{$C^\ast$-correspondence} $(\module{M},\A,\B)$, for two C*-algebras $\A$ and $\B$, is a right Hilbert module $\module{M}$ over $\B$ (whose $\B$-valued inner product is denoted by $\inner{\cdot}{\cdot}{\module{M}}$), together with a unital $\ast$-morphism from $\A$ to the C*-algebra of adjoinable $\B$-linear operators over $\module{M}$. 
\end{definition}

\begin{definition}[{\cite[Definition 2.2]{Latremoliere18g}}]\label{mcc-def}
  An \emph{$(\Omega,\Omega',\Omega_{\mathrm{mod}}, \Omega_{\mathrm{inner}})$-metrical C*-correspondence} $(\module{M},\CDN,\A,\Lip,\B,\SLip)$, where $\Omega,\Omega_{\mathrm{inner}} \geq 1$, $\Omega_{\mathrm{mod}}\geq 2$, and $\Omega'\geq 0$, is given by two {$(\Omega,\Omega')$-\qcms s} $(\A,\Lip)$ and $(\B,\SLip)$, an $\A$-$\B$ C*-correspondence $(\module{M},\A,\B)$, and a norm $\CDN$ defined on a dense $\C$-subspace $\dom{\TDN}$ of $\module{M}$, such that
  \begin{enumerate}
  \item $\forall \omega\in\dom{\CDN} \quad \CDN(\omega)\geq\norm{\omega}{\module{M}}\coloneqq \sqrt{\norm{\inner{\omega}{\omega}{\module{M}}}{\B}}$,
  \item $\{ \omega\in\dom{\CDN} : \CDN(\omega)\leq 1\}$ is compact in $(\module{M},\norm{\cdot}{\module{M}})$,
  \item for all $a\in \dom{\Lip}$ and $\omega\in \dom{\TDN}$,
    \begin{equation*}
      \CDN(a\omega)\leq \Omega_{\mathrm{mod}}(\norm{a}{\A}+\Lip(a))\CDN(\omega) \text,
    \end{equation*}
  \item for all $\omega,\eta\in\dom{\CDN}$,
    \begin{equation*}
      \max\{\SLip(\Re\inner{\omega}{\eta}{\module{M}}),\SLip(\Im\inner{\omega}{\eta}{\module{M}})\} \leq \Omega_{\mathrm{inner}} \CDN(\omega) \CDN(\eta) \text.
    \end{equation*}
  \end{enumerate}
  In particular, the norm $\CDN$ is called a \emph{D-norm}.
\end{definition}

\begin{convention}\label{FKGH-convention}
  In this work, we \emph{fix} $\Omega_{\mathrm{mod}}\geq 2$ and $\Omega_{\mathrm{inner}}\geq 1$ all throughout the paper. \emph{All} {\qcms s} will be assumed to be in the class of $(\Omega,\Omega')$-{\qcms s} and all metrical C*-correspondences will be assume to be in the class of $(\Omega,\Omega',\Omega_{\mathrm{mod}}, \Omega_{\mathrm{inner}})$-metrical C*-correspondences, unless otherwise specified.
\end{convention}
Note that the compactness condition in Definition (\ref{mcc-def}) borrows and extends on Theorem (\ref{tt-thm}).

The importance of Definition (\ref{mcc-def}) is that one can extend the propinquity to {\gMVB s} as follows. First, we employ a natural notion of morphism between {\gMVB s}.

\begin{definition}[{\cite[Definition 2.13]{Latremoliere18g}}]
  For each $j \in \{1,2\}$, let
  \begin{equation*}
    \mathds{M}_j = \left( \module{M}_j,\CDN_j,\A_j,\Lip_j,\B_j,\SLip_j \right)
  \end{equation*}
  be a metrical C*-correspondence. 

  A \emph{metrical quantum isometry} $(\Pi,\pi,\theta)$ from $\mathds{M}_1$ to $\mathds{M}_2$ is a given by:
  \begin{enumerate}
  \item a continuous, surjective $\C$-linear map $\Pi : \module{M}_1 \rightarrow\module{M}_2$,
  \item a quantum isometry $\pi : (\A_1,\Lip_1) \rightarrow (\A_2,\Lip_2)$,
  \item a quantum isometry $\theta: (\B_1,\SLip_1) \rightarrow (\B_2,\SLip_2)$,
  \end{enumerate}
  such that
  \begin{enumerate}
  \item $\forall a \in \A \quad \forall \omega \in \module{M}_1 \quad \Pi(a\omega) = \pi(a)\Pi(\omega)$,
  \item $\forall b \in \B \quad \forall \omega \in \module{M}_2 \quad \Pi(\omega\cdot b) = \Pi(\omega)\theta(b)$,
  \item $\forall \omega,\eta\in \module{M}_1 \quad \theta(\inner{\omega}{\eta}{\module{M}_1}) = \inner{\Pi(\omega)}{\Pi(\eta)}{\module{M}_2}$,
  \item $\Pi(\dom{\CDN_1})\subseteq\dom{\CDN_2}$ and, for all $\omega\in\dom{\CDN_2}$, the equality $\CDN_2(\omega) = \inf\left\{ \CDN_1(\eta) : \eta\in\dom{\CDN_1},\Pi(\eta)=\omega \right\}$.
  \end{enumerate}
\end{definition}

The definition of a distance between {\gMVB s}, called the \emph{metrical propinquity}, relies on a notion of isometric embedding called a tunnel, and is defined as follows.

\begin{definition}[{\cite[Definition 2.19]{Latremoliere18g}}]\label{mcc-tunnel-def}
  Let $\mathds{M}_1$ and $\mathds{M}_2$ be two metrical C*-correspondences. A \emph{(metrical) tunnel} $\tau = (\mathds{J},\Pi_1,\Pi_2)$ from $\mathds{M}_1$ to $\mathds{M}_2$ is a triple given by a metrical C*-correspondence $\mathds{J}$, and for each $j\in\{1,2\}$, a metrical quantum isometry $\Pi_j : \mathds{J}\mapsto \mathds{M}_j$.
\end{definition}

\begin{remark}
  It is important to note that our tunnels involve $(\Omega,\Omega',\Omega_{\mathrm{mod}}, \Omega_{\mathrm{inner}})$-C*-metrical correspondences only (as per Convention (\ref{FKGH-convention})). We will dispense calling our tunnels $(\Omega,\Omega',\Omega_{\mathrm{mod}}, \Omega_{\mathrm{inner}})$-tunnels, to keep our notation simple, but it should be stressed that \emph{fixing} $(\Omega,\Omega',\Omega_{\mathrm{mod}}, \Omega_{\mathrm{inner}})$ and staying within the class of $(\Omega,\Omega',\Omega_{\mathrm{mod}}, \Omega_{\mathrm{inner}})$-C*-metrical correspondences is crucial to obtain a metric from tunnels.
\end{remark}

We now proceed by defining the extent of a metrical tunnel; remarkably this only involves our previous notion of extent of a tunnel between {\qcms s}.

\begin{definition}[{\cite[Definition 2.21]{Latremoliere18g}}]\label{mcc-extent-def}
  Let $\mathds{M}_j = (\module{M}_j,\CDN_j,\A_j,\Lip_j,\B_j,\SLip_j)$ be a metrical C*-correspondence, for each $j \in \{1,2\}$. Let $\tau = (\mathds{P},(\Pi_1,\pi_1,\theta_1),(\Pi_2,\pi_2,\theta_2))$ be a metrical tunnel from $\mathds{M}_1$ to $\mathds{M}_2$, with $\mathds{P} = (\module{P},\TDN,\D,\Lip_\D,\alg{E},\Lip_{\alg{E}})$.
  
  The \emph{extent} $\tunnelextent{\tau}$ of a metrical tunnel $\tau$ is
  \begin{equation*}
    \tunnelextent{\tau} \coloneqq \max\left\{ \tunnelextent{\D,\Lip_\D,\pi_1,\pi_2}, \tunnelextent{\alg{E},\TLip_{\alg{E}},\theta_1,\theta_2} \right\} \text.
  \end{equation*}
\end{definition}

Given two metric spectral triples, we can thus either take the Gromov-Hausdorff distance between their underlying {\qcms s}, or take the metrical propinquity \cite{Latremoliere16c,Latremoliere18d} between the metrical C*-correspondence they define, which is defined as the infimum of the extent of every possible metrical tunnels between them. However, the spectral propinquity involves our work on the geometry of quantum dynamics \cite{Latremoliere18a,Latremoliere18b,Latremoliere18g} as well. We recall the construction of the spectral propinquity; the new quantity called the $\varepsilon$-magnitude was introduced in \cite[Definition 3.31]{Latremoliere18g}, but is simpler to express for spectral triples, based on \cite{Latremoliere22}.

\begin{definition}[{\cite[Theorem 3.6]{Latremoliere22}}]\label{magnitude-def}
  Let $(\A_1,\Hilbert_1,\Dirac_1)$ and $(\A_2,\Hilbert_2,\Dirac_2)$ be two metric spectral triples. Let
  \begin{equation*}
    \tau \coloneqq \left( \underbracket[1pt]{\left(\module{P},\TDN,\D,\Lip_\D,\alg{E},\SLip\right)}_{\text{metrical C*-correspondence}}, \underbracket[1pt]{\left(\Pi_1,\pi_1,\theta_1\right)}_{\text{metrical quantum isometry}}, \underbracket[1pt]{\left(\Pi_2,\pi_2,\theta_2\right)}_{\text{metrical quantum isometry}} \right)
  \end{equation*}
  be a metrical tunnel from $\mcc{\A_1}{\Hilbert_1}{\Dirac_1}$ to $\mcc{\A_2}{\Hilbert_2}{\Dirac_2}$, We define the \emph{$\varepsilon$-magnitude} $\tunnelmagnitude{\tau}{\varepsilon}$ of $\tau$ as the maximum of the extent $\tunnelextent{\tau}$ of $\tau$, and the \emph{$\varepsilon$-reach} of $\tau$, which is the number:
  \begin{equation}
    \sup_{\substack{\xi\in\dom{\Dirac_j} \\ \CDN_j(\xi)\leq 1}} \; \operatorname*{\inf\vphantom{p}}_{\substack{\eta\in\dom{\Dirac_k} \\ \CDN_k(\eta)\leq 1}}\;  \sup_{\substack{\omega\in\dom{\TDN} \\ \TDN(\omega)\leq 1 \\ 0\leq t \leq \frac{1}{\varepsilon}}} \left|\inner{\exp(i t \Dirac_j)\xi}{\Pi_j(\omega)}{\Hilbert_j} - \inner{\exp(i t \Dirac_k)\eta}{\Pi_k(\omega)}{\Hilbert_k} \right| \text,
  \end{equation}
  for $\{j,k\}=\{1,2\}$.
\end{definition}

\begin{definition}[{\cite[Definition 4.2]{Latremoliere18g}}]\label{spectral-propinquity-def}
  The \emph{spectral propinquity} between two metric spectral triples $(\A_1,\Hilbert_1,\Dirac_1)$ and $(\A_2,\Hilbert_2,\Dirac_2)$ is
  \begin{equation*}
    \begin{split}
      \spectralpropinquity{}((\A_1,\Hilbert_1,\Dirac_1),
      &(\A_2,\Hilbert_2,\Dirac_2)) \coloneqq \\
      &\inf\Bigg\{ \frac{\sqrt{2}}{2}, \varepsilon > 0 : \tunnelmagnitude{\tau}{\varepsilon} < \varepsilon \text{ for $\tau$ a tunnel} \\
      &\text{from $\mcc{\A_1}{\Hilbert_1}{\Dirac_1}$ to $\mcc{\A_2}{\Hilbert_2}{\Dirac_2}$} \Bigg\} \text.
    \end{split}
  \end{equation*}
\end{definition}

The key property of the spectral propinquity is that, for any two metric spectral triples $(\A_1,\Hilbert_1,\Dirac_1)$ and $(\A_2,\Hilbert_2,\Dirac_2)$, we have the following equivalence:
\begin{equation*}
  \spectralpropinquity{}((\A_1,\Hilbert_1,\Dirac_1),(\A_2,\Hilbert_2,\Dirac_2)) = 0
\end{equation*}
if, and only if, there exists a unitary $U : \Hilbert_1 \rightarrow \Hilbert_2$ such that
\begin{itemize}
\item $U\dom{\Dirac_1} = \dom{\Dirac_2}$,
\item $U\Dirac_1 = \Dirac_2 U$ on $\dom{\Dirac_1}$,
\item $a\in\A_1\mapsto U a U^\ast$ is a $\ast$-isomorphism from $\A_1$ onto $\A_2$.
\end{itemize}

A nontrivial example of convergence in the sense of the spectral propinquity is provided in \cite{Latremoliere21a} with the approximation of spectral triples on quantum tori by spectral triples of certain matrix algebras known as fuzzy tori. These examples include many examples of previously informally stated convergences in mathematical physics, dealing with matrix models and their limits as the dimension of the algebra grows to infinity. Such examples are a major motivation for the construction of the spectral propinquity. Another example on fractals is presented in \cite{Latremoliere20a}. Moreover, convergence for the spectral propinquity implies convergence of the spectra of the Dirac operators and, in an appropriate sense, the convergence of the bounded functional calculi of these operators, among other properties. Of course, convergence for the spectral propinquity implies convergence of the underlying {\qcms s} for the propinquity. In this paper, we will construct new examples of convergence for new spectral triples defined over noncommutative solenoids and over Bunce-Deddens algebras, seen as limits of spectral triples. 

\subsection{Preliminaries: Inductive Limits of Spectral Triples}

While the spectral propinquity allows the discussion of convergence of spectral triples defined on vastly different C*-algebras, there are certain more restricted situations where the C*-algebras of a sequence of spectral triples may be related in a manner compatible with the spectral triples. In \cite{Floricel19}, a simple notion of inductive limit for spectral triples is introduced, based on the following encoding of such a compatibility via a natural, and rigid, notion of morphism between spectral triples.

\begin{definition}[{\cite{Floricel19}}]\label{iso-isomorphism-spectral-triple-def}
  An isometric morphism $(\pi,S)$ from $(\A_1,\Hilbert_1,\Dirac_1)$ to $(\A_2,\Hilbert_2,\Dirac_2)$ is given by a unital $\ast$-morphism $\pi : \A_1\rightarrow\A_2$ and a linear isometry $S : \Hilbert_1 \rightarrow \Hilbert_2$ such that:
  \begin{enumerate}
  \item $\pi(\dom{\Lip_1})\subseteq\dom{\Lip_2}$,
  \item $S\dom{\Dirac_1}\subseteq\dom{\Dirac_2}$ and $S\Dirac_1=\Dirac_2 S$ on $\dom{\Dirac_1}$,
  \item $\forall a \in \A_1\quad S a = \pi(a) S$.
  \end{enumerate}
\end{definition}

Since $S$ is a linear isometry, $\Hilbert_1$ can be identified with the closed subspace $S\Hilbert_1$ of $\Hilbert_2$ via $S$ at no cost in our definition. In that case, $\Dirac_1$ is only defined on $\Hilbert_1\subseteq\Hilbert_2$, and we simply require that $\Dirac_1$ is the restriction of $\Dirac_2$ to $\dom{\Dirac_1}$.

We also note that if $\pi(a) = 0$ for some $a\in\A_1$, then $\pi(a)S = S a = 0$. Since $S$ is an isometry, $a=0$. So $\pi$ is actually automatically a $\ast$-monomorphism, and we thus can also identify $\A_1$ with the C*-subalgebra $\pi(\A_1)$ of $\A_2$, since Definition (\ref{iso-isomorphism-spectral-triple-def}) ensures that $a\Hilbert_1\subseteq\Hilbert_1$ and $[\Dirac_1,a]$ is identified with $P[\Dirac_2,\pi(a)]P = P[\Dirac_2,\pi(a)]=[\Dirac_2,\pi(a)]P$ where $P$ is the orthogonal projection of $\Hilbert_2$ onto $\Hilbert_1$. Furthermore, since $\pi$ is unital, the unit of $\A_2$ is contained in $\A_1$ with this identification.

An inductive sequence of spectral triples, as defined in \cite{Floricel19}, with a somewhat more involved notation, is simply a sequence of the form $((\A_n,\Hilbert_n,\Dirac_n),(\pi_n,S_n))_{n\in\N}$ where $(\A_n,\Hilbert_n,\Dirac_n)$ is a spectral triple and $(\pi_n,S_n)$ is an isometric morphism from $(\A_n,\Hilbert_n,\allowbreak \Dirac_n)$ to $(\A_{n+1},\Hilbert_{n+1},\Dirac_{n+1})$, for each $n\in\N$. As we have seen above, we can identify such a sequence with one of the following type, which we will take as our notion of inductive limit of spectral triples.

\begin{definition}\label{ind-limit-spectral-triple-def}
  Let $\A_\infty = \closure{\bigcup_{n\in\N}\A_n}$ be a C*-algebra which is the closure of an increasing sequence of C*-subalgebras $(\A_n)_{n\in\N}$ in $\A_\infty$, with the unit of $\A_\infty$ in $\A_0$. A spectral triple $(\A_\infty,\Hilbert_\infty,\Dirac_\infty)$ is the inductive limit of a sequence $(\A_n,\Hilbert_n,\Dirac_n)_{n\in\N}$ of spectral triples when:
  \begin{enumerate}
  \item $\Hilbert_\infty = \closure{\bigcup_{n\in\N}}\Hilbert_n$, where each $\Hilbert_n$ is a Hilbert subspace of $\Hilbert_\infty$,
  \item for each $n\in\N$, the restriction of $\Dirac_\infty$ to $\dom{\Dirac_n}$ is $\Dirac_n$,
  \item for each $n\in\N$, the subspace $\Hilbert_n$ is reducing for $\A_n$, which is equivalent to $\A_n \Hilbert_n \subseteq \Hilbert_n$.
  \end{enumerate}
\end{definition}
We note, using the notation of Definition (\ref{ind-limit-spectral-triple-def}), that the operator which, to any $\xi \in \bigcup_{n\in\N}\dom{\Dirac_n}$, associates $\Dirac_n\xi$ whenever $\xi \in \dom{\Dirac_n}$ for any $n\in\N$, is indeed well-defined, and shown in \cite{Floricel19} to be essentially self-adjoint, so $\Dirac_\infty$ is the closure of this operator.

For our purpose, the following result from \cite{Floricel19} will play an important role.
\begin{theorem}[{\cite[Theorem 3.1, partial]{Floricel19}}]\label{Floricel-thm}
  If $(\A_n,\Hilbert_n,\Dirac_n)_{n\in\N}$ is an inductive sequence of spectral triples converging to a \emph{spectral triple} $(\A_\infty,\Hilbert_\infty,\Dirac_\infty)$, then for any $\C$-valued continuous function $f \in C_0(\R)$ which vanishes at infinity, the sequence $(P_n f(\Dirac_n) P_n)_{n\in\N}$ converges to $f(\Dirac_\infty)$ in norm.
\end{theorem}

This section is concerned with the question: if a spectral triple is an inductive limit of spectral triples, then what additional assumptions should be made to get a more geometric convergence, specifically in the sense of the spectral propinquity? In order to make sense of this question, we will work with metric spectral triples, which give rise to {\qcms s}, and lie within the realm of noncommutative metric geometry.

\subsection{Main result}

The notion of inductive limit of spectral triples is simpler to define than the spectral propinquity but only applies to rather narrow examples --- it is not applicable to fuzzy and quantum tori \cite{Latremoliere21a} or the fractals in \cite{Latremoliere20a}. It is certainly interesting to wonder how much metric information from the spectral triples are continuous with respect to the inductive limit process. In this section, we establish a sufficient condition for the convergence, in the sense of the spectral propinquity, of a sequence of metric spectral triples which already converges to a metric spectral triple in the categorical sense. This sufficient condition is simply the existence of an appropriate bridge builder which is also a full quantum isometry. Thus, the main difficulty in establishing convergence for the spectral propinquity, in this context, reduces to proving metric convergence for the propinquity using adequate tunnels.

\begin{theorem}\label{main-spec-thm}
  Let $(\A_\infty,\Hilbert_\infty,\Dirac_\infty)$ be a metric spectral triple which is the inductive limit of a sequence of metric spectral triples $(\A_n,\Hilbert_n,\Dirac_n)$, in the sense of Definition (\ref{ind-limit-spectral-triple-def}). For each $n\in\Nbar$, let
  \begin{equation*}
    \dom{\Lip_n} \coloneqq \left\{a\in\sa{\A_n} : a\,\dom{\Dirac_n}\subseteq\dom{\Dirac_n} \text{ and }[\Dirac_n,a]\text{ is bounded} \right\}\text,
  \end{equation*}
  and  for all $a\in\dom{\Lip_n} $, define
  \begin{equation*}
   \quad \Lip_n(a) \coloneqq \opnorm{[\Dirac_n,a]}{}{\Hilbert_n} \text.
  \end{equation*}

  If there exists a full quantum isometry $\pi : (\A_\infty,\Lip_\infty) \rightarrow (\A_\infty,\Lip_\infty)$ which is also a bridge builder for $((\A_n,\Lip_n)_{n\in\N},(\A_\infty,\Lip_\infty))$, then
  \begin{equation*}
    \lim_{n\rightarrow\infty} \spectralpropinquity{}((\A_n,\Hilbert_n,\Dirac_n),(\A_\infty,\Hilbert_\infty,\Dirac_\infty)) = 0 \text.
  \end{equation*}
\end{theorem}

\begin{proof}
  Fix $\varepsilon > 0$. By Proposition (\ref{bridge-builder-prop}), the sequence $(\A_n,\Lip_n)_{n\in\N}$ converges to $(\A_\infty,\Lip_\infty)$ for the propinquity. More specifically, set, for convenience, $\tilde\varepsilon=\frac{\varepsilon}{2} > 0$. Let $N_\pi\in\N$ be given so that, for all $n\geq N_\pi$, we have:
  \begin{itemize}
  	\item $\forall a \in \dom{\Lip_\infty} \quad \exists b \in \dom{\Lip_n} :\quad \Lip_n(b)\leq\Lip_\infty(a) \text{ and }\norm{\pi(a)-b}{\A_\infty} < \tilde{\varepsilon}\Lip_\infty(a)$,	
\item $\forall b \in \dom{\Lip_n} \quad \exists a \in \dom{\Lip_\infty} :\quad \Lip_\infty(a)\leq\Lip_n(b)	 \text{ and }\norm{\pi(a)-b}{\A_\infty} < \tilde{\varepsilon}\Lip_n(b)$.
  \end{itemize}

  For each $n\in\N$, we constructed in Proposition (\ref{bridge-builder-prop}) a tunnel $\tau_n = (\D_n,\TLip_n,\psi_n,\theta_n)$ with $\D_n = \A_\infty\oplus\A_n$, and for all $(a,b) \in \dom{\Lip_\infty}\oplus\dom{\Lip_n}$,
  \begin{equation*}
    \TLip_n(a,b) \coloneqq \max\left\{ \Lip_\infty(a),\Lip_n(b),\frac{1}{\tilde{\varepsilon}}\norm{\pi(a)-b}{\A_\infty}\right\}\text,
  \end{equation*}
  while $\psi_n : (a,b) \in \D_n\mapsto a$, $\theta_n:(a,b)\in\D_n\mapsto b$. We proved that $\tunnelextent{\tau_n}<\tilde{\varepsilon}$. It is immediate, since $\pi$ is a full quantum isometry, that $\tau'_n \coloneqq(\D_n,\TLip_n,\pi\circ\psi_n,\theta_n)$ is also a tunnel with the same extent as $\tau_n$.

  \medskip

  For each $n\in\Nbar$ and for all $\xi \in \dom{\Dirac_n}$, we define 
  \begin{equation*}
  \CDN_n(\xi) \coloneqq \norm{\xi}{\Hilbert_n} + \norm{\Dirac_n\xi}{\Hilbert_n}\text,
  \end{equation*}
   following Expression (\ref{Dirac-D-norm-eq}).
  
  Now, since $\CDN_\infty$ is a D-norm, the set $X_\infty = \left\{ \xi \in \dom{\Dirac_\infty} : \CDN_\infty(\xi) \leq 1 \right\}$ is compact in $\Hilbert_\infty$. Thus, there exists a finite subset $F\subseteq X_\infty$ of $X_\infty$ such that $\Haus{\Hilbert_\infty}(X_\infty,F)<\frac{\tilde{\varepsilon}}{3}$. As $\Dirac_\infty$ is the closure of an operator on $\bigcup_{n\in\N}\Hilbert_n$ by \cite{Floricel19}, for any $\xi \in F$, there exists a sequence $(\xi_n)_{n\in\N}$, with $\xi_n\in\bigcup_{j\in\N}\Hilbert_j$ for all $n\in\N$, such that $\lim_{n\rightarrow\infty}\xi_n = \xi$, and $\lim_{n\rightarrow\infty}\Dirac_\infty\xi_n = \Dirac_\infty\xi$. Since $F$ is finite, there exists $N_F\in\N$ such that if $n\geq N_F$ and $\xi \in F$, then $\norm{\xi-\xi_n}{\Hilbert_\infty} < \frac{\tilde{\varepsilon}}{3}$ and $\norm{\Dirac_\infty\xi-\Dirac_\infty\xi_n}{\Hilbert_\infty} < \frac{\tilde{\varepsilon}}{3}$. Again by Definition (\ref{ind-limit-spectral-triple-def}), we also have $\Dirac_\infty \xi_n = \Dirac_n \xi_n$.

  Fix $n\in\N, n \geq N\coloneqq \max\{N_\pi,N_F\}$. Let $\module{M}_n \coloneqq \Hilbert_\infty \oplus \Hilbert_n$, seen as a $\D_n$-$(C\oplus\C)$ C*-correspondence, with the C*-correspondence structure:
  \begin{equation*}
    \forall (a,b) \in \D_n \quad \forall (\xi,\eta) \in \module{M}_n \quad (a,b)\triangleleft(\xi,\eta) \coloneqq (\pi(a)\xi,b\eta)\text,
  \end{equation*}
  and
  \begin{equation*}
    \forall (\xi,\eta),(\xi',\eta') \in \module{M}_n \quad \inner{(\xi,\xi')}{(\eta,\eta')}{n} \coloneqq \left( \inner{\xi}{\xi'}{\Hilbert_\infty}, \inner{\eta}{\eta'}{\Hilbert_n} \right) \in \C\oplus\C \text{,}
  \end{equation*}
  while
  \begin{equation*}
    \forall (t,s) \in \C\oplus\C \quad \forall (\xi,\eta) \in \module{M}_n \quad (\xi,\eta)\cdot(t,s) \coloneqq (t\xi,s\eta)\text.
  \end{equation*}
  We note that here, $\C^2$ is the C*-algebra of $\C$-valued functions over a two points set, and in particular, the norm of $(z,w)\in\C^2$ is $\max\{|z|,|w|\}$.
  
  We then define, for all $(\xi,\eta)\in\dom{\Dirac_\infty}\oplus\dom{\Dirac_n}$:
  \begin{equation*}
    \TDN_n(\xi,\eta) \coloneqq \max\left\{ \CDN_\infty(\xi), \CDN_n(\eta), \frac{1}{\tilde{\varepsilon}} \norm{\xi-\eta}{\Hilbert_\infty} \right\} \text,
  \end{equation*}
  while we also set
  \begin{equation*}
    \QLip : (z,w) \in \C\oplus \C \mapsto \frac{1}{\tilde{\varepsilon}}|z-w| \text.
  \end{equation*}

  It is immediate to see that $\QLip$ is a Lipschitz seminorm on $\C\oplus\C$ (it is, in fact, the Lipschitz seminorm for the metric on the two point set which places these two points exactly $\tilde{\varepsilon}$ apart).

  Now, we check that $\TDN_n$ is a $D$-norm. Of course, for all $(\xi,\eta)\in\module{M}_n$:
  \begin{equation*}
    \TDN_n(\xi,\eta)\geq\max\{\CDN_\infty(\xi),\CDN_n(\eta)\} \geq \max\left\{\norm{\xi}{\Hilbert_\infty},\norm{\eta}{\Hilbert_n}\right\} = \norm{(\xi,\eta)}{\module{M}_n}\text.
  \end{equation*}
  
  We observe that
  \begin{equation*}
  \begin{split}
    \{(\xi,\eta)\in\module{M}_n:\TDN_n(\xi,\eta)\leq 1\} &\subseteq \\
    \{\xi\in\dom{\Dirac_\infty}:&\CDN_\infty(\xi)\leq 1\}\times \{\eta\in\dom{\Dirac_n}:\CDN_n(\eta)\leq 1\}\text,
  \end{split}
\end{equation*}
  the latter set being compact as a product of two compact sets -- since $\CDN_n$ and $\CDN_\infty$ are indeed D-norms. Since in addition, $\TDN_n$ is lower semicontinuous over $\module{M}_n$ as the maximum of three lower semicontinuous functions over this space, the unit ball of $\TDN_n$ is indeed closed, hence compact, in $\module{M}_n$ (which is complete). 
  We now check the Leibniz inequalities. If $(a,b)\in\dom{\TLip_n}$ and $(\xi,\eta)\in\dom{\TDN_n}$, then we compute:
  \begin{align*}
    \norm{(a,b)\triangleleft(\xi,\eta)}{\Hilbert_\infty}
    &=\norm{\pi(a)\xi-b\eta}{\Hilbert_\infty} \\
    &\leq \norm{\pi(a)-b}{\A_\infty}\norm{\xi}{\Hilbert_\infty} + \norm{b}{\A_\infty} \norm{\xi-\eta}{\Hilbert_\infty} \\
    &\leq \tilde{\varepsilon}\TLip_n(a,b) \CDN_n(\xi) + \norm{(a,b)}{\D_n} \tilde{\varepsilon} \TDN_n(\xi,\eta) \\
    &\leq \tilde{\varepsilon}\left(\TLip_n(a,b) + \norm{(a,b)}{\D_n}\right)\TDN_n(\xi,\eta) \text.
  \end{align*}
  From this, it follows that for all $(a,b)\in\dom{\TLip_n}$ and for all $(\xi,\eta)\in\dom{\TDN_n}$,
  \begin{equation*}
    \TDN_n((a,b)\triangleleft(\xi,\eta)) \leq \left(\TLip_n(a,b) + \norm{(a,b)}{\D_n}\right)\TDN_n(\xi,\eta) \text.
  \end{equation*}
  
  On the other hand, if $(\xi,\eta),(\xi',\eta')\in\dom{\TDN_n}$, we have:
  \begin{align*}
    \QLip(\inner{(\xi,\eta)}{(\xi',\eta')}{\module{M}_n})
    &= \frac{1}{\tilde{\varepsilon}}\left|\inner{\xi}{\xi'}{\Hilbert_\infty} - \inner{\eta}{\eta'}{\Hilbert_\infty}\right| \\
    &\leq \frac{1}{\tilde{\varepsilon}}\left( \left|\inner{\xi-\eta}{\xi'}{\Hilbert_\infty}\right| + \left|\inner{\eta}{\xi'-\eta'}{\Hilbert_\infty}\right| \right) \\
    &\leq \frac{1}{\tilde{\varepsilon}}\left( \norm{\xi-\eta}{\Hilbert_\infty} \norm{\xi'}{\Hilbert_\infty} + \norm{\eta}{\Hilbert_\infty} \norm{\xi'-\eta'}{\Hilbert_\infty} \right) \\
    &\leq \TDN_n(\xi,\eta)\norm{\xi'}{\Hilbert_\infty} + \norm{\eta}{\Hilbert_\infty}\TDN_n(\xi',\eta') \\
    &\leq 2 \TDN_n(\xi,\eta)\TDN_n(\xi',\eta') \text.
  \end{align*}

  \medskip
  
  We now define the maps: 
  
  \begin{equation*} \Pi_n:(\xi,\eta)\in\module{M}_n\mapsto \xi \in \Hilbert_\infty, \text{ and }\Theta_n : (\xi,\eta)\in\module{M}_n \mapsto \eta\in\Hilbert_n.\end{equation*} Our goal is to show that
  \begin{equation*}
    \Upsilon_n \coloneqq \left( \mathds{M}_n, (\Pi_n,\pi\circ\psi_n), (\Theta_n,\theta_n) \right) \text{ where }\mathds{M}_n \coloneqq (\module{M}_n,\TDN_n,\D_n,\TLip_n,\C\oplus\C,\QLip)
  \end{equation*}
  is a metrical tunnel, using Definition (\ref{mcc-tunnel-def}).

  By construction, $\Pi_n(a\cdot\xi,b\cdot\eta) = \pi(a)\xi = \pi\circ\psi_n(a,b)\Pi_n(\xi,\eta)$ and $\Theta_n(a\cdot \xi, b\cdot \eta) = b\eta = \theta_n(a,b)\Theta_n(\xi,\eta)$, for all $(a,b)\in\D_n$ and $(\xi,\eta)\in\module{M}_n$.
  
  Now, let $\xi \in \Hilbert_\infty$ with $\CDN_\infty(\xi) = 1$. By construction of $F$, there exists $\xi' \in F$ such that $\norm{\xi-\xi'}{\Hilbert_\infty}<\frac{\tilde{\varepsilon}}{3}$. By our choice of $N$, there exists $\eta (=\xi'_n) \in \Hilbert_n$ such that $\CDN_n(\eta)\leq 1+\frac{\tilde{\varepsilon}}{3}$ and $\norm{\xi'-\eta}{\Hilbert_\infty}<\frac{\tilde{\varepsilon}}{3}$. Let $\chi = \frac{1}{1+\frac{\tilde{\varepsilon}}{3}}\eta \in \Hilbert_n$, so that $\CDN_n(\chi)\leq 1$. Moreover,
  \begin{align*}
    \norm{\xi-\chi}{\Hilbert_\infty}
    &\leq \norm{\xi-\eta}{\Hilbert_\infty} + \frac{\frac{\tilde{\varepsilon}}{3}}{1+\frac{\tilde{\varepsilon}}{3}}\norm{\eta}{\Hilbert_\infty} \\
    &\leq \norm{\xi-\eta}{\Hilbert_\infty} + \frac{\frac{\tilde{\varepsilon}}{3}}{1+\frac{\tilde{\varepsilon}}{3}}\CDN_n(\eta) \\
    &\leq \norm{\xi-\xi'}{\Hilbert_\infty} + \norm{\xi'-\eta}{\Hilbert_\infty} + \frac{\tilde{\varepsilon}}{3} \\
    &< \tilde{\varepsilon} \text.
  \end{align*}
  Thus $\TDN_n(\xi,\chi) = 1$, and therefore, $(\Pi_n,\pi\circ\psi_n)$ is indeed a metrical quantum isometry.

  Now, let $\eta\in\Hilbert_n$. By construction, $\Dirac_\infty\eta=\Dirac_n\eta$, so $\CDN_\infty(\eta)=\CDN_n(\eta)$. Therefore, $\TDN_n(\eta,\eta) = \CDN_n(\eta)$. Again, we conclude that $(\Theta_n,\theta_n)$ is a metrical quantum isometry as well.

  Therefore, $\Upsilon_n$ is a metrical tunnel. It is immediate, of course, that the canonical surjections from $\C\oplus\C$ to $\C$ are quantum isometries --- the only Lipschitz seminorm on $\C$ being the $0$ function. So $\Upsilon_n$ is a metrical tunnel.

  \medskip
  
  We now compute the extent of $\Upsilon_n$. It is, by Definition (\ref{mcc-extent-def}), the maximum of the extent of the tunnel $\tau'_n$, which is at most $\tilde{\varepsilon}$, and the extent of the tunnel $(\C,0) \xleftarrow{} (\C\oplus\C,\QLip) \xrightarrow{} (\C,0)$, which is immediately computed to be $\tilde{\varepsilon}$. So the extent of $\Upsilon_n$ is $\tilde{\varepsilon}$ as well.

  Therefore, for all $n\geq N$, we have
  \begin{equation*}
    \dmetpropinquity{}((\Hilbert_n,\CDN_n,\A_n,\Lip_n,\C,0),(\Hilbert_\infty,\CDN_\infty,\A_\infty,\Lip_\infty,\C,0)) \leq \tunnelextent{\Upsilon_n} = \tilde{\varepsilon} < \varepsilon \text,
  \end{equation*}
  and therefore,
  \begin{equation*}
    \lim_{n\rightarrow\infty}\dmetpropinquity{}((\Hilbert_n,\CDN_n,\A_n,\Lip_n,\C,0),(\Hilbert_\infty,\CDN_\infty.\A_\infty,\Lip_\infty,\C,0)) = 0\text.
  \end{equation*}

  \medskip
  
  It remains to compute an upper bound for the $\varepsilon$-reach of our tunnels $\Upsilon_n$ (see Definition \ref{magnitude-def}). We will once again use our finite set $F$ with $\Haus{\Hilbert_\infty}(F,X_\infty) < \frac{\tilde{\varepsilon}}{3}$ where $X_\infty$ is the closed unit ball of $\CDN_\infty$.
  
  Now, let $(v_k)_{k\in\N}$ be a sequence of continuous functions on $\R$ vanishing at $\infty$, valued in $[0,1]$, and converging pointwise to $1$ over $\R$. Therefore, $(v_k(\Dirac_\infty))_{k\in\N}$ converges to $\Dirac_\infty$ in the strong operator topology. Since $F$ is finite, there exists $k \in \N$ such that,  for all $\xi\in F $
  \begin{equation}\label{F-eq}
   \quad \norm{v_k(\Dirac_\infty)\xi-\xi}{\Hilbert_\infty} < \frac{\tilde{\varepsilon}}{12} \text.
  \end{equation}

  We identify, from now on, $\Dirac_n$ with the linear operator on $\Hilbert_\infty$ whose restriction to $\Hilbert_n$ is $\Dirac_n$, and whose restriction to $\Hilbert_n^\perp$ is $0$; thus $\dom{\Dirac_n}$ is replaced with $\dom{\Dirac_n}\oplus\Hilbert_n^\perp$. We denote by $P_n$ the orthogonal projection of $\Hilbert_\infty$ onto $\Hilbert_n$, so that $P_n\Dirac_n = \Dirac_nP_n=\Dirac_n$ on $\dom{\Dirac_n}$.
  
  For each $t \in [0,\infty)$, let $u^t: s \in \R \mapsto \exp(its)$, and for each $n\in\Nbar$, we denote $u^t(\Dirac_n)$ by $U_n^t$.

  Fix $t\in\R$. The function $u^t v_k$ is continuous over $\R$ and vanishes at infinity. By Theorem (\ref{Floricel-thm}), since $(\A_\infty,\Hilbert_\infty,\Dirac_\infty)$ is a spectral triple, and the inductive limit of the sequence $(\A_n,\Hilbert_n,\Dirac_n)_{n\in\N}$ of spectral triples, the sequence of operators $(P_nu^t v_k (\Dirac_n)P_n)_{n\in\N}$ converges in norm to $u^t v_k(\Dirac_\infty)$. Moreover, $u^tv_k(\Dirac_n)P_n = P_nu^tv_k(\Dirac_n)P_n$ for all $n\in\N$ by construction.
  
  Let $F'$ be a finite subset of the compact set $\left[0,\frac{1}{\varepsilon}\right]$ such that $\Haus{\R}(F',\left[0,\frac{1}{\varepsilon}\right]) < \frac{\tilde{\varepsilon}}{12}$. Since $F'$ is finite, there exists $N_\nu \in \N$ such that if $n\geq N_\nu$, then for all $t \in F' $:
  \begin{equation}\label{Fprime-eq}
     \quad \opnorm{U_n^t (v_k(\Dirac_n))P_n - U_\infty^t (v_k(\Dirac_\infty))}{}{\Hilbert_\infty} < \frac{\tilde{\varepsilon}}{12}\text.
  \end{equation}

  Let $n\in\Nbar$. Now, we note that if $\xi \in \dom{\CDN_n}$ with $\CDN_n(\xi)\leq 1$, then for all $s < t \in \R$:
  \begin{align*}
    \norm{U_n^t\xi - U_n^s\xi}{\Hilbert_n}
    &\leq \int_s^t \norm{\frac{d}{dr}U_n^r\xi}{\Hilbert_n} \, dr \\
    &\leq \int_s^t \norm{U_n^r\Dirac_n\xi}{\Hilbert_n} \, dr \\
    &\leq |s-t| \text.
  \end{align*}
  Thus, for all $s,t \in \R$ and $\xi \in \dom{\CDN_n}$ with $\CDN_n(\xi)\leq 1$, we have
  \begin{equation*}
    \norm{U_n^t\xi-U_n^s\xi}{\Hilbert_n} \leq |s-t| \text.
  \end{equation*}
 Now, let $n\geq N' \coloneqq \max\{N_\nu,N_F\}$. Since $\Dirac_n$ and $P_n$ commute, if $\xi \in X_\infty$, then $\CDN_n(\xi)\leq\CDN_\infty(\xi)$ and:
  \begin{align*}
    \CDN_n(\nu_k(\Dirac_n)P_n\xi)
    &=\norm{v_k(\Dirac_n)P_n \xi}{\Hilbert_\infty} + \norm{\Dirac_n v_k(\Dirac_n) P_n\xi}{\Hilbert_\infty} \\
    &\leq \norm{v_k}{C_0(\R)} \opnorm{P_n}{}{\Hilbert_\infty}\CDN_n(\xi) \leq 1 \text. 
  \end{align*}

  For all $\xi \in X_\infty$ and $t \in \left[0,\frac{1}{\tilde{\varepsilon}}\right]$, let $s \in F'$ and $\xi' \in F$ such that $|s-t|<\frac{\tilde{\varepsilon}}{12}$, $\norm{\xi-\xi'}{\Hilbert_\infty}<\frac{\tilde{\varepsilon}}{3}$, then:
  \begin{align*}
    \norm{U_n^t \nu_k(\Dirac_n) P_n \xi - U_\infty^t\xi}{\Hilbert_\infty}
    &\leq \underbracket[1pt]{\norm{U_n^t\nu_k(\Dirac_n) P_n\xi - U_n^s\nu_k(\Dirac_n)P_n\xi}{\Hilbert_\infty}}_{\leq |t-s|<\frac{\tilde{\varepsilon}}{12} \text{ since }\CDN_n(\nu_k(\Dirac_n)P_n\xi)\leq 1} \\
    &\quad + \underbracket[1pt]{\norm{U_n^s\nu_k(\Dirac_n)P_n \xi - U_\infty^s\nu_k(\Dirac_\infty)\xi}{\Hilbert_\infty}}_{\leq\opnorm{U_n^s\nu_k(\Dirac_n)P_n -U_\infty^s\nu_k(\Dirac_\infty)}{}{\Hilbert_\infty}<\frac{\tilde{\varepsilon}}{12} \text{ by Eq. (\ref{Fprime-eq})}} \\
    &\quad + \underbracket[1pt]{\norm{U_\infty^s\nu_k(\Dirac_\infty)(\xi-\xi')}{\Hilbert_\infty}}_{\leq\norm{\xi-\xi'}{\Hilbert_\infty}<\frac{\tilde{\varepsilon}}{3} \text{ by Eq. (\ref{F-eq})}} \\
    &\quad + \underbracket[1pt]{\norm{U_\infty^s \nu_k(\Dirac_\infty)\xi' - U_\infty^s \xi'}{\Hilbert_\infty}}_{\leq\norm{\nu_k(\Dirac_\infty)\xi'-\xi'}{\Hilbert_\infty}<\frac{\tilde{\varepsilon}}{12}} \\
    &\quad + \underbracket[1pt]{\norm{U_\infty^s\xi' - U_\infty^t \xi'}{\Hilbert_\infty}}_{\leq |s-t|<\frac{\tilde{\varepsilon}}{12}} \\
    &\quad + \underbracket[1pt]{\norm{U_\infty^t\xi' - U_\infty^t\xi}{\Hilbert_\infty}}_{\leq\norm{\xi-\xi'}{\Hilbert_\infty}<\frac{\tilde{\varepsilon}}{3}} \\
    &< \tilde{\varepsilon} \text.
  \end{align*}

  Let $\xi \in \dom{\Dirac_\infty}$ with $\CDN_\infty(\xi)\leq 1$. Let $n\geq N'$, and set $\eta = \nu_k(\Dirac_\infty)P_n\xi$. For all $t\in\left[0,\frac{1}{\tilde{\varepsilon}}\right]$, we have, $\eta \in \dom{\Dirac_n}$ and $\CDN_n(\eta)\leq 1$. Therefore,
  \begin{align*}
    \inf_{\substack{\eta\in\dom{\CDN_n} \\ \CDN_n(\eta)\leq 1}}&\sup_{\substack{\omega\in\dom{\TDN_n} \\ \TDN_n(\omega)\leq 1}}\left|\inner{U_n^t\eta}{\Theta_n(\omega)}{\Hilbert_n} - \inner{U_\infty^t\xi}{\Pi_n(\omega)}{\Hilbert_\infty}\right|
    \\ 
    &\leq \sup_{\substack{\omega\in\dom{\TDN_n} \\ \TDN_n(\omega)\leq 1}} \left|\inner{U_n^t\nu_k(\Dirac_n)P_n \xi}{\Theta_n(\omega)}{\Hilbert_n} - \inner{U_\infty^t \xi}{\Pi_n(\omega)}{\Hilbert_\infty}\right| \\
    &\leq \sup_{\substack{\omega\in\dom{\TDN_n} \\ \TDN_n(\omega)\leq 1}}\left(\norm{U_n^t\nu_k(\Dirac_n)\xi - U_\infty^t\xi}{\Hilbert_\infty}\norm{\omega}{\module{M}_n} + \norm{U_\infty^t\xi}{\Hilbert_\infty}\underbracket[1pt]{\norm{\Theta_n(\omega)-\Pi_n(\omega)}{\Hilbert_\infty}}_{<\tilde{\varepsilon} \text{ since }\TDN_n(\omega)\leq 1}\right) \\
    &< \tilde{\varepsilon} + \tilde{\varepsilon} = \varepsilon \text.
  \end{align*}

  Now, take $\xi \in \Hilbert_n$, with $\CDN_n(\xi)\leq 1$. By construction, $\norm{\Dirac_\infty \xi}{\Hilbert_\infty} = \norm{\Dirac_n \xi}{\Hilbert_n}$, and $U_n^t \xi = U_\infty^t \xi$. So for all $\xi \in \dom{\CDN_n}$ with $\CDN_n(\xi)\leq 1$, we have, for all $t\in\R$:
  \begin{multline*}
    \inf_{\substack{\eta\in\dom{\Dirac_\infty} \\ \CDN_\infty(\eta)\leq 1}} \sup_{\substack{\omega\in\dom{\TDN_n} \\ \TDN_n(\omega)\leq 1}} \left|\inner{U_\infty^t\eta}{\Pi_n(\omega)}{\Hilbert_\infty} - \inner{U_n^t\xi}{\Theta_n(\omega)}{\Hilbert_n}\right| \\
    \leq \left|\inner{U_\infty^t\xi}{\Pi_n(\omega)}{\Hilbert_\infty} - \inner{U_n^t\xi}{\Theta_n(\omega)}{\Hilbert_n}\right|  = 0 < \varepsilon
    \text.
  \end{multline*}

  Therefore, for all $n\geq \max\{N,N'\}$, the $\varepsilon$-reach of $\Upsilon_n$ is no more than $\varepsilon$, and thus the $\varepsilon$-magnitude $\tunnelmagnitude{\Upsilon_n}{\varepsilon}$ of $\Upsilon_n$ is no more than $\varepsilon$ (by Definition (\ref{magnitude-def})). Therefore  for all $n \geq N$:
  \begin{equation*}
    \spectralpropinquity{}((\A_n,\Hilbert_n,\Dirac_n),(\A_\infty,\Hilbert_\infty,\Dirac_\infty)) \leq \tunnelmagnitude{\Upsilon_n}{\varepsilon} < \varepsilon \text,
  \end{equation*}
  and thus
  \begin{equation*}
    \lim_{n\rightarrow\infty}  \spectralpropinquity{}((\A_n,\Hilbert_n,\Dirac_n),(\A_\infty,\Hilbert_\infty,\Dirac_\infty)) = 0\text,
  \end{equation*}
  as claimed.
\end{proof}

\begin{remark}
  A corollary of Theorem (\ref{main-spec-thm}) is that we obtain convergence for the bounded continuous functional calculus for the Dirac operators from the work in \cite{Latremoliere22}, which extends Theorem (\ref{Floricel-thm}).
\end{remark}

\section{Even Spectral Triples on Twisted Group $C^*$-algebras}\label{group-sec}

We now apply our results of the previous sections to the construction of  inductive limits of  spectral triples for the spectral propinquity  on twisted C*-algebras of discrete groups endowed with length functions. In particular we will prove in this section our third main theorem, Theorem (\ref{main-group-thm}).
Our approach introduces new metric spectral triples on certain twisted group C*-algebras which generalize the  
related, though distinct,  past constructions using length functions over discrete groups such as the ones in \cite{FarsiPacker22}. Our main applications would be  the construction of new spectral triples over noncommutative solenoids and some Bunce-Deddens algebras. In particular,  we shall prove that the noncommutative solenoids spectral triples  are limits of spectral triples over quantum 2-tori for the spectral propinquity. We will start with detailing in the next two subsections  some background material that will be used to state and prove our main result.

\subsection{Discrete Groups, Proper Length Functions, 2-Cocycles, and Classical Spectral Triples. }

Let $G_\infty$ be a discrete group, and let $\sigma$ be a $2$-cocycle over $G_\infty$. Let $\lambda$ be the left regular $\sigma$-projective representation of $G_\infty$ on $\ell^2(G_\infty)$, defined by, for all $g \in G_\infty $ and for all $\xi \in \ell^2(G_\infty)$:
\begin{equation*}
   \quad \lambda(g)\xi : h \in G_\infty \longmapsto \sigma(g,g^{-1}h)\xi(g^{-1}h) \text.
\end{equation*}
Of course, each operator $\lambda(g)$ is unitary for each $g \in G_\infty$. Let $C_{\mathrm{red}}^\ast(G_\infty,\sigma)$ be the reduced C*-algebra of $G_\infty$ twisted by $\sigma$, i.e. the C*-algebra of operators on $\ell^2(G_\infty)$ generated by $\left\{\lambda(g) : g \in G_\infty \right\}$. For any $f \in \ell^1(G_\infty)$, the operator $\lambda(f)$ on $\ell^2(G_\infty)$ is defined as $\sum_{g \in G_\infty}f(g)\lambda(g)$ --- it is easily checked that $\opnorm{\lambda(f)}{}{\ell^2(G_\infty)} \leq \norm{f}{\ell^1(f)}$. The reduced group C*-algebra $C_{\mathrm{red}}^\ast(G)$ is, in particular, $C_{\mathrm{red}}^\ast(G_\infty,1)$.

In \cite{Connes89}, Connes introduced spectral triples $(C^\ast_{\mathrm{red}}(G_\infty),\ell^2(G_\infty),M_{\mathds{L}})$ using any proper length function $\mathds{L}$ over $G_\infty$, where $M_{\mathds{L}}$ is the operator of multiplication by $\mathds{L}$, defined on its natural domain in the Hilbert space $\ell^2(G_\infty)$. Connes proved that $\opnorm{[M_{\mathds{L}},\lambda(g)]}{}{\ell^2(G_\infty)} = \mathds{L}(g)$ --- which immediately follows from the triangle inequality and the fact that $[M_{\mathds{L}},\lambda(g)]\delta_e \allowbreak = \mathds{L}(g)\sigma(g,1)\delta_g$, where, for all $g \in G_\infty$:
\begin{equation*}
   \quad \delta_g : h \in G_\infty \mapsto \begin{cases} 1 \text{ if $g=h$, } \\ 0 \text{ otherwise.} \end{cases}
\end{equation*}
It then follows that for the $\ast$-algebra $C_c(G_\infty)$ of $\C$-valued functions with finite support, we obtain the inequality,  for all $f \in C_c(G_\infty)$:
\begin{equation}\label{D-upper-bound-eq}
  \quad \opnorm{\left[M_{\mathds{L}},\lambda(g)\right]}{}{\ell^2(G_\infty)} \leq \sum_{g \in G_\infty} |f(g)| \mathds{L}(g) \text,
\end{equation}
since $\lambda(g)$ is unitary for all $g \in G_\infty$.

Note that by construction, for the multiplication operator by $\mathds{L}$ to have compact resolvent, the spectral projection of this operator on any compact interval must have finite rank. Thus, in particular, the set $\{ \delta_h \in \ell^2(G_\infty) : \mathds{L}(h) \leq r \}$ must be finite for all $r\geq 0$. In other words, all closed balls in $G_\infty$ for $\mathds{L}$ must be finite, i.e., $\mathds{L}$ must indeed be proper.

\medskip

However, natural length functions on $G_\infty$ may not be proper, or even give the discrete topology. An example of this situation is given when $G_\infty$ is the additive group $\left(\Z\left[\frac{1}{p}\right]\right)^2$ where:
\begin{equation*}
  \Z\left[\frac{1}{p}\right] \coloneqq \left\{ \frac{a}{p^n} : n \in \N, a \in \Z \right\} \text,
\end{equation*}
and where $p \in \N$ is prime. It is natural to regard $\Z\left[\frac{1}{p}\right]$ as a subgroup of $\Q$, and thus equip it with the induced length function from the usual absolute value on $\Q$ (see Figure (\ref{dyadic-Q-pic})). However, this length function is not proper --- and induces a non-discrete topology. We moreover note that $\Z\left[\frac{1}{p}\right]=\bigcup_{n\in\N}\frac{1}{p^n}\Z$, and we would like to capture this inductive limit structure metrically; while the sequence $\left(\frac{1}{p^n}\Z\right)_{n\in\N}$ converges to $\Z\left[\frac{1}{p}\right]$ for the Hausdorff distance induced by $|\cdot|$, we can not apply this observation directly to the associated twisted C*-algebras since $|\cdot|$ does not define a spectral triple using Connes' methods.

\medskip

Let us discuss this situation by returning to a general discrete group $G_\infty$ and some $2$-cocycle $\sigma$ on $G_\infty$. We now assume that we are given a strictly  increasing sequence $(G_n)_{n\in\N}$ of subgroups of $G_\infty$ such that $G_\infty = \bigcup_{n\in\N} G_n$ --- in fancier terms, $G_\infty$ is the inductive limit of the sequence of groups $(G_n)_{n\in\N}$, which we identify with a sequence of subgroups of $G_\infty$ from now on. We also identify $\sigma$ with its restriction to $G_n$ for all $n\in\N$.

We now have a conundrum. If we choose a proper length function $\mathds{L}$ on $G_\infty$, then, since $G_\infty = \bigcup_{n\in\N}G_n$ with $(G_n)_{n\in\N}$ increasing, any finite subset of $G_\infty$ is contained in some $G_N$ (and thus in all $G_n$ with $n\geq N$). This implies that $(G_n)_{n\in\N}$ converges to $G_\infty$ for the pointed Gromov-Hausdorff distance for proper metric space, where we always use $1$ as our base point, and the metrics are induced by $\mathds{L}$ (see \cite{Gromov81}). On the other hand, as soon as $G_\infty$ is infinite --- which is the only interesting case to consider when $G_\infty$ is the union of countably many groups, otherwise of course $G_\infty$ is just $G_n$ for $n$ large enough --- not only the diameter of $G_\infty$ is infinite --- it can not be a closed ball as these are finite --- but the subgroups $G_n$ are not close to $G_\infty$ for the Hausdorff distance induced by $\mathds{L}$ in general. So, we can define the spectral triples $(C^\ast_{\mathrm{red}}(G_n,\sigma),\ell^2(G_n),M_{\mathds{L}})$ as before since $\mathds{L}$ is proper, but in general, there is no apparent reason why $\opnorm{[M_{\mathds{L}},a]}{}{\ell^2(G_\infty)}$ is particularly close to $\opnorm{[M_{\mathds{L}},a]}{}{\ell^2(G_n)}$ for $a \in C^\ast_{\mathrm{red}}(G_n,\sigma)$.

\medskip

On the other hand, there may be length functions on $G_\infty$ for which $(G_n)_{n\in\N}$ does converge in the Hausdorff distance for these length functions, but these length functions are not proper whenever $G_\infty$ is infinite. We are thus led to build a new type of spectral triples which combine these two apparently opposite situations: one where we do not know how to build a spectral triple using a non-proper length with otherwise good metric properties for our purpose, and one with a proper length function which has bad metric property. The following construction is inspired, but different from \cite{FarsiPacker22}, where a proper length function is constructed as a sum of a non-proper length function with a $p$-norm.

\subsection{The Spectral Triples}
\label{spectral-triples-section}

We now define our new spectral triples on a particular type of twisted group C*-algebras, which are the subject of our main third theorem, Theorem (\ref{main-group-thm}), and its corollaries.

  From now on we assume that  $G_\infty$ is a discrete group  endowed with a $2$-cocycle $\sigma$ with values in $\T \coloneqq \{z\in\C:|z|=1\}$, and that $G_\infty$ is  the union of   a  strictly increasing sequence for inclusion,  $(G_n)$, of subgroups of $G_\infty$ such that $G_\infty = \bigcup_{n\in\N} G_n$. 

  We also assume that we are given a length function $\mathds{L}_H$ on $G_\infty$,  whose restriction to each $G_n$ is proper for each $n\in\N$, and with the property that
  \begin{equation}\label{L-H-eq}
    \lim_{n\rightarrow\infty}\Haus{\mathds{L}_H}(G_\infty,G_n) = 0 \text.
  \end{equation}
  In addition we require  that we are  given a strictly increasing unbounded function $\mathrm{scale} : \N \rightarrow[0,\infty)$, together with  $\mathds{F} : G_\infty \rightarrow [0,\infty)$ such that for all $g \notin G_0$:
  \begin{equation*}
    \mathds{F}(g) = \mathrm{scale}(\min\{ n \in \N : g \in G_n\}),
  \end{equation*}
  while $\mathds{F}$ restricted to $G_0$ satisfies:
  \begin{itemize}
  \item $\forall g \in G_0 \quad \mathds{F}(g) = \mathds{F}(g^{-1})$,
  \item $\forall g,h \in G_0 \quad \mathds{F}(gh) \leq \max\{\mathds{F}(g),\mathds{F}(h)\}$,
  \item $\forall g \in G_0 \quad \mathds{F}(g) \in [0,\mathrm{scale}(0)]$,
  \item $\mathds{F}(1) = 0$.
  \end{itemize}

 Clearly,  the above assumptions  provide us with many length functions on $G_\infty$ and $G_n$; we will use them in our spectral triples constructions.
 
One of our  main examples  for this section will be the noncommutative solenoids, whose fundamental components are  described below.
We will give more details on this example later in this work. 
 
 \begin{example}\label{nc-solenoid-ex}
 	Let $d \geq 2$ and $p$ a prime number. Let $G_\infty = \left(\Z\left[\frac{1}{p}\right]\right)^d$, and let $G_n = \left(\frac{1}{p^n}\Z\right)^d$ for all $n\in\N$. We note that $G_\infty = \bigcup_{n\in\N}G_n$. We can then choose $\mathds{L}_H$ to be the restriction of any norm on $\R^d$, and  $\mathrm{scale}:n \in \N \rightarrow p^n \in [0,\infty)$, so that:
 	\begin{equation*}
 	\mathds{F} : g\in G_\infty\mapsto \mathrm{scale}\left(\min\left\{n\in\N : g \in \left(\frac{1}{p^n}\Z\right)^d \right\}\right)\text.
 	\end{equation*}

 	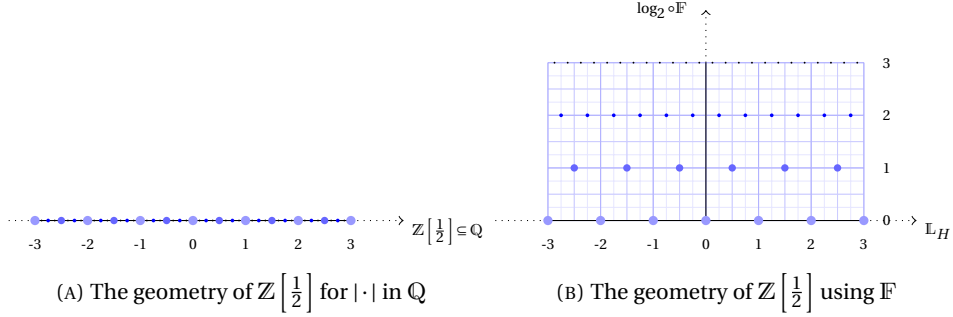
\begin{figure}
 		\centering
 		\subcaptionbox{%
 			The geometry of $\Z\left[\frac{1}{2}\right]$ for $|\cdot|$ in $\Q$}[0.49\textwidth]{
 			\begin{tikzpicture}[scale=0.7]
 			\draw (-3,0)--(3,0);
 			\draw[dotted] (-3.5,0)--(-3,0);
 			\draw[dotted,->] (3,0)--(4,0) node[below=4pt,right]{\tiny $\Z\left[\frac{1}{2}\right]\subseteq \Q$};
 			\foreach \x in {-2.875,-2.75,...,2.875}
 			\fill (\x,0) circle [radius=0.5pt];
 			\foreach \x in {-2.75,-2.5,...,2.75}
 			\fill[blue] (\x,0) circle [radius=1pt];
 			\foreach \x in {-2.5,-2,...,2.5}
 			\fill[blue!60!white] (\x,0) circle [radius=2pt];
 			\foreach \x in {-3,-2,...,3}
 			{
 				\fill[blue!40!white] (\x,0) circle [radius=2.5pt];
 				\draw (\x,0) node[below=4pt]{\tiny \x};
 			}
 			\end{tikzpicture}
 		}%
 		\hfill
 		\subcaptionbox{%
 			The geometry of $\Z\left[\frac{1}{2}\right]$ using $\mathds{F}$}[0.49\textwidth]{
 			\begin{tikzpicture}[scale=0.7]
 			\draw[very thin, blue!10,step=0.25] (-3,0) grid (3,3);
 			\draw[very thin, blue!20,step=0.5] (-3,0) grid (3,3);
 			\draw[thin,blue!30] (-3,0) grid (3,3);
 			
 			\draw (-3,0)--(3,0);
 			\draw[dotted] (-4,0)--(-3,0);
 			\draw[dotted,->] (3,0)--(4,0) node[below=4pt,right]{\tiny $\mathds{L}_H$};
 			\draw (0,0)--(0,3);
 			\draw[dotted,->] (0,3)--(0,4) node[left=4pt]{\tiny $\log_2\circ\mathds{F}$};
 			
 			\foreach \x in {0,1,2,3}
 			\draw (3,\x) node[right=4pt]{\tiny \x};
 			
 			\foreach \x in {-3,-2,...,3}
 			{
 				\fill[blue!40!white] (\x,0) circle [radius=2.5pt];
 				\draw (\x,0) node[below=4pt]{\tiny \x};
 			}
 			
 			\foreach \x in {-2.5,-1.5,...,2.5}
 			{
 				\fill[blue!60!white] (\x,1) circle [radius=2pt];
 			}
 			
 			\foreach \x in {-2.75,-1.75,...,2.75}
 			{
 				\fill[blue] (\x,2) circle [radius=1pt];
 				\fill[blue] (\x+0.5,2) circle [radius=1pt];
 			}
 			
 			\foreach \x in {-2.875,-1.875,...,2.875}
 			{
 				\fill (\x,3) circle [radius=0.5pt];
 				\fill (\x+0.25,3) circle [radius=0.5pt];
 				\fill (\x+0.5,3) circle [radius=0.5pt];
 				\fill (\x+0.75,3) circle [radius=0.5pt];
 			}
 			\end{tikzpicture}
 		}
 		
 		\caption{The geometry of $\Z\left[\frac{1}{2}\right]$}\label{dyadic-Q-pic}
 	\end{figure}
 	
 \end{example}

  Now, for any function $f : G_n \rightarrow \C$, we denote by $M_f$ the operator of multiplication by $f$ on the subspace:
  \begin{equation*}
    \dom{M_f} \coloneqq \left\{ \xi \in \ell^2(G_n) : (h\in G_n \mapsto f(h)\xi(h)) \in \ell^2(G_n) \right\}
  \end{equation*}
  of $\ell^2(G_n)$. Of course, $M_f$ is bounded by $\norm{f}{C(G_n)}$ if $f$ is bounded, and unbounded otherwise; nonetheless $\dom{M_f}$ always contains $C_c(G_n)$ and thus is always dense in $\ell^2(G_n)$.

  Let $E$ be a finite dimensional Hilbert space with inner product $\inner{\cdot}{\cdot}{E}$ and $\dim E \in 2\N\setminus\{0\}$, and let $\mathfrak{c}$ be a $\ast$-representation of the Clifford algebra of $\C^2$ on $E$. Let $\gamma_1 = \mathfrak{c}\left(\begin{pmatrix} 1 \\ 0 \end{pmatrix}\right)$ and $\gamma_2 = \mathfrak{c}\left(\begin{pmatrix} 0 \\ 1 \end{pmatrix}\right)$. For our purpose, we record that  for all $j,k \in \{1,2\}$:
\begin{equation*}
  \quad \gamma_j\gamma_k + \gamma_k \gamma_j = \begin{cases} 2 \text{ if $j=k$,} \\ 0 \text{ otherwise.}\end{cases} \text.
\end{equation*}

\begin{remark}
  There is no particular reason to restrict ourselves  to $E = \C^2$, though it is the natural choice. In this case, we can choose the usual Weyl matrices:
  \begin{equation*}
    \gamma_1 = \begin{pmatrix} 1 & 0 \\ 0 & -1 \end{pmatrix} \text{ and }\gamma_2 = \begin{pmatrix} 0 & 1 \\ 1 & 0 \end{pmatrix}
  \end{equation*}
  as the most natural choice for our construction.
\end{remark}

For each $n\in\Nbar \coloneqq \N\cup\{\infty\}$, we identify the Hilbert space $\ell^2(G_n,E)$ of $E$-valued functions over $G_n$ (with inner product $\inner{\xi}{\eta}{\ell^2(G_n,E)} \coloneqq \sum_{g \in G_n} \inner{\xi(g)}{\eta(g)}{E}$) with $\ell^2(G_n)\otimes E$. We then let
  \begin{equation*}
    \dom{\Dirac_n}\coloneqq \left\{ \xi \in \ell^2(G_n,E) : (\mathds{L}_H(g)\gamma_1\xi(g)+\mathds{F}(g)\gamma_2\xi(g))_{g\in G_n} \in \ell^2(G_n,E) \right\}
  \end{equation*}
  and on $\dom{\Dirac_n}$, we define the Dirac operator:
  \begin{equation}\label{eq:def-Dirac}
  \Dirac_n \coloneqq M_{\mathds{L}_H} \otimes \gamma_1 + M_{\mathds{F}} \otimes \gamma_2\text.
\end{equation}

\bigskip

We now prove that $(C^\ast(G_n,\sigma),\ell^2(G_n)\otimes E,\Dirac_n)$, as defined above, are  indeed spectral triples, for all $n\in\Nbar$. A first step is the computation of the domain of our Dirac operators of Equation \eqref{eq:def-Dirac}. To do so, we will need the following lemma. Recall that a norm $\norm{\cdot}{\R^2}$ on $\R^2$ is \emph{monotone} when it is increasing with respect to the product order on $\R^2$; the most important such norm for our purpose will be the max norm $x,y \in \R^2\mapsto\norm{(x, y)}{\infty} = \max\{|x|,|y|\}$; we also note that we will often write elements of $\R^d$ as simple $d$-tuples.

\medskip

\begin{lemma}\label{length-lemma}
  With the notation and   assumptions of this section,  the following identities hold. 
  
  \begin{enumerate}
 \item  For all $ g \in G_\infty $: 
  \begin{equation*}
    \quad \mathds{F}(g^{-1}) = \mathds{F}(g) \text;
  \end{equation*}
\item   For all $g,h \in G_\infty$: 
  \begin{equation*}
     \mathds{F}(gh) \leq \max\left\{ \mathds{F}(g) , \mathds{F}(h)\right\} \leq \mathds{F}(g) + \mathds{F}(h) \text.
  \end{equation*}
   \end{enumerate}
  Moreover, if $\norm{\cdot}{\R^2}$ is any monotone norm on $\R^2$, then $g \in G_\infty \mapsto \norm{(\mathds{L}_H(g), \mathds{F}(g))}{\R^2}$ is a proper, unbounded length function over $G_\infty$.
\end{lemma}

\begin{proof}
  Let $g \in G_\infty$, and let $n\in\N$ be the unique natural number such that $\mathds{F}(g) = \mathrm{scale}(n)$, or $n=0$ if $\mathds{F}(g)<\mathrm{scale}(0)$. If $n=0$ then $\mathds{F}(g) = \mathds{F}(g^{-1})$ by assumption. If $n > 0$, then $g \in G_n$ and $g \notin G_p$ for $p<n$; therefore, $g^{-1} \in G_n$ and $g^{-1} \notin G_p$ if $p<n$; hence, $\mathds{F}(g^{-1}) = \mathrm{scale}(n) = \mathds{F}(g)$. 

  Now, take $h \in G_\infty$. Again, let $m \in \N$ be uniquely defined by $\mathds{F}(h) = \mathrm{scale}(m)$ or $m = 0$ otherwise. Let $k = \max\{m,n\}$. Thus $g, h \in G_k$ and therefore, $g h \in G_k$. First, if $g,h,gh \in G_0$, then $\mathds{F}(gh)\leq\max\{\mathds{F}(g),\mathds{F}(h)\}$ by assumption on $\mathds{F}$. Otherwise, $k > 0$, and we simply observe that either $g h \in G_0$ and then $\mathds{F}(gh)\leq \mathrm{scale}(0) < \mathrm{scale}(k)$, or $gh \notin G_0$, and again $\mathds{F}(gh)\leq\mathrm{scale}(k)$; either way we observe:
  \begin{equation*}
    \mathds{F}(gh) \leq \mathrm{scale}(k) = \mathrm{scale}(\max\{n,m\}) = \max\{\mathrm{scale}(n),\mathrm{scale}(m)\} = \max\{\mathds{F}(g),\mathds{F}(h)\} \text.
  \end{equation*}
  Fix a monotone norm $\norm{\cdot}{\R^2}$ on $\R^2$ and let
  \begin{equation*}
    \mathds{L} : g\in G_\infty \longmapsto \norm{(\mathds{L}_H(g), \mathds{F}(g))}{\R^2} \text.
  \end{equation*}
  It is then immediate to check that if $g,h \in G_\infty$, then, since $\norm{\cdot}{\R^2}$ is monotone:
  \begin{multline*}
    \norm{(\mathds{L}_H(g h),\mathds{F}(g h)) }{\R^2} \leq \norm{(\mathds{L}_H(g) + \mathds{L}_H(h) , \mathds{F}(g) + \mathds{F}(h))}{\R^2} \\ \leq \norm{(\mathds{L}_H(g), \mathds{F}(g) )}{\R^2} + \norm{( \mathds{L}_H(h), \mathds{F}(h) )}{\R^2} \text.
  \end{multline*}
  Moreover $\norm{( \mathds{L}_H(g^{-1}) , \mathds{F}(g^{-1}) )}{\R^2} = \norm{( \mathds{L}_H(g) , \mathds{F}(g) )}{\R^2}$ for all $g\in G_\infty$.

  Finally, if $\norm{( \mathds{L}_H(g), \mathds{F}(g) )}{\R^2} = 0$, then $\mathds{L}_H(g) = 0$, which in turns implies $g = 1$. On the other hand, $\mathds{F}(1) = 0$ and $\mathds{L}_H(1) = 0$, so $\mathds{L}(1) = 0$. Thus as claimed, $\mathds{L}$ is a length function on $G_\infty$.

  \medskip

  Now, let be more specific in our choice of $\norm{\cdot}{\R^2}$, and fix it to be the usual max norm $\norm{\cdot}{\infty}$; we then rename our length $\mathds{L}_\infty$; so
  \begin{equation*}
    \mathds{L}_\infty(g) \coloneqq \max\left\{\mathds{L}_H(g),\mathds{F}(g) \right\} \text.
  \end{equation*}
 Fix  $n \in \N$. By definition, the following equality between closed balls hold:
  \begin{equation*}
    \left\{ g \in G_\infty : \mathds{L}(g) \leq \mathrm{scale}(n)\right\} = \left\{ g \in G_n : \mathds{L}_H(g) \leq \mathrm{scale}(n) \right\} \text.
  \end{equation*}
Since $\mathds{L}_H$ is proper on $G_n$, this set is finite. So $\mathds{L}$ is indeed proper on $G_\infty$.

 By assumption, the function $\mathrm{scale}$ is unbounded on $\N$ and, for all $n\in \N$, there exists $g \in G_\infty\setminus G_n$ (since $(G_n)_{n\in\N}$ is assumed to be strictly increasing), i.e. $\mathds{F}(g)\geq \mathrm{scale}(n)$, so $\mathds{L}$ is unbounded.

  \medskip

  We now return to a general monotone norm $\norm{\cdot}{\R^2}$ on $\R^2$. Since all norms on $\R^2$ are equivalent, there exists $c > 0$ such that $\frac{1}{c}\norm{\cdot}{\infty}\leq\norm{\cdot}{\R^2} \leq c \norm{\cdot}{\infty}$. Therefore,
  \begin{equation*}
    \frac{1}{c} \mathds{L}_\infty \leq \mathds{L} \leq c \mathds{L}_\infty \text.
  \end{equation*}
  It is now easy to check that $\mathds{L}$ is again proper and unbounded on $G_\infty$. This concludes our proof.
\end{proof}

\begin{remark}
  It is quite natural to simply set $\mathds{F}(g) = \mathrm{scale}(0)$ for all $g \in G_0\setminus\{1\}$. The difference between such a choice of $\mathds{F}$, vs any other $\mathds{F}'$, which meets our assumptions over $G_0$, is a bounded perturbation. We refer to \cite{Latremoliere15b} for a discussion on bounded perturbations of spectral triples from the metric perspective.
\end{remark}

As seen in the above discussion, the above length function $\mathds{L}_H$  will not be proper, so it won't define a spectral triple by itself, however $\mathds{L}$ is proper, and thus can be used to define a spectral triple on $C^\ast(G_\infty,\sigma)$. However, we take a slightly different route by working with what we shall prove is an even spectral triple, replacing the linear geometry of $G_\infty$  with a sort of ``two-dimensional'' geometry (see Figure (\ref{dyadic-Q-pic}) for the noncommutative solenoid case). 
\medskip

We now prove that in the above hypotheses we can indeed define spectral triples. We begin with a computation of the domain of the proposed Dirac operators defined in   Equation \eqref{eq:def-Dirac}.

\begin{lemma}\label{gamma-norm-lemma}
 With the notation and   assumptions of this section, the following assertion holds; for all $\xi \in E$ and for all  $ a,b \in \R $:
  \begin{equation*}
     \quad \norm{(a\gamma_1+b\gamma_2)\xi}{E}^2 = \left(a^2+b^2\right)\norm{\xi}{E}^2 \text.
  \end{equation*}
In particular, for all $n\in\Nbar$, the domain $\dom{\Dirac_n}$ of the Dirac operator ${\Dirac_n}$ is given by
  \begin{equation*}
    \left\{ \xi \in \ell^2(G_n,E) : \sum_{g \in G_n} (\mathds{L}_H(g)^2 + \mathds{F}(g)^2)\norm{\xi(g)}{E}^2 < \infty \right\} \text.
  \end{equation*}
\end{lemma}

\begin{proof}
  Let $\xi \in E$. The following identity holds for all $a,b\in\R$:
  \begin{align*}
    \norm{a \gamma_1\xi + b\gamma_2\xi}{E}^2
    &= a^2 \inner{\gamma_1\xi}{\gamma_1\xi}{E} + ab\inner{\gamma_1\xi}{\gamma_2\xi}{E} + ab\inner{\gamma_2\xi}{\gamma_1\xi}{E} + b^2 \inner{\gamma_2\xi}{\gamma_2\xi}{E} \\
    &= a^2\inner{\gamma_1^2\xi}{\xi}{E} + ab\inner{(\gamma_1\gamma_2+\gamma_2\gamma_1)\xi}{\xi}{E} + b^2\inner{\gamma_2^2\xi}{\xi}{E} \\
    &= (a^2+b^2)\norm{\xi}{E}^2 \text.
  \end{align*}
  as claimed.

  The computation of the $\dom{\Dirac_n}$, for all $n\in\N$, follows immediately.
\end{proof}

We now prove that our Dirac operators are indeed self-adjoint with compact resolvent, and that they can be used to define spectral triples. We also establish some useful estimates which will later allow us to prove that our construction gives metric spectral triples over noncommutative solenoids.

\begin{definition}\label{G-notation}
  If $a$ is a bounded operator on $\ell^2(G_\infty)$, we denote by $a^\circ$ the operator $a\otimes 1_E$ acting on $\ell^2(G_\infty,E)$. We also define the representation $\lambda$ of $C^\ast(G_\infty,\sigma)$ on $\ell^2(G_\infty,E)$ by setting $\lambda_E \coloneqq \lambda\otimes 1_E$, so for all $f \in C_c(G_\infty)$, we have $\lambda(f)^\circ \coloneqq \lambda_E(f)$. Moreover
  \begin{enumerate}
\item For each $n\in\Nbar$, define:
\begin{equation*}
  \dom{\Lip_n} \coloneqq \left\{ a \in \sa{C^\ast_{\mathrm{red}}(G_n,\sigma)} : a^\circ \dom{\Dirac_n}\subseteq\dom{\Dirac_n} \text{ and }[\Dirac_n,a^\circ]\text{ is bounded} \right\}.
\end{equation*}
\item For all $a \in \dom{\Lip_n}$ define: 
\begin{equation*}
 \Lip_n(a) \coloneqq \opnorm{[\Dirac_n,a^\circ]}{}{\ell^2(G_n,E)} \text.
\end{equation*}
\end{enumerate}
\end{definition}

We conclude this subsection by proving that we indeed defined even spectral triples, and lay the groundwork for our third main theorem in the next subsection. Recall that, by Lemma
(\ref{length-lemma}), $\mathds{L}_H + \mathds{F}$ is proper and unbounded on $G_\infty$.

\begin{lemma}\label{G-lemma}
   With the notation and assumptions  of this section,  for any fixed  $n\in\Nbar$, the ordered triple \[(C^\ast_{\mathrm{red}}(G_n,\sigma),\ell^2(G_n,E),\Dirac_n)\] 
   is an even spectral triple, where the grading on $\ell^2(G_n,E)$ is given by $1_{\ell^2(G_n)}\otimes i \gamma_1\gamma_2$. 
   Moreover, for all $a\in \dom{\Lip_n}$:
  \begin{equation*}
    \opnorm{[M_{\mathds{L}_H},a]}{}{\ell^2(G_n)} \leq \opnorm{[\Dirac_n,a^\circ]}{}{\ell^2(G_n,E)},
  \end{equation*}
  together with:
  \begin{equation*}
    \opnorm{[M_{\mathds{F}},a]}{}{\ell^2(G_n)} \leq \opnorm{[\Dirac_n,a^\circ]}{}{\ell^2(G_n,E)}\text.
  \end{equation*}
In particular, if we define $\mathds{L}\coloneqq \mathds{L}_H + \mathds{F}$, then for all $n \in \Nbar$ and  all $a \in \dom{\Lip_n}$: 
  \begin{equation*}
    \opnorm{[M_{\mathds{L}},a]}{}{\ell^2(G_n)} \leq 2\opnorm{[\Dirac_n,a^\circ]}{}{\ell^2(G_n,E)}\text.
  \end{equation*}
If, for any $n\in\Nbar$, the spectral triple $(C^\ast_{\mathrm{red}}(G_n,\sigma),\ell^2(G_n),\Dirac_{\mathds{L}})$ is metric, then so is $(C^\ast_{\mathrm{red}}(G_n,\sigma),\allowbreak \ell^2(G_n,E),\Dirac_n)$.
\end{lemma}

\begin{proof}
  We will start by showing that, for any fixed  $n \in \Nbar$,  $\Dirac_n$ is self-adjoint with compact resolvent.  Fixed  any $n\in \Nbar$, note that the domain of $\Dirac_n$ contains all finitely supported functions in $\ell^2(G_n,E)$ and is therefore dense. Moreover, since $\gamma_1$ and $\gamma_2$ are self-adjoint, if $\xi,\eta\in\dom{\Dirac_n}$, it follows that:
  \begin{align*}
    \inner{\Dirac_n\xi}{\eta}{\ell^2(G_n,E)}
    &=\sum_{g\in G_n} \inner{\left(\mathds{L}_H(g)\gamma_1+\mathds{F}(g)\gamma_2\right)\xi}{\eta}{E} \\
    &=\sum_{g\in G_n} \inner{\xi}{\left(\mathds{L}_H(g)\gamma_1+\mathds{F}(g)\gamma_2\right)\eta}{E} \\
    &= \inner{\xi}{\Dirac_n\eta}{\ell^2(G_n,E)},
  \end{align*}
  so $\Dirac_n$ is also a symmetric operator.
By using Lemma (\ref{gamma-norm-lemma}),  we now note that:
  \begin{equation*}
    \dom{\Dirac_n^2} = \left\{ \xi \in \ell^2(G_n,E) : \sum_{g \in G_n} \left(\mathds{L}_H(g)^2 + \mathds{F}(g)^2\right)^2 \norm{\xi(g)}{E}^2 < \infty \right\}
  \end{equation*}
  and, over $\dom{\Dirac_n^2}$, the Clifford algebra relations imply:
  \begin{equation*}
    \Dirac_n^2 + 1 = \left(M_{\mathds{L}_H}^2 + M_{\mathds{F}}^2 + 1 \right)\otimes 1_E \text.
  \end{equation*}
Now  define an operator $K$ on $\ell^2(G_n,E)$ by setting, for all $\xi \in \ell^2(G_n,E)$:
  \begin{equation*}
    K \xi : g \in G_n \mapsto \frac{1}{\sqrt{\mathds{L}_H(g)^2 + \mathds{F}(g)^2 + 1}} \xi(g)\text.
  \end{equation*}
  By construction, $K$ is positive. Moreover, if $n\in\N$, then $\mathds{L}_H$ restricted to $G_n$ is proper and $\mathds{F}$ is bounded over $G_n$ by our hypotheses, so $K$ is compact.
  If $n=\infty$, by our hypotheses, for all $r \geq 0$, the set $\{ g \in G_\infty : \mathds{F}(g) \leq r \}$ is a subset of $G_k$ for some $k\in\N$. Since $\mathds{L}_H$ is proper on $G_k$, the set $\{ g \in G_\infty : \mathds{L}_H ^2(g) + \mathds{F}^2(g) \leq r \}$ is finite. Thus, the eigenspaces of $K$ are all finite dimensional. It follows easily that $K$ is compact, as well.

  In any case, i.e., for all $n \in \Nbar$,   $(\Dirac_n^2+1)K^2\xi = \xi$ for all $\xi \in \ell^2(G_n,E)$, while $K^2(\Dirac_n^2+1)\xi = \xi$ for all $\xi \in \dom{\Dirac_n^2}$, as seen by a direct computation; in particular, we note that $K\ell^2(G_n,E) = \dom{\Dirac_n}$ by construction.

  By Lemma (\ref{gamma-norm-lemma}), for all $\xi \in \ell^2(G_n,E)$, we obtain:
  \begin{align*}
    \sum_{g\in G_n} &\norm{\Dirac_n K\xi(g)}{E}^2 \\
    &= \sum_{g\in G_n} \norm{\frac{\mathds{L}_H(g)}{\sqrt{\mathds{L}_H(g)^2+\mathds{F}(g)^2+1}} (\gamma_1\xi(g)) + \frac{\mathds{F}(g)}{\sqrt{\mathds{L}_H(g)^2+\mathds{F}(g)^2+1}} (\gamma_2\xi(g)) }{E}^2 \\
    &= \sum_{g\in G_n} \frac{\mathds{L}_H(g)^2 + \mathds{F}(g)^2}{\mathds{L}_H(g)^2 + \mathds{F}(g)^2 + 1} \norm{\xi(g)}{E}^2 \\
    &\leq \norm{\xi}{\ell^2(G_n,E)}^2 \text.
  \end{align*}
  Thus, $\Dirac_n K$ is bounded, of norm at most $1$. Consequently, $(\Dirac_n \pm i)K$ is also bounded on $\ell^2(G_n,E)$.
Therefore, $(\Dirac \pm i)K^2$ is compact. It follows that  $\Dirac \pm i$ both have compact inverse $(\Dirac \mp i) K^2$. Specifically for our purpose, if $\xi \in \ell^2(G_n,E)$, then:
  \begin{equation*}
    (\Dirac_n+i) \left( (\Dirac_n-i)K^2\right)\xi = (\Dirac_n^2+1)K^2\xi = \xi \text.
  \end{equation*}
  Therefore, the range of $\Dirac_n+i$ is $\ell^2(G_n,E)$. Similarly, the range of $\Dirac_n-i$ is also $\ell^2(G_n,E)$. As $\Dirac_n$ is also a symmetric operator defined on a dense domain, we conclude by \cite[Sec. VIII.2]{ReedSimon} and \cite[Lemma 2.48]{Latremoliere21a} that $\Dirac_n$ is indeed self-adjoint, with compact resolvent (since the inverse of $\Dirac_n + i$ is the compact $(\Dirac_n-i)K^2$). 

\medskip

 We will now verify the commutator spectral triples condition. Note that if $g \in G_n$, then
  \begin{equation*}
    \opnorm{[\Dirac_n,\lambda_E(g)]}{}{\ell^2(G_n,E)}\leq \mathds{L}_H(g) + \mathds{F}(g) = \mathds{L}(g) \text.  
  \end{equation*}
  Therefore, if $f \in C_c(G_n)$, then the operator $[\Dirac_n,\lambda_E(f)]$ is bounded, and in fact,
  \begin{equation*}
    \opnorm{[\Dirac_n,\lambda_E(f)]}{}{\ell^2(G_n,E)} \leq \sum_{g \in G_n} |f(g)| \mathds{L}(g) \text.
  \end{equation*}
We conclude that  $(C^\ast_{\mathrm{red}}(G_\infty,\sigma),\ell^2(G_\infty),\Dirac)$ is a spectral triple for all $n\in\Nbar$.

  \medskip

 We will now prove that our spectral triple is metric. Let $a \in \dom{\Lip_n}$ for some $n\in\Nbar$. We then note that,
  \begin{equation*}
    (1\otimes\gamma_1)[\Dirac_n,a^\circ] + [\Dirac_n,a^\circ](1\otimes \gamma_1) = [M_{\mathds{L}_H},a]\otimes 2,
  \end{equation*}
  which implies:
  \begin{align*}
    \opnorm{[M_{\mathds{L}_H},a]}{}{\ell^2(G_n)}
    &\leq \frac{1}{2} \opnorm{(1\otimes\gamma_1)[\Dirac_n,a^\circ] + [\Dirac_n,a^\circ](1\otimes \gamma_1)}{}{\ell^2(G_n,E)} \\
    &\leq \frac{1}{2}\left( \opnorm{(1\otimes\gamma_1)[\Dirac_n,a^\circ]}{}{\ell^2(G_n,E)} + \opnorm{[\Dirac_n,a^\circ](1\otimes \gamma_1)}{}{\ell^2(G_n,E)}\right) \\
    &\leq \frac{1}{2} \left(\opnorm{[\Dirac_n,a^\circ]}{}{\ell^2(G_n,E)} + \opnorm{[\Dirac_n,a^\circ]}{}{\ell^2(G_n,E)}\right) \\
    &=\opnorm{[\Dirac_n,a^\circ]}{}{\ell^2(G_n,E)} \text.
  \end{align*}
The same reasoning, with $1\otimes\gamma_2$ in place of $1\otimes \gamma_1$, leads to
  \begin{equation*}
    \opnorm{[M_{\mathds{F}},a]}{}{\ell^2(G_n)} \leq \opnorm{[\Dirac_n,a^\circ]}{}{\ell^2(G_n,E)} \text.
  \end{equation*}
 Therefore, for all $a\in\dom{\Lip_n}$, we obtain:
  \begin{equation*}
    \opnorm{[M_{\mathds{L}},a]}{}{\ell^2(G_n)} \leq \opnorm{[M_{\mathds{L}_F},a]}{}{\ell^2(G_n)} + \opnorm{[M_\mathds{F},a]}{}{\ell^2(G_n)} \leq 2 \opnorm{[\Dirac_n,a^\circ]}{}{\ell^2(G_n,E)} \text.
  \end{equation*}

  In particular, if $(C^\ast_{\mathrm{red}}(G_n,\sigma),\ell^2(G_n),\Dirac_{\mathds{L}})$ is a metric spectral triple, then, by \cite[Lemma 1.10]{Rieffel98a}, so is $(C^\ast_{\mathrm{red}}(G_n,\sigma),\ell^2(G_n),\Dirac_n)$.

  Finally, we will show that our spectral triples  are in fact even, with   grading given by $1_{\ell^2(G_n)}\otimes \gamma$ where $\gamma \coloneqq i \gamma_1\gamma_2$. By construction, $\gamma^2$ is the identity, and $\gamma^\ast=\gamma$, so $\gamma$ is a self-adjoint unitary; therefore so is $1_{\ell^2(G_n)}\otimes\gamma$, which splits $\ell^2(G_n,E)$ in its two spectral subspaces for $1$ and $-1$, in such a way that $\lambda_E$ commutes with $1\otimes\gamma$, while $\Dirac_n(1\otimes\gamma) = -(1\otimes\gamma)\Dirac_n$. So $(C^\ast_{\mathrm{red}}(G_n,\sigma),\ell^2(G_n,E),\Dirac_n)$ is an even spectral triple. 
\end{proof}

\begin{remark}
  With the notation of Lemma (\ref{G-lemma}), we note that for each finite $n\in\N$, the spectral triple $(C^\ast_{\mathrm{red}}(G_n),\ell^2(G_n,E),\Dirac_n)$ is, in some sense, a bounded perturbation of the odd spectral triple $(C^\ast_{\mathrm{red}}(G_n),\ell^2(G_n),M_{\mathds{L}})$, since $\mathds{F}$ is bounded on $G_n$. The situation is quite different when $n=\infty$, of course.
\end{remark}

\begin{remark}
  Suppose $\rho$ is some other $2$-cocycle of $G_\infty$, which is equivalent to $\sigma$, i.e., for some function $f : G_\infty\rightarrow\T$, the following holds:
  \begin{equation*}
    \forall g,h \in G_\infty \quad \rho(g,h) = f(g)f(h)\overline{f(gh)} \sigma(g,h) \text.
  \end{equation*}
  The operator $M_f$ is then a unitary which intertwines the left regular $\sigma$ and $\rho$ projective representation of $G_\infty$. Thus, $(\mathrm{Ad}M_f)^\circ$ implements a *-isomorphism from $\lambda_E(C^\ast(G_\infty,\rho))$ onto $\lambda_E(C^\ast(G_\infty,\sigma))$. Furthermore, $M_{f}^\circ$ commutes with $\Dirac_\infty$. Therefore,  the spectral triples $(C^\ast(G_\infty,\sigma),\ell^2(G_\infty,E),\Dirac_\infty)$ and $(C^\ast(G_\infty,\rho),\ell^2(G_\infty,E),\Dirac_\infty)$ are unitarily equivalent. In particular, whenever one is metric, so is the other, and then they are at distance zero from each others for the spectral propinquity. 
\end{remark}

\subsection{Main result}

We begin this section by making some basic identifications that will be used throughout the rest of the paper. We will use the notation introduced  in the above sections. Fixed $n \in \N$, the C*-algebra $C_{\mathrm{red}}^\ast(G_n,\sigma)$ is technically the closure, in the operator norm, of the linear span of the operators $\lambda_{n}(g)$ defined on $\ell^2(G_n)$ by $\lambda_{n}(g)\xi : h \in \ell^2(G_n) \mapsto \sigma(g,g^{-1}h)\xi(g^{-1}h)$. On the other hand, since $G_n\subseteq G_\infty$, we obtain a different unitary $\sigma$-projective representation of $G_n$, via the restriction of the $\sigma$-projective representation $\lambda$ of $G_\infty$ to $G_n$ on $\ell^2(G_\infty)$, giving us an alternative C*-algebra generated by $\{\lambda(h) : h \in G_n \}$.
If $S\subseteq G_\infty$ is any nonempty subset of $G_\infty$, we identify the space  $\ell^2(S)$ with
\begin{equation*}
 	 \{ \xi \in \ell^2(G_\infty) : \forall g \in G_\infty\setminus S\quad \xi(g) = 0 \}\text.
\end{equation*}
Let $Q_n\subseteq G_\infty$ be a subset of $G_\infty$ such that every right coset of $G_n$ in $G_\infty$ is of the form $G_n k$ for a unique $k \in Q_n$. Of course,
\begin{equation*}
  \ell^2(G_\infty) = \overline{\bigoplus}_{k \in Q_n} \ell^2(G_n k) \text,
\end{equation*}
where $\overline{\oplus}$ is the Hilbert sum, i.e. the closure of the direct sum.

Now, we set, for all $k\in G_\infty$ and $\xi \in \ell^2(G_\infty)$:
\begin{equation}\label{right-rep-eq}
  \rho(k) \xi : h \in \ell^2(G_\infty) \mapsto \sigma(h k,k^{-1}) \xi(h k)\text.
\end{equation}
Thus defined, $\rho$ is the right regular $\breve{\sigma}$-projective representation of $G_\infty$ on $\ell^2(G_\infty)$, where $\breve{\sigma}:g,h\in G_\infty\mapsto\sigma(h^{-1},g^{-1})$ is indeed a $2$-cocycle of $G_\infty$.

\begin{remark}
  If the $2$-cocycle $\sigma$ is normalized, i.e. $\sigma(g,g^{-1})=1$ for all $g\in G_\infty$, then $\breve{\sigma}=\overline{\sigma}$; we will however not need to work with normalized cocycles here.
\end{remark}

Since $\sigma$ is a $2$-cocycle, we obtain, for all $g,h,k \in G_\infty$ and $\xi \in \ell^2(G_\infty)$:
\begin{align*}
  \lambda(g)\rho(k)\xi(h)
  &= \sigma(g,g^{-1}h)\rho(k)\xi(g^{-1}h) \\
  &= \sigma(g,g^{-1}h)\sigma(g^{-1}hk,k^{-1})\xi(g^{-1}hk) \\
  &= \sigma(\underbracket[1pt]{g}_{\eqqcolon x},\underbracket[1pt]{g^{-1}h k}_{\eqqcolon y} \; \underbracket[1pt]{k^{-1}}_{\eqqcolon z})\sigma(\underbracket[1pt]{g^{-1}h k}_{=y},\underbracket[1pt]{k^{-1}}_{= z})\xi(g^{-1}hk) \\
  &= \sigma(\underbracket[1pt]{g}_{=x}, \underbracket[1pt]{g^{-1}h k}_{=y})\sigma(\underbracket[1pt]{h k}_{=xy},\underbracket[1pt]{k^{-1}}_{=z}) \xi(g^{-1}h k) \\
  &= \sigma(hk,k^{-1}) (\lambda(g)\xi) (hk) \\
  &= \rho(k)\lambda(g)\xi(h) \text.
\end{align*}
Therefore, $\lambda(g)$ and $\rho(k)$ commute, for all $g,k \in G_\infty$. 
It is moreover immediate that $\rho(k^{-1})$ maps $\ell^2(G_n k)$ onto $\ell^2(G_n)$. 

We now define the unitary $V$ from $\ell^2(G_\infty) = \overline{\bigoplus}_{k \in Q_n}\ell^2(G_n k)$ to $\overline{\bigoplus}_{k \in Q_n}\ell^2(G_n)$ by setting, for all $\xi = (\xi_k)_{k \in Q_n} \in \overline{\bigoplus}_{k\in Q_n}\ell^2(G_n k)$:
\begin{equation*}
  V\xi = \left(\rho(k^{-1})\xi_k\right)_{k\in Q_n} \; \in \; \overline{\bigoplus}_{k \in Q_n} \ell^2(G_n) \text.
\end{equation*}
By construction, $V$ is unitary, and moreover, for any $g \in G_n$:
\begin{equation*}
  V \lambda(g) V^\ast (\xi_k)_{k \in Q_n} = (\lambda_{n}(g)\xi_k)_{k\in Q_n} \text.
\end{equation*}
Thus, $\mathrm{Ad} V$ is a $\ast$-isomorphism from the C*-subalgebra generated by $\{\lambda(g) : g \in G_n \}$ and the C*-algebra $C^\ast_{\mathrm{red}}(G_n,\sigma)$ which maps $\lambda(g)$ to $\lambda_{n}(g)$ for all $g \in G_n$.

From now on, we thus identify $C_{\mathrm{red}}^\ast(G_n,\sigma)$ with the C*-algebra generated by $\{\lambda(g):g\in G_n\}$ in $C_{\mathrm{red}}^\ast(G_\infty,\sigma)$ and work exclusively in the latter. We will also identify $\ell^2(G_n,E)$ with $\left\{\xi\in\ell^2(G_\infty,E) : \forall g \in G_\infty\setminus G_n \quad \xi(g) = 0 \right\}$.

To complete our picture, we also identify $\Dirac_n$ with the operator defined for $\xi = \xi_n + \xi_n^\perp$, with $\xi_n \in \ell^2(G_n,E)^2$ and $\xi_n^\perp \in \ell^2(G_n.E) = \overline{\bigoplus}_{k\in Q_n\setminus G_n} \ell^2(G_n k,E)$, by $\Dirac_n \xi = \Dirac_n \xi_n \in \ell^2(G_n,E)$. We then observe that if $P_n$ is the orthogonal projection from $\ell^2(G_\infty)$ onto $\ell^2(G_n)$, we have, for all $\xi \in \dom{\Dirac_n}$ and for all $g \in G_\infty$:
\begin{align*}
  P_n^\circ \Dirac_n \xi(g)
  &= \begin{cases}
    (\mathds{L}_H(g)\otimes\gamma_1+\mathds{F}(g)\otimes\gamma_2)\xi(g) \text{ if $g \in G_n$, }\\
    0 \text{ otherwise, }
  \end{cases}\\
  &= \Dirac_\infty P_n^\circ \xi (g) \text.
\end{align*}
We thus have shown that $P_n^\circ\dom{\Dirac_n}\subseteq \dom{\Dirac_\infty}$ and $P_n^\circ \Dirac_n = \Dirac_\infty P_n^\circ$. Moreover, $P_n^\circ\Dirac_\infty P_n^\circ= \Dirac_n$ and thus, for all $a\in \dom{\Lip_n}$ we compute the following expression, using the fact that $[P_n,a]=0$,:
\begin{align*}
  P_n^\circ[\Dirac_\infty,a^\circ]P_n^\circ
  &= P_n^\circ \Dirac_\infty a^\circ P_n^\circ - P_n^\circ a^\circ \Dirac_\infty P_n^\circ \\
  &= P_n^\circ \Dirac_\infty P_n^\circ a^\circ - a^\circ P_n^\circ \Dirac_\infty P_n^\circ \\
  &= \Dirac_n a^\circ - a^\circ \Dirac_n.
\end{align*}
So we have, for all $a \in \dom{\Lip_\infty}$:
\begin{equation}\label{Ln-less-Linfty-eq}
  \Lip_n(a) = \opnorm{[\Dirac_n,a^\circ]}{}{\ell^2(G_n,E)} = \opnorm{P_n^\circ [\Dirac_\infty,a^\circ]P_n^\circ}{}{\ell^2(G_\infty,E)} \leq \Lip_\infty(a) \text.
\end{equation}

With all of the above  identifications, we thus have a natural unital $\ast$-morphism from $C^\ast_{\mathrm{red}}(G_n,\sigma)$ into $C^\ast_{\mathrm{red}}(G_\infty,\sigma)$ which becomes just the natural inclusion, and
\begin{equation*}
  \lambda(g)\ell^2(G_n k) \subseteq \ell^2(G_n k)
\end{equation*}
for each $g\in G_n$ and $k \in G_\infty$. By linearity and continuity, we conclude that if $a \in C^\ast_{\mathrm{red}}(G_n,\sigma)$, then $a\ell^2(G_n k) \subseteq \ell^2(G_n k)$ for all $k \in G_\infty$. We also note that $[\Dirac_\infty,a^\circ]\ell^2(G_n k,E)\subseteq \ell^2(G_n k,E)$ for all $k \in G_\infty$ and $a \in \dom{\Lip_n}$.

We will work for the rest of this section with the above identifications and their basic properties without further mention.

\medskip

Our main theorem in this section involves, in particular, a strong result about the convergence of some of the {\qcms s} induced by our spectral triples: namely, we obtain some convergence in the sense of the \emph{Lipschitz distance}.

The \emph{Lipschitz distance} $\LipschitzD$, extended to noncommutative metric geometry in \cite{Latremoliere16b}, is defined between any two {\qcms s} $(\A,\Lip_\A)$ and $(\B,\Lip_\B)$, by
\begin{multline*}
  \LipschitzD((\A,\Lip_\A),(\B,\Lip_\B)) \coloneqq \\
  \inf\left\{ \ln(k) : \exists\pi:(\A,\Lip_\A)\rightarrow(\B,\Lip_\B) \text{ Lipschitz *-isomorphism} \text{ with }\frac{1}{k}\Lip_\A\leq\Lip_\B\circ\pi\leq k \Lip_\A \right\}\text,
\end{multline*}
with the convention that $\inf\emptyset=\infty$. Thus $\LipschitzD$ is finite only between {\qcms s} built over isomorphic C*-algebras. As shown in \cite{Latremoliere16b}, the Lipschitz distance dominates the Gromov-Hausdorff propinquity; in fact, closed balls for the Lipschitz distance are compact in the propinquity.

In particular, if $\A$ is a unital C*-algebra, and if $\Lip_1$ and $\Lip_2$ are two Lipschitz seminorms over $\A$ with the same domain, then the identity is bi-Lipschitz, and we do obtain, by definition:
\begin{equation*}
  \LipschitzD((\A,\Lip_1),(\A,\Lip_2)) \leq \ln(C) \text{ if }\frac{1}{C}\Lip_1\leq\Lip_2\leq C\Lip_1 \text.
\end{equation*}

\medskip

We now prove our main result about inductive limits of discrete groups and the convergence, for the spectral propinquity, of their spectral triples. Note that below we use the notation established in Definition (\ref{G-notation}).

\begin{theorem}\label{main-group-thm} With the notation and assumptions of Subsection \ref{spectral-triples-section}, if  \[(C^\ast_{\mathrm{red}}(G_n,\sigma),\ell^2(G_n,E),\Dirac_n)\] 
  is a \emph{metric} spectral triple for all $n\in\Nbar$,  and if
  \begin{equation*}
    \left\{a\in\dom{\Lip_n}:\Lip_n(a)\leq 1\right\} = \closure{\left\{ a\in C_c(G_n):\Lip_n(a)\leq 1 \right\}} \text,
  \end{equation*}
  then 
  \begin{equation*}
    \lim_{n\rightarrow\infty} \spectralpropinquity{}\left((C^\ast_{\mathrm{red}}(G_n,\sigma),\ell^2(G_n,E), \Dirac_n),(C^\ast_{\mathrm{red}}(G_\infty,\sigma),\ell^2(G_\infty,E),\Dirac_\infty)\right) = 0 \text.
  \end{equation*}

  Moreover, for any fixed $k\in\N$, the sequence $(C^\ast_{\mathrm{red}}(G_k,\sigma),\Lip_n)_{n\geq k}$ converges in the Lipschitz distance $\LipschitzD$ to the {\qcms} $(C^\ast_{\mathrm{red}}(G_k,\sigma),\Lip_\infty)$.
\end{theorem}

\begin{proof}
  We shall check that the identity automorphism of $C^\ast_{\mathrm{red}}(G_\infty,\sigma)$ satisfies the hypothesis of Theorem (\ref{main-spec-thm}).

  Obviously, the identity is a full quantum isometry of $(C^\ast_{\mathrm{red}}(G_\infty,\sigma),\Lip_\infty)$.
  
 Let $C = 2\qdiam{C^\ast(G_\infty,\sigma)}{\Lip_\infty}$ --- note that since $G_\infty \neq\{1\}$, we have $C > 0$. Let $\mathrm{tr} : a \in C^\ast_{\mathrm{red}}(G_\infty,\sigma)\mapsto \inner{a\delta_1}{\delta_1}{\ell^2(G_\infty)}$; $\mathrm{tr}$ is a tracial state of $C^\ast(G_\infty,\sigma)$ which maps $a\in C_c(G_\infty)$ to $a(1)$.

 \medskip
  
  Fix $\varepsilon \in \left(0,\frac{C}{2}\right)$. Since $(C^\ast_{\mathrm{red}}(G_\infty,\sigma),\Lip_\infty)$ is a {\qcms} by assumption, the set $X_\infty \coloneqq \{ a \in \dom{\Lip_\infty} : \Lip_\infty(a)\leq 1, \mathrm{tr}(a) = 0 \}$ is compact. Thus, there exists a finite $\varepsilon$-dense subset $X_\infty^\varepsilon\subseteq X_\infty$. Since $X_\infty = \closure{\{a\in C_c(G_\infty):\Lip_\infty(a)\leq 1, \mathrm{tr}(a) = 0\}}$, we can moreover assume that $X_\infty^\varepsilon\subseteq C_c(G_\infty)$ as well.

Since $X_\infty^\varepsilon$ is finite and each of its element has finite support, there exists a finite subset $S\subseteq G_\infty$ which contains the support of all the elements in $X_\infty^\varepsilon$. Since $G_\infty=\bigcup_{n\in\N} G_n$ and $(G_n)_{n\in\N}$ is increasing, there exists $N_1 \in \N$ such that, for all $n\geq N_1$, we have $S\subseteq G_n$. Thus $X_\infty^\varepsilon \subseteq C_c(G_n)$. Moreover, by Expression (\ref{Ln-less-Linfty-eq}), we also obtain $\Lip_n(a) \leq \Lip_\infty(a)$ for all $a\in X_\infty^\varepsilon$.

  In summary,
  \begin{equation*}
    \forall a \in X_\infty \quad \exists  b \in X_\infty^\varepsilon\subseteq C_c(G_n)\subseteq C^\ast_{\mathrm{red}}(G_n,\sigma): \quad \norm{a-b}{C^\ast_{\mathrm{red}}(G_\infty, \sigma)} < \varepsilon \text{ and }\Lip_n(a) \leq \Lip_\infty(a)\leq 1 \text.
  \end{equation*}
If $a\in\dom{\Lip_\infty}$, then there exists $b \in X_\infty^\varepsilon$ such that $\norm{a-\mathrm{tr}(a) - b}{C^\ast_{\mathrm{red}}(G_\infty,\sigma)} < \varepsilon$. Of course, $b + \mathrm{tr}(a) \in C^\ast_{\mathrm{red}}(G_n,\sigma)$ and $\Lip_n(b+\mathrm{tr}(a)) = \Lip_n(b) \leq 1$.
By homogeneity, it follows that for all $a\in \dom{\Lip_\infty}$, and for all $n\geq N_1$, there exists $b \in \dom{\Lip_n}$ such that $\norm{a-b}{C^\ast_{\mathrm{red}}(G_\infty,\sigma)} < \varepsilon \Lip_\infty(a)$ and $\Lip_n(b)\leq \Lip_\infty(a)$.

  \medskip
  
 Now, using our assumption of Equation (\ref{L-H-eq}), there exists $N_2\in\N$, with $N_2\geq N_1$, such that
  \begin{equation*}
    \Haus{\mathds{L}_H}(G_\infty,G_n) < \frac{\varepsilon}{C^2}\text.
  \end{equation*}

  For each right coset $c$ of $G_n$ in $G_\infty$, let $k \in c$. Since the distance for $\mathds{L}_H$ from $k \in G_\infty$ to $G_n$ is strictly less than $\frac{\varepsilon}{C^2}$, there exists $g \in G_n$ such that $\mathds{L}_H(g^{-1}k) < \frac{\varepsilon}{C^2}$. Setting $k_c = g^{-1}k$, we have by definition of right cosets that $c = G_n k_c$. Therefore, there exists a subset $Q_n\subseteq G_\infty$ of $G_\infty$ such that,  if $k \in Q_n$ then $\mathds{L}_H(k)<\frac{\varepsilon}{C^2}$, and if $c$ is a right coset of $G_n$ in $G_\infty$, then there exists a unique $k \in Q_n$ such that $c = G_n k$.

  \medskip
  
  Let $n\geq N_2$ and let $b \in C_c(G_n)\subseteq C^\ast_{\mathrm{red}}(G_n, \sigma)$ with $b(1) = \mathrm{tr}(b) = 0$. Note that $b\in \dom{\Lip_\infty}\cap\dom{\Lip_n}$ so, in particular, both $\Lip_n(b)$ and $\Lip_\infty(b)$ are \emph{finite}.

 We thus have $\ell^2(G_\infty) = \overline{\oplus}_{k \in Q_n} \ell^2(G_n k)$, where $\overline\oplus$ is the Hilbert sum (the closure of the sum). If $h \in G_n$, then, by definition of a right coset, $\lambda(h) \ell^2(G_n k) \subseteq\ell^2(G_n k)$ for all $k \in Q_n$. As $\Dirac_\infty\left(\dom{\Dirac_n}\right) \subseteq\ell^2(G_n k,E)$ as well for all $k \in Q_n$, we conclude that $[\Dirac_\infty,b^\circ]\left(\ell^2(G_n k,E)\right)\subseteq\ell^2(G_n k,E)$ --- i.e. $b$, $\Dirac_\infty$ and its commutators are all block diagonal in this decomposition of $\ell^2(G_\infty)$. It follows that
  \begin{equation}\label{L-sup-L-eq}
    \opnorm{[\Dirac_\infty,b^\circ]}{}{\ell^2(G_\infty,E)} = \sup_{k \in Q_n} \opnorm{[\Dirac_n,b^\circ]}{}{\ell^2(G_n k,E)},
  \end{equation}
  allowing for any of the above norm to be infinite.
                       
  Now, the restriction of $\Dirac_\infty$ to $\dom{\Dirac_n}$ is exactly $\Dirac_n$, so:
  \begin{equation*}
    \opnorm{[\Dirac_\infty,b^\circ]}{}{\ell^2(G_n,E)} = \opnorm{[\Dirac_n,b^\circ]}{}{\ell^2(G_n,E)} = \Lip_n(b)\text.
  \end{equation*}

  Now, fix $k \in Q_n$ and $k\notin G_n$.  By assumption, and using repeatedly that $(G_p)_{p\in\N}$ is increasing, we observe that $\mathds{F}(gk) = \mathds{F}(k)$ for all $g \in G_n$: Lemma (\ref{length-lemma}) implies that $\mathds{F}(gk)\leq \mathds{F}(k)$ since $\mathds{F}(g)\leq n < \mathds{F}(k)$; on the other hand, if $p \in \{0,\ldots,m-1\}$ where $\mathds{F}(k)=\mathrm{scale}(m)$, noting that $m>0$ since $k\notin G_0$, then $g k \in G_p$ implies $k = g^{-1} g k \in G_n$, which is a contradiction; hence $\mathds{F}(kg) = \mathrm{scale}(m)$, as claimed.

  Therefore, the operator $M_{\mathds{F}}$ is constant on $\ell^2(G_n k)$, and thus $[M_{\mathds{F}},b] = 0$ on $\ell^2(G_n k)$. So,
  \begin{equation*}
    \opnorm{[\Dirac_\infty,b^\circ]}{}{\ell^2(G_n k,E)} = \opnorm{[M_{\mathds{L}_H},b]}{}{\ell^2(G_n k)} \text.
  \end{equation*}

  We will use the $\breve{\sigma}$-projective right representation of $G_\infty$ on $\ell^2(G_\infty)$, as defined in Expression (\ref{right-rep-eq}). By construction, the restriction of $\rho(k)$ to $\ell^2(G_n)$ (which we will keep denoting by $\rho(k)$) is a unitary onto $\ell^2(G_n k)$ (with inverse the restriction to $\ell^2(G_n k)$ of its adjoint $\rho(k)^\ast=\rho(k^{-1})$). Therefore,
  \begin{equation}\label{eq:comm}
    \opnorm{[M_{\mathds{L}_H},b]}{}{\ell^2(G_n k)} = \opnorm{[M_{\mathds{L}_H},b]\rho(k)}{\ell^2(G_n k)}{\ell^2(G_n)} \text.
  \end{equation}

  Next, a simple computation shows (like with $\lambda$) that the unitary $\rho(k)$ maps $\dom{M_{\mathds{L}_H}}$ to itself, and for all $\xi \in \ell^2(G_\infty)$ and $h \in G_\infty$:
  \begin{equation*}
    [M_{\mathds{L}_H},\rho(k)] \xi(h) = (\mathds{L}_H(h)-\mathds{L}_H(hk))\sigma(hk,k^{-1})\xi(hk)
  \end{equation*}
  so $\opnorm{[M_{\mathds{L}_H},\rho(k)] \xi}{}{\ell^2(G_n)} \leq \sup_{h \in G_n} |\mathds{L}_H(h)-\mathds{L}_H(hk)| \norm{\xi}{\ell^2(G_n)} \leq \mathds{L}_H(k) \norm{\xi}{\ell^2(G_n)}$. Choosing $\xi = \delta_1$, we obtain
  \begin{equation}\label{Drho-eq}
    \opnorm{[M_{\mathds{L}_H},\rho(k)]}{}{\ell^2(G_n)} = \mathds{L}_H(k) \text.
  \end{equation}
  
  Now, since $\rho(k)$ commutes with $\lambda(g)$ for all $g \in G_\infty$, we conclude $[b,\rho(k)] = 0$, and thus, on $\dom{M_{\mathds{L}_H}}$
  \begin{align}\label{Dbrho-eq}
      [M_{\mathds{L}_H},b]\rho(k)
      &= M_{\mathds{L}_H} b \rho(k) - b M_{\mathds{L}_H}  \rho(k) \\
      &= M_{\mathds{L}_H} \underbracket[1pt]{\rho(k) b}_{[b,\rho(k)]=0} - b \rho(k) M_{\mathds{L}_H} - b [M_{\mathds{L}_H},\rho(k)] \nonumber \\
      &= [M_{\mathds{L}_H},\rho(k)] b + \rho(k) M_{\mathds{L}_H} b - \underbracket[1pt]{\rho(k) b}_{[\rho(k),b]=0} M_{\mathds{L}_H} - b [M_{\mathds{L}_H},\rho(k)] \nonumber \\
      &= [M_{\mathds{L}_H},\rho(k)] b + \rho(k) [M_{\mathds{L}_H}, b] - b [M_{\mathds{L}_H},\rho(k)] \text. \nonumber
  \end{align}
  Therefore, by Equation \eqref{eq:comm},
  \begin{align*}
    \;&\; \opnorm{[\Dirac_\infty,b^\circ]}{}{\ell^2(G_n k,E)}\\
    &= \opnorm{[M_{\mathds{L}_H},b]\rho(k)}{\ell^2(G_n k)}{\ell^2(G_n)} \\
    &\leq \underbracket[1pt]{\opnorm{[M_{\mathds{L}_H},\rho(k)] b}{\ell^2(G_n k)}{\ell^2(G_n)} + \opnorm{\rho(k) [M_{\mathds{L}_H}, b]}{\ell^2(G_n k)}{\ell^2(G_n)} + \opnorm{ b [M_{\mathds{L}_H},\rho(k)]}{\ell^2(G_n k)}{\ell^2(G_n)}}_{\text{by Eq. (\ref{Dbrho-eq}) }} \\
    &\leq \underbracket[1pt]{\opnorm{[M_{\mathds{L}_H},\rho(k)]}{\ell^2(G_n k)}{\ell^2(G_n)}}_{\leq \mathds{L}_H(k) \text{ by Eq. (\ref{Drho-eq})}}\norm{b}{C^\ast(G_n,\sigma)} + \underbracket[1pt]{\opnorm{\rho(k)}{\ell^2(G_n k)}{\ell^2(G_n)}}_{=1 \text{ as $\rho(k)$ is unitary}} \quad  \underbracket[1pt]{\opnorm{[M_{\mathds{L}_H},b]}{}{\ell^2(G_n)}}_{\leq \Lip_n(b) \text{ by Lemma (\ref{G-lemma})}} \\
    &\quad + \norm{b}{C^\ast(G_n,\sigma)} \underbracket[1pt]{\opnorm{[M_{\mathds{L}_H},\rho(k)]}{\ell^2(G_n k)}{\ell^2(G_n)}}_{\leq \mathds{L}_H(k) \text{ by Eq. (\ref{Drho-eq})}} \\
    &\leq \mathds{L}_H(k) \underbracket[1pt]{\norm{b}{C^\ast_{\mathrm{red}}(G_n,\sigma)}}_{\leq \frac{C}{2} \Lip_\infty(b)} + \Lip_n(b) + \norm{b}{C^\ast_{\mathrm{red}}(G_n,\sigma)} \mathds{L}_H(k) \\
    &\leq \frac{\varepsilon}{C^2} \frac{C}{2} \Lip_\infty(b) + \Lip_n(b) + \frac{C}{2} \Lip_\infty(b) \frac{\varepsilon}{C^2}  \\
    &\leq \Lip_n(b) + \frac{\varepsilon}{C}\Lip_\infty(b) \text.
  \end{align*}
  By Expression (\ref{L-sup-L-eq}), we thus get
  \begin{equation*}
    \Lip_\infty(b) \leq \Lip_n(b) + \frac{\varepsilon}{C}\Lip_\infty(b) \text.
  \end{equation*}

  Therefore, we have shown that since $\varepsilon\in\left(0,\frac{C}{2}\right)$,
  \begin{equation}\label{Lip-n-infty-eq}
    \forall b\in C_c(G_n) \quad \mathrm{tr}(b) =0 \implies \Lip_\infty(b) \leq \frac{1}{1-\frac{\varepsilon}{C}}\Lip_n(b) \text.
  \end{equation}

  Now, let $b \in C_c(G_n)$. We then easily compute:
  \begin{equation*}
    \Lip_\infty(b) = \Lip_\infty(b - \mathrm{tr}(b)1) \leq \frac{1}{1-\frac{\varepsilon}{C}}\Lip_n(b-\mathrm{tr}(b)1) = \frac{1}{1-\frac{\varepsilon}{C}}\Lip_n(b) \text.
  \end{equation*}
  
  Now, let $a \in \dom{\Lip_n}$ with $\Lip_n(a)\leq 1$. By assumption, there exists a sequence $(a_k)_{k\in\N}$ converging in $C^\ast_{\mathrm{red}}(G_n,\sigma)$ to $a$ such that $\Lip_n(a_k)\leq 1$ and $a_k \in C_c(G_n)$ for all $k \in \N$.  We thus have, by lower semicontinuity of $\Lip_n$, and Expression (\ref{Lip-n-infty-eq}):
  \begin{align*}
    \Lip_\infty(a)\leq\liminf_{k\rightarrow\infty}\Lip_\infty(a_k) \leq \frac{1}{1-\frac{\varepsilon}{C}}\liminf_{k\rightarrow\infty}\Lip_n(a_k) \leq \frac{1}{1-\frac{\varepsilon}{C}}\text.
  \end{align*}
  Thus, we have shown that, for all $n\geq N$, if $a\in \dom{\Lip_n}$, then $a \in \dom{\Lip_\infty}$, and moreover,
  \begin{equation*}
    \forall a\in\dom{\Lip_n} \quad \Lip_\infty(a) \leq \frac{1}{1-\frac{\varepsilon}{C}}\Lip_n(a) \text.
  \end{equation*}

  It is immediate by construction that $\Lip_n\leq  \Lip_\infty$ on $\dom{\Lip_n}$. Thus we have proven that for all $n\geq N$ and $k\geq n$, we have $\Lip_k\leq\Lip_\infty\leq\frac{1}{1-\frac{\varepsilon}{C}}\Lip_k(a)$. As a byproduct of this, we have shown that $\lim_{k\rightarrow\infty}\LipschitzD((C^\ast(G_n,\sigma),\Lip_k),(C^\ast(G_n,\sigma),\Lip_\infty)=0$.

  We now pause to note that, thanks to our identifications discussed prior to this theorem, and the observation that $\dom{\Lip_n}\subseteq\dom{\Lip_\infty}$ which we have just now established, $(C^\ast_{\mathrm{red}}(G_n,\sigma),\ell^2(G_n)\otimes E,\Dirac_n)_{n\in\N}$ is an inductive sequence of spectral triples in the sense of \cite{Floricel19}, where the $\ast$-morphisms from $C^\ast_{\mathrm{red}}(G_n,\sigma)$ to $C^\ast_{\mathrm{red}}(G_{n+1},\sigma)$ and the linear isometry from $\ell^2(G_n)$ to $\ell^2(G_{n+1})$ are just the inclusion maps. Moreover $(C^\ast_{\mathrm{red}}(G_\infty,\sigma),\ell^2(G_\infty,E),\Dirac_\infty)$ is indeed the inductive limit of this system.

  We now note that since $\Lip_\infty\leq\left(\frac{1}{1-\frac{\varepsilon}{C}}\right)\Lip_n$ and  $\varepsilon\in\left(0,\frac{C}{2}\right)$, we have
  \begin{equation*}
    \qdiam{C^\ast_{\mathrm{red}}(G_n,\sigma)}{\Lip_n} \leq \left(\frac{1}{1-\frac{\varepsilon}{C}}\right) \qdiam{C^\ast_{\mathrm{red}}(G_\infty,\sigma)}{\Lip_\infty} = \frac{C^2}{2(C-\varepsilon)} \leq C \text.
  \end{equation*}

  Let $b\in \dom{\Lip_n}$, and let $a=\left(1-\frac{\varepsilon}{C}\right)b \in\dom{\Lip_\infty}$. We then compute:
  \begin{align*}
    \norm{b-a}{\C^\ast_{\mathrm{red}}(G_\infty,\sigma)}
    &=\norm{b-\left(1-\frac{\varepsilon}{C}\right)b}{C^\ast_{\mathrm{red}}(G_\infty, \sigma)}\\
    &\leq \frac{\varepsilon}{C} \norm{b}{C^\ast_{\mathrm{red}}(G_\infty,\sigma)} \\
    &\leq \frac{\varepsilon}{C} \qdiam{C^\ast_{\mathrm{red}}(G_n,\sigma)}{\Lip_n}\Lip_n(b) \\
    &\leq \frac{\varepsilon}{C} C \, \Lip_n(b) = \varepsilon\,  \Lip_n(b)\text,
  \end{align*}
  while
  \begin{equation*}
    \Lip_\infty(a) = \Lip_\infty\left(\left(1-\frac{\varepsilon}{C}\right)b\right) \leq \frac{1}{1-\frac{\varepsilon}{C}} \Lip_n\left(\left(1-\frac{\varepsilon}{C}\right)b\right) = \Lip_n(b) \text.
  \end{equation*}

  Hence, if $n\geq N_2$, then:
  \begin{itemize}
  \item $\forall a \in \dom{\Lip_\infty} \quad \exists b \in \dom{\Lip_n}: \quad \Lip_n(b)\leq\Lip_\infty(a) \text{ and }\norm{b-a}{C^\ast_{\mathrm{red}}(G_\infty,\sigma)}<\varepsilon\,  \Lip_{\infty}(a)$,
  \item $\forall b \in \dom{\Lip_n} \quad\exists a \in \dom{\Lip_\infty}: \quad \Lip_\infty(a)\leq \Lip_n(b) \text{ and }\norm{a-b}{C^\ast_{\mathrm{red}}(G_\infty,\sigma)} < \varepsilon\,  \Lip_n(b)$.
  \end{itemize}
  Therefore, by Theorem (\ref{main-spec-thm}), we conclude that
  \begin{equation*}
    \lim_{n\rightarrow\infty} \spectralpropinquity{}((C^\ast_{\mathrm{red}}(G_n,\sigma),\ell^2(G_n,E),\Dirac_n),(C^\ast_{\mathrm{red}}(G_\infty,\sigma),\ell^2(G_\infty,E),\Dirac_\infty)) = 0 \text,
  \end{equation*}
  as claimed.
\end{proof}

We now wish to apply Theorem (\ref{main-group-thm}) to the family in Example (\ref{nc-solenoid-ex}), as well as to the Bunce-Deddens algebras. Thus, we shall now focus on Abelian groups.

So from now on we assume   that $G_\infty$ is Abelian. Therefore we will employ the additive notation for the groups $G_n$ ($n\in\Nbar$).
Since Abelian groups are amenable, we will also from now on identify $C^\ast_{\mathrm{red}}(G_n,\sigma)$ with $C^\ast(G_n,\sigma)$ for all $n\in\Nbar$.

\medskip

A key condition for Theorem (\ref{main-group-thm}) is always met when working with Abelian groups, as seen in the following lemma.

\begin{lemma}\label{domain-lemma}
  With the assumptions and notation of Subsection (\ref{spectral-triples-section}), for any $n\in\Nbar$, if $G_n$ is Abelian, then we have that
  \begin{equation*}
    \quad \left\{ a \in \dom{\Lip_n} : \Lip_n(a) \leq 1 \right\} = \closure{\left\{ a\in C_c(G_n) :\Lip_n(a) \leq 1 \right\} }\text.
  \end{equation*}
\end{lemma}

\begin{proof}
  Fix $n\in\Nbar$.  Since $\Lip_n$ is lower semicontinuous, we get
  \begin{equation*}
    \closure{\left\{a\in\dom{\Lip_n}\cap C_c(G_n) : \Lip_n(a)\leq 1\right\}} \subseteq \left\{ a \in \dom{\Lip_n} : \Lip_n(a) \leq 1 \right\}\text.
  \end{equation*}
  We now prove that when $G_n$ is Abelian, the converse inclusion holds.

  Let $\widehat{G_n}$ be the Pontryagin dual of $G_n$ (we will use the multiplicative notation for $\widehat{G_n}$). The dual action $\beta$ of $\widehat{G_n}$ on $C^\ast(G_n,\sigma)$ is unitarily implemented by defining, for each $z\in \widehat{G_n}$, the unitary $v^z$ of $\ell^2(G_n,E)$ which is given by,  for all $\xi \in \ell^2(G_n)\otimes E$:
\begin{equation*}
 \quad v^z \xi : g \in G_n \longmapsto \overline{z(g)} \xi(g) ( \, = z(-g)\xi(g) )\text.
\end{equation*}

It is easily checked that $z\in \widehat{G_n}\mapsto v^z$ is a unitary representation of $\widehat{G_n}$. We then note that:
\begin{equation*}
  \forall z \in \widehat{G_n} \quad  v^z \lambda_E(g) (v^z)^\ast = \beta^z \lambda_E(g) \text.
\end{equation*}

By construction, $\Dirac_n$ commutes with $v^z$ for all $z \in \widehat{G_n}$, so $\beta$ acts by full quantum isometries on $(C^\ast(G_n,\sigma),\Lip_n)$.

Let $\mu$ be the Haar probability measure on $\widehat{G_n}$. As seen in \cite[Lemma 3.1]{Latremoliere05},\cite[Theorem 8.2]{Rieffel00}, there exists a sequence $(\varphi_k)_{k\in\N}$ of non-negative functions over $\widehat{G_n}$, each obtained as a linear combination of characters of $\widehat{G_n}$ (i.e. of the form $z\in \widehat{G_n} \mapsto z(g)$ for some $g \in G_n$, by Pontryagin duality), such that $\int_{\widehat{G_n}}\varphi_k \,d\mu = 1$ for all $k\in \N$, and $(\varphi_k)_{k\in\N}$ converges, in the sense of distributions, to the Dirac measure at $1 \in \widehat{G_n}$, i.e., for all $f \in C(\widehat{G_n})$,
\begin{equation*}
  \lim_{k\rightarrow\infty} \int_{\widehat{G_n}} f(z) \, \varphi_k(z)d\mu(z) = f(1) \text.
\end{equation*}
We define, for each $k \in \N$, the continuous linear endomorphism:
\begin{equation*}
  \beta^{\varphi_k}: a \in C^\ast(G_n,\sigma) \mapsto \int_{\widehat{G_n}} \beta^z(a) \, \varphi_k(z)d\mu(z) \text,
\end{equation*}
acting on $C^\ast(G_n,\sigma)$. Since the dual action is strongly continuous, we conclude that, for all $a \in C^\ast(G_n,\sigma)$:
\begin{equation*}
   \quad \lim_{k\rightarrow\infty} \norm{\beta^{\varphi_k}(a) - a}{C^\ast(G_n)} = 0 \text.
\end{equation*}

Since $\Lip_n$ is lower semicontinuous, $\varphi_k \geq 0$ and $\int_{\widehat{G_\infty}}\varphi_k\,d\mu=1$ for all $k\in\N$, and $\beta$ acts by quantum isometries, we also get, for all $a\in\dom{\Lip_n}$,
\begin{equation*}
  \Lip_n\left(\beta^{\varphi_k}(a)\right) \leq \int_{\widehat{G_n}}\varphi(z) \Lip_{n}(a) \,d\mu(z) = \Lip_n(a) \text.
\end{equation*}
As a quick digression, lower semicontinuity also implies that $\Lip_n(a) \leq \liminf_{k\rightarrow\infty} \Lip_n(\beta^{\varphi_k(a)})$, so altogether we have shown that $\Lip_n(a) = \liminf_{k\rightarrow\infty}\Lip_n(\beta^{\varphi_k}(a)))$.

For each $k \in \N$, as $\varphi_k$ is a linear combination of characters of $\widehat{G_n}$, there exists a finite subset $F\subseteq G_n$ and a function $t : F\rightarrow\C$ such that $\varphi_k : z \in \widehat{G_n} \mapsto \sum_{g \in F} t(g) z(g)$; the range of $\beta^{\varphi_k}$ is then the finite dimensional subspace of $C_c(G_n)$ consisting of the functions supported on $F$. For our purpose, the main observations here are that, given $a\in \dom{\Lip_n}$, and $\varepsilon > 0$, there exists $K \in \N$ such that if $k\geq K$, then $\norm{a-\beta^{\varphi_k}(a)}{C^\ast(G_n,\sigma)} < \varepsilon$ and $\Lip_{n}(\beta^{\varphi_k}(a))\leq\Lip_n(a)$. In particular, again since $\Lip_n$ is lower semi-continuous, it follows that:
\begin{equation}\label{fejer-eq}
  \left\{ a\in\dom{\Lip_\Dirac}: \Lip_\Dirac(a) \leq 1 \right\} = \closure{\left\{ a \in \dom{\Lip_\Dirac}\cap C_c(G_n) : \Lip_\Dirac(a) \leq 1 \right\}} \text,
\end{equation}
as claimed.
\end{proof}

\begin{remark}
  With  the notation of the proof of Lemma (\ref{domain-lemma}), fix $\varphi\in\StateSpace(C^\ast(G_n,\sigma))$. Since, for all $k \in \N$, we have $\int_{\widehat{G_n}}\varphi^k\,d\mu=1$, we conclude that $\beta^{\varphi_k}$ is a unital map, and thus
\begin{multline*}
  \sup\left\{ \norm{a-\beta^{\varphi_k}(a) }{C^\ast(G_n)} : a\in\dom{\Lip_n}, \Lip_n(a)\leq 1 \right\} \\
  = \sup\left\{ \norm{a-\beta^{\varphi_k}(a) }{C^\ast(G_n)} : a\in\dom{\Lip_n}, \Lip_n(a)\leq 1, \mu(a) = 0 \right\} 
\end{multline*}
where the second supremum is indeed finite since $X = \{a\in\dom{\Lip_n}:\Lip_n(a)\leq 1,\mu(a) = 0\}$ is compact and we take the supremum of a continuous function over this set. In fact, Arzel\`a-Ascoli theorem can be applied here to prove that the convergence of $(\beta^{\varphi_k})_{k\in\N}$ to the identity on $X$ is uniform, though we here offer a simple $\frac{\varepsilon}{3}$-type argument. First, note that for all $a,b \in C^\ast(G_\infty)$, and for all $k\in\N$,
  \begin{equation*}
    \norm{\beta^{\varphi_k}(a)-\beta^{\varphi_k}(b)}{C^\ast(G_\infty)} \leq \int_{\widehat{G_\infty}} \norm{a-b}{C^\ast(G_\infty)} \, \varphi^k(z) d\mu(z) = \norm{a-b}{C^\ast(G_\infty)} \text.
  \end{equation*}

  Moreover, for all $\varepsilon > 0$, there exists a finite $\frac{\varepsilon}{3}$-dense subset $X_\varepsilon$ of $X$; as $X_\varepsilon$ is finite, there exists $K\in\N$ such that, for all $k\geq K$ and for all $a\in X_\varepsilon$, then $\norm{a-\beta^{\varphi_k}(a)}{C^\ast(G_\infty)} < \frac{\varepsilon}{3}$, as seen above; therefore for all $k\geq K$, we have
  \begin{align*}
    \norm{a-\beta^{\varphi_k}(a)}{C^\ast(G_\infty)}
    &\leq \norm{a-a'}{C^\ast(G_\infty)} + \norm{a'-\beta^{\varphi_k}(a')}{C^\ast(G_\infty)} + \norm{\beta^{\varphi_k}(a' - a)}{C^\ast(G_\infty)} \\
    < \frac{\varepsilon}{3} + \frac{\varepsilon}{3} + \frac{\varepsilon}{3} = \varepsilon \text.
  \end{align*}
  This proves that indeed, $(\beta^{\varphi_k})_{k\in\N}$ converges \emph{uniformly} to the identity over $X$.
\end{remark}

\medskip

We will prove that some of the spectral triples introduced in Subsection (\ref{spectral-triples-section}) are metric by invoking a property central to the work in \cite{Rieffel15b,LongWu21}, called \emph{bounded doubling}, which we now recall in the formulation of \cite{LongWu21}.

\begin{definition}[\cite{Rieffel15b,LongWu21}]
  A proper length function $\mathds{L}$ on a discrete group $G$ satisfies the \emph{bounded doubling property} when there exists $\theta>1$ and $c>0$ such that, for all $r \geq 1$:
  \begin{equation*}
    \left|\left\{g\in G : \mathds{L}(g) \leq \theta \cdot r\right\}\right| \leq c \left| \left\{ g \in G : \mathds{L}(g)\leq r \right\} \right| \text.
  \end{equation*}
\end{definition}

The bounded doubling property indeed ensures the following result.
\begin{lemma}\label{metric-spectral-triples-lemma}
  The spectral triples constructed in Subsection (\ref{spectral-triples-section}) are metric if the proper length function $\mathds{L}\coloneqq\max\{\mathds{L}_H,\mathds{F}\}$ has the bounded doubling property.
\end{lemma}

\begin{proof}
  We note that Lemma (\ref{length-lemma}) proves that $\mathds{L}$ is a proper unbounded length function.
  
  By \cite{Rieffel15b,LongWu21}, since all our groups are Abelian hence nilpotent, for any $\mu \in \StateSpace(C^\ast(G_n),\sigma)$, the set
  \begin{equation*}
    \left\{ a \in C_c(G_n) : \opnorm{[M_{\mathds{L}},a]}{}{\ell^2(G_n)}\leq 1, \mu(a) = 0 \right\}
  \end{equation*}
  is totally bounded. Since $\opnorm{[M_{\mathds{L}},\cdot]}{}{\ell^2(G_n)} \leq \Lip_n$ on $C_c(G_n)$, we thus conclude that
  \begin{equation*}
    \left\{ a \in C_c(G_n) : \Lip_n(a) \leq 1, \mu(a) = 0 \right\} \subseteq \left\{ a \in C_c(G_n) : \opnorm{[M_{\mathds{L}},a]}{}{\ell^2(G_n)}\leq 1, \mu(a) = 0 \right\}
  \end{equation*}
  and thus $\left\{ a \in C_c(G_n) : \Lip_n(a) \leq 1, \mu(a) = 0 \right\}$ is also totally bounded. By Lemma (\ref{domain-lemma}), we also have:
  \begin{equation*}
    \{a\in\dom{\Lip_n} : \Lip_n(a)\leq 1, \mu(a) = 0\} =\closure{\left\{ a \in C_c(G_n) : \Lip_n(a) \leq 1, \mu(a) = 0 \right\}}
  \end{equation*}
  so $\{a\in\dom{\Lip_n}:\Lip_n(a)\leq 1,\mu(a) = 0\}$ is compact. Thus by Theorem (\ref{tt-thm}), $\Lip_n$ is a Lipschitz seminorm, i.e. our spectral triples are metric.
\end{proof}

We are now ready to establish the following theorem.

\begin{theorem}\label{main-Abelian-group-thm}
  Let $G = \bigcup_{n\in\N} G_n$ be an Abelian discrete group, arising  as the union of a strictly increasing sequence $(G_n)_{n\in\N}$ of subgroups of $G$. Let $\sigma$ be a $2$-cocycle of $G$ and  $\mathds{L}_H$ a length function on $G$ such that
  \begin{equation*}
    \lim_{n\rightarrow\infty} \Haus{\mathds{L}_H}(G_n,G) = 0,
  \end{equation*}
  and whose restriction to $G_n$ is proper for all $n\in\N$.  Assume $\mathrm{scale} : \N \rightarrow [0,\infty)$  is a strictly increasing, unbounded function such that, if we set
  \begin{equation*}
    \mathds{F} : g \in G \longmapsto \mathrm{scale}(\min\{n\in\N : g \in G_n\})
  \end{equation*}
  then the proper length function $\mathds{L} \coloneqq \max\{ \mathds{L}_H, \mathds{F} \}$ has the \emph{bounded doubling property}.

  Then, for any Hermitian space $E$,
  \begin{equation*}
    \lim_{n\rightarrow\infty} \spectralpropinquity{}((C^\ast(G,\sigma),\ell^2(G)\otimes E,\Dirac),(C^\ast(G_n,\sigma),\ell^2(G_n)\otimes E,\Dirac_n)) = 0 \text,
  \end{equation*}
  where
  \begin{itemize}
  \item $\Dirac = M_{\mathds{L}_H}\otimes\gamma_1 + M_{\mathds{F}}\otimes\gamma_2$ on $\left\{ \xi \in \ell^2(G)\otimes E : \sum_{g\in G} (\mathds{L}_H(g)^2 + \mathds{F}(g)^2)\norm{\xi(g)}{E}^2 < \infty \right\}$, with $\gamma_1,\gamma_2$ unitaries of $E$ such that, for all $j,k \in \{1,2\}$:
    \begin{equation*}
      \gamma_j\gamma_k + \gamma_k\gamma_j = \begin{cases} 2 \text{ if $j=k$,} \\ 0 \text{ otherwise.} \end{cases}
    \end{equation*}
  \item $\ell^2(G_n)\otimes E$ is identified with the subspace of $G_n$-supported vectors in $\ell^2(G)\otimes E$,
  \item $\Dirac_n$ is the restriction of $\Dirac$ to $\dom{\Dirac}\cap\left(\ell^2(G_n)\otimes E\right)$,
  \item $C^\ast(G,\sigma)$ and $C^\ast(G_n,\sigma)$ act via their left regular $\sigma$-projective representations.
  \end{itemize}
\end{theorem}

\begin{proof}
	Our theorem follows from Theorem (\ref{main-group-thm}).
 We first note that  Lemma (\ref{metric-spectral-triples-lemma}) proves that all our spectral triples are metric.  By Lemma (\ref{domain-lemma}), since $G_\infty$ is Abelian, we conclude that, for all $n\in\Nbar$,
  \begin{equation*}
    \left\{a\in\dom{\Lip_n}:\Lip_n(a)\leq 1\right\} = \closure{\left\{a\in C_c(G_n) : \Lip_n(a)\leq 1\right\}} \text.
  \end{equation*}
  Since  all hypotheses of Theorem (\ref{main-group-thm}) are met, the result follows.
\end{proof}

In particular, for the noncommutative solenoids of Example (\ref{nc-solenoid-ex}), we obtain the following.

\begin{corollary}\label{main-cor}
  Fix a prime number $p \in \N$ and $d \in \N\setminus\{0,1\}$. For each $n\in\N$, let
  \begin{equation*}
    G_n \coloneqq \left( \frac{1}{p^n} \Z \right)^d
  \end{equation*}
  and
  \begin{equation*}
    G_\infty \coloneqq \left( \Z\left[\frac{1}{p}\right] \right)^d \text.
  \end{equation*}

  Fix a $2$-cocycle $\sigma$ on $G_\infty$ such that $\forall g \in G_\infty \quad \sigma(g,-g) = 1$.
  
  Let $\mathds{L}_H$ be the restriction to $G_\infty$ of some norm on $\R^2$. We define $\mathds{F}$ by setting, for all $g \in G_\infty$:
  \begin{equation*}
    \mathds{F}(g) \coloneqq \min\left\{ p^n : g \in \left(\frac{1}{p^n}\right)^d \right\}\text.
  \end{equation*}
  
  Let $E$ be an even dimensional hermitian space, with $\gamma_1,\gamma_2$ be two unitaries on $E$ such that, for all $j,k\in\{1,2\}$:
  \begin{equation*}
    \gamma_j \gamma_k + \gamma_k \gamma_j =
    \begin{cases}
      2 \text{ if $j=k$,} \\
      0 \text{ otherwise.}
    \end{cases}
  \end{equation*}
  
  If we define, for all $n\in\Nbar$, the operator
  \begin{equation*}
    \Dirac_n \coloneqq M_{\mathds{L}_H}\otimes\gamma_1 + M_{\mathds{F}}\otimes\gamma_2 \text{ on }\dom{\Dirac_n}
  \end{equation*}
  on the domain
  \begin{equation*}
    \dom{\Dirac_n} \coloneqq \left\{ \xi \in \ell^2(G_n,E) : \sum_{g \in G_n} (\mathds{L}_H(g)^2 + \mathds{F}(g)^2)\norm{\xi(g)}{E}^2 < \infty \right\} \text,
  \end{equation*}
  then, for all $n\in \Nbar$, the triple $(C^\ast(G_n,\sigma),\ell^2(G_n,E),\Dirac_n)$ is a metric spectral triple, and:
  \begin{equation*}
    \lim_{n\rightarrow\infty}\spectralpropinquity{}((C^\ast(G_n,\sigma),\ell^2(G_n,E),\Dirac_n),(C^\ast(G_\infty,\sigma),\ell^2(G_\infty,E),\Dirac_\infty)) = 0 \text.
  \end{equation*}
  Moreover, for each $n \in \N$, the sequence $(C^\ast(G_n,\sigma),\Lip_k)_{k\geq n}$ of {\qcms s} converge to $(C^\ast(G_n,\sigma),\Lip_\infty)$ in the Lipschitz distance.
\end{corollary}

\begin{proof}
  We first establish the bounded doubling property of certain related length functions. 
  
  Fix a prime number $p$ and $d \geq 2$. For all $g \in G_\infty$, let
  \begin{equation*}
    \mathds{L}'(g) = \max\left\{ \norm{g}{\R^d}, p^{\min\left\{ n \in \N : g \in \left(\frac{1}{p^n}\Z\right)^d\right\}} \right\}\text,
  \end{equation*}
  where the norm we choose on $\R^d$ for this proof is the max norm. By Lemma (\ref{length-lemma}), the function $\mathds{L}'$ is an unbounded proper length function. By Lemma (\ref{G-lemma}), we have that $\opnorm{[M_{\mathds{L}'},\cdot]}{}{\ell^2(G_n)} \leq \Lip_n$ on $C(G_n)$ for all $n\in\Nbar$. By \cite{Connes89}, the triple $(C^\ast(G_n,\sigma),\ell^2(G_n), M_{\mathds{L}'})$ is a spectral triple.

  Assume $\mathds{L}'(g) \leq p^n$. Since $g \in \left(\frac{1}{p^n}\Z\right)^d$, we can write $g = \left(\frac{a_j}{p^n}\right)_{1\leq j\leq d}$ for $a_1,\ldots,a_d \in \Z$. Since $\norm{g}{\R^d}\leq p^n$, we also have $a_1,\ldots,a_d \in [-p^{2n}, p^{2n}]$. Conversely, if $g = \left(\frac{a_j}{p^{n}}\right)_{1\leq j\leq d}$ with $-p^{2n} \leq a_j \leq p^{2n}$ for all $j\in\{1,\ldots,d\}$, then $\mathds{L}'(g) \leq p^n$ by definition. Hence, the closed ball of center $(0,0)$ and radius $p^n$ has cardinal $(2 p^{2n} + 1)^d$.

  Consequently:
  \begin{align*}
    \left|\left\{g\in G_\infty : \mathds{L}'(g) \leq p^{n+1}\right\}\right|
    &= (2 p^{2n+2} + 1)^d \\
    &\leq ( 2 p^{2n+2} + p^2)^d \\
    &= p^{2d} (2 p^{2n} + 1)^d \\
    &\leq p^{2d} \left|\left\{ g \in G_\infty : \mathds{L}'(g) \leq p^n \right\}\right| \text.
  \end{align*}

  Therefore, $\mathds{L}'$ is a proper unbounded length with the bounded doubling property.

  Let $\mathds{L}_H$ be any norm on $\R^d$. Since all the norms on $\R^d$ are equivalent, there exists $C>0$ such that $\frac{1}{C} \mathds{L}_H\leq\norm{\cdot}{\R^d} \leq C \mathds{L}_H$. Then
  \begin{equation*}
    \frac{1}{C} (\max\{\mathds{L}_H,\mathds{F}\}) \leq \mathds{L}' \leq C \max\left\{\mathds{L}_H,\mathds{F}\right\}\text.
  \end{equation*}

  Therefore,
  \begin{equation*}
    \left|\left\{g\in G_\infty : \max\left\{\mathds{L}_H(g),\mathds{F}(g)\right\} \leq p^{n+1}\right\} \right|\leq C^2 p^{2d}\left|\left\{g\in G_\infty : \max\left\{\mathds{L}_H(g),\mathds{F}(g)\right\} \leq p^{n}\right\}\right|\text.
  \end{equation*}

  Write $\mathds{L} \coloneqq \max\{\mathds{L}_H,\mathds{F}\}$ on $C_c(G_n)$. We thus have shown that $\mathds{L}$, which is unbounded and proper by Lemma (\ref{length-lemma}), also has the bounded doubling property.

  By Lemma (\ref{domain-lemma}), since $G_\infty$ is Abelian, we conclude that
  \begin{equation*}
    \forall n \in \Nbar \quad \left\{a\in\dom{\Lip_n}:\Lip_n(a)\leq 1\right\} = \closure{\left\{a\in C_c(G_n) : \Lip_n(a)\leq 1\right\}} \text.
  \end{equation*}
  Thus, our corollary follows from Theorem (\ref{main-group-thm}).
\end{proof}

We can choose somewhat different length functions over $\left(\Z\left[\frac{1}{p}\right]\right)^d$, by varying not only $\mathds{L}_H$,  but also $\mathds{F}$. For instance, Corollary (\ref{main-cor}) remains valid if we replace $\mathds{F}$ by $\mathds{F}':(g_1,\ldots,g_d)\in G_\infty\mapsto \max_{j=1}^d |g_j|_p$, where $|\cdot|_p$ is now the $p$-norm. The resulting length function $\max\{\mathds{L}_H,\mathds{F}'\}$ has the bounded doubling property, as seen by applying \cite[Proposition 3.17]{FarsiPacker22} up to an equivalence of metrics. We also note that for this construction to give us something different from Corollary (\ref{main-cor}), we require that $\mathds{L}_H(g) < \mathds{F}'(g)$ for at least one $g \in \Z^d\setminus\{0\}$. In general, the difference is only up to a bounded perturbation of the underlying Dirac operator.

\bigskip

Another interesting family of C*-algebras to which our work applies are certain \emph{Bunce-Deddens algebras}.

\begin{notation}
  Let $\mathscr{P}$ be the set of all sequences $(\alpha_n)_{n\in\N}$ of nonzero natural numbers such that $\frac{\alpha_{n+1}}{\alpha_n}$ is a prime number for all $n\in\N$.
\end{notation}

\begin{notation}
  For any integer $m \in \Z$, we denote the quotient group $\faktor{\Z}{m\Z}$ simply by $\faktor{\Z}{m}$.
\end{notation}

\begin{notation}
  Let $\alpha\coloneqq(\alpha_n)_{n\in\N} \in \mathscr{P}$. If $n \in \N$, then $\alpha_n$ divides $\alpha_{n+1}$, and thus the map
  \begin{equation*}
    \rho_n: (m \mod \alpha_{n+1}) \in \faktor{\Z}{\alpha_{n+1}} \rightarrow \left(m \mod \alpha_n \right) \in \faktor{\Z}{\alpha_n}
  \end{equation*}
  where $x \mod y$ is the equivalence class of $x \in \Z$ modulo $y \in \Z\setminus\{0\}$, is a well-defined surjective group morphism. The projective limit of the projective sequence
  \begin{equation*}
    \faktor{\Z}{\alpha_0} \quad \xleftarrow{\ \ \rho_0\ \ }\quad \faktor{\Z}{\alpha_1} \quad\xleftarrow{\ \ \rho_1\ \ }\quad \faktor{\Z}{\alpha_2} \quad\xleftarrow{\ \  \rho_2\ \ }\quad \faktor{\Z}{\alpha_3} \quad\xleftarrow{\ \  \rho_4\ \ }\quad \cdots
  \end{equation*}
  is denoted by $\faktor{\Z}{\alpha}$. By construction, we observe that:
  \begin{equation*}
    \faktor{\Z}{\alpha} = \left\{ (z_n)_{n\in\N} \in \prod_{j=0}^\infty\faktor{\Z}{\alpha_n}: \rho_n(z_{n+1}) = z_n \right\} \text.
  \end{equation*}
  We endow $\faktor{\Z}{\alpha}$ with its topology as a projective space of compact spaces, i.e. with the topology induced by the product topology on $\prod_{j=0}^\infty\faktor{\Z}{\alpha_n}$, which is compact by Tychonoff theorem.
\end{notation}

We identify, for any $m\in\N\setminus\{0\}$, the Pontryagin dual $\dualZmod{m}$ of $\faktor{\Z}{m}$ with the subgroup of $\T$ of $m$-th roots of unity in the obvious manner --- while of course, $\faktor{\Z}{m}$ is self-dual, this identification will be helpful to our presentation. The Pontryagin dual $\Z(\alpha)\coloneqq \dualZmod{\alpha}$ of $\faktor{\Z}{\alpha}$ is thus, by contravariant functoriality, the limit of the inductive sequence:
\begin{equation*}
  \dualZmod{\alpha_0} \quad \xrightarrow{\ \  j_0 \ \ } \quad  \dualZmod{\alpha_1} \quad  \xrightarrow{\ \  j_1 \ \ } \  \dualZmod{\alpha_2} \quad  \xrightarrow{\ \  j_2\ \ }\quad  \dualZmod{\alpha_3} \quad  \xrightarrow{\ \  j_3 \ \ } \quad \cdots 
\end{equation*}
where $j_1$,$j_2$,$\ldots$, are simply the injection maps. Of course, by construction:
\begin{equation*}
  \Z(\alpha) = \left\{ \zeta \in \T : \exists n \in \N \quad \zeta^{\alpha_n} = 1 \right\}\text,
\end{equation*}
where $\T = \{ u \in \C : |u| = 1 \}$ is the circle group; moreover $\Z(\alpha)$ is a discrete group as the dual of a compact group (i.e. we \emph{do not} endow it with the topology inherited as a subset of $\T$).

The Pontryagin duality pairing between $\Z(\alpha)$ and its dual $\faktor{\Z}{\alpha}$ is given for all $\zeta\in Z(\alpha)$ and for all $z\coloneqq(z_n)_{n\in\N} \in \faktor{\Z}{\alpha}$ by $\zeta^z \coloneqq \lim_{n\rightarrow\infty} \zeta^{z_n}$, noting that the sequence $(\zeta^{z_n})_{n\in\N}$ is eventually constant, by construction.

In the special case when $\alpha=(p,p^2,p^3,\ldots)$, the group $\Z(\alpha)$ is the Pr{\"u}fer group $\Z(p^\infty)$ and the group $\faktor{\Z}{\alpha}$ is the group $\Z_p$ of $p$-adic integers.

\begin{lemma}\label{Zalpha-bounded-doubling-lemma}
  Let $\alpha\coloneqq(\alpha_n)_{n\in\N} \in \mathscr{P}$. Let $\mathds{L}_H$ be a length function over the circle group $\T$ restricted to $\Z(\alpha)$ such that $\lim_{n\rightarrow\infty} \Haus{\mathds{L}_H}\left(\dualZmod{\alpha_n}\; ,\Z(\alpha)\right) = 0$.
  
  For all $\zeta\in \Z(\alpha)$, we define
  \begin{equation*}
    \mathds{F}(\zeta) \coloneqq \min\left\{ p \in \N : \zeta^p = 1 \right\} \text.
  \end{equation*}
  
  Let $\norm{\cdot}{\R^2}$ be any monotone norm on $\R^2$. The function
  \begin{equation*}
    \mathds{L} : \zeta\in \Z(\alpha) \mapsto \norm{( \mathds{L}_H(\zeta) , \mathds{F}(\zeta) )}{\R^2}
  \end{equation*}
  is a proper unbounded length function over $\Z(\alpha)$.

  Moreover, $\mathds{L}$ has the bounded doubling property if, and only if, the sequence $\left(\frac{\alpha_{n+1}}{\alpha_n}\right)_{n\in\N}$ is bounded.
\end{lemma}

\begin{proof}
  First, it is easy to see that, for all $\zeta \in \Z(\alpha)\setminus\{1\}$,
  \begin{equation*}
    \mathds{F}(\zeta) = p^{\min\left\{ n \in \N : \zeta\in\dualZmod{\alpha_n}\right\}} \text,
  \end{equation*}
  while $\mathds{F}(1) = 0$. Therefore, by Lemma (\ref{length-lemma}), we already know that $\mathds{L}$ is a proper unbounded length function on $\Z(\alpha)$.

  For now, let us assume $\norm{\cdot}{\R^2}$ is the max norm.
  
  For any $\rho>0$, we write $B[\rho]$ the cardinality of the closed ball centered at $(1,0) \in \Z(\alpha)\times\Z$ of radius $\rho$. For any $d \in \N$, we compute the following expression:
  \begin{equation*}
    B\left[\alpha_d\right]
    =\left|\left\{ \zeta\in \Z(\alpha) : \mathds{L}(\zeta) \leq \alpha_d \right\}\right| = \left|\left\{ \zeta \in\dualZmod{\alpha_d} \right\}\right| 
    = \alpha_d \text.
  \end{equation*}

  Now, let $R \geq 1$. Then, there exists $d \in \N$ such that $\alpha_d \leq R \leq \alpha_{d+1}$. We note that since $B[R]\leq B[\alpha_{d+1}]<\infty$, our length function $\mathds{L}$ is indeed proper; we also note that since $B[R]\geq B[\alpha_d] = \alpha_d\geq 2^d$, the length function $\mathds{L}$ is also unbounded.

  Now, assume that $M \coloneqq \sup_{n\in\N}\frac{\alpha_{n+1}}{\alpha_n} < \infty$. We then compute:
  \begin{align*}
    B[2 R] \leq B[2 \alpha_{d+1}]
    &\leq B[\alpha_{d+2}] = \alpha_{d+2} \\
    &= \frac{\alpha_{d+2}}{\alpha_{d+1}}\frac{\alpha_{d+1}}{\alpha_d} \alpha_d \leq M^2 \alpha_d = M^2 B[\alpha_d] \leq M^2 B[R] \text. 
  \end{align*}
  Therefore, our length $\mathds{L}$ has the bounding doubling property. Now, if we allow for a different choice of monotone norm for $\norm{\cdot}{\R^2}$, then, as all norms on $\R^2$ are equivalent, the resulting length function still has the property of bounded doubling.

  Now, assume instead that $\sup_{n\in\N}\frac{\alpha_{n+1}}{\alpha_n} = \infty$. Let $n \in \N$, and let $r_n = \alpha_{n+1} / 2$. We then note, using our above computation, that
  \begin{equation*}
    B[2r_n] = \alpha_{n+1} = \frac{\alpha_{n+1}}{\alpha_n} \cdot  B[r_n] \text,
  \end{equation*}
  and thus $\frac{\alpha_{n+1}}{\alpha_n} = \frac{B[2r_n]}{B[r_n]}$ for all $n\in\N$; therefore, our length $\mathds{L}$ does not actually have the bounded doubling property. 
\end{proof}

\begin{corollary}
  Let $\alpha = (\alpha_n)_{n\in\N}$ be a sequence of nonzero natural numbers such that $\left(\frac{\alpha_{n+1}}{\alpha_n}\right)_{n\in\N}$ is a bounded sequence of prime numbers, and let
  \begin{equation*}
    \Z(\alpha) \coloneqq \left\{ \zeta \in \C : \exists n \in \N \quad \zeta^{\alpha_n} = 1 \right\}\text.
  \end{equation*}
  Define:
  \begin{equation*}
    G_\infty \coloneqq \Z(\alpha)\times \Z \text{ and }\forall n \in \N \quad G_n \coloneqq \dualZmod{\alpha_n}\times \Z \text,
  \end{equation*}
  i.e. $G_n = \{ (\zeta,z) \in G_\infty : z\in \Z, \zeta^{\alpha_n}=1 \}$. Let $\sigma$ be a $2$-cocycle of $G_\infty$.

  Let $\mathds{L}_Z$ be the restriction of any continuous length function on $\T$ to $\Z(\alpha)$, and  define $\mathds{L}_H : (u,z)\in G_\infty\mapsto \mathds{L}_Z(u) + |z|$. 
  
  For all $\zeta \in \Z(\alpha)$, set:
  \begin{equation*}
    \mathds{F}(\zeta) \coloneqq \min\{ n \in \N : u^n = 1 \} \text.
  \end{equation*}

  Let $E$ be a Hermitian vector space, and let $\gamma_1$,$\gamma_2$ be unitaries such that $\gamma_1\gamma_2 = -\gamma_2\gamma_1$ and $\gamma_1^2 = \gamma_2^2 = 1_E$.

  If we set, for all $n\in\Nbar$,
  \begin{equation*}
    \Dirac_n \coloneqq M_{\mathds{L}_H}\otimes\gamma_1 + M_{\mathds{F}}\otimes\gamma_2 \text,
  \end{equation*}
  then for all $n\in\N$, the spectral triple $(C^\ast(G_n,\sigma),\ell^2(G_n)\otimes E,\Dirac_n)$ is metric, and
  \begin{equation*}
    \lim_{n\rightarrow\infty} \spectralpropinquity{}\left( \left( C^\ast(G_n,\sigma), \ell^2(G_n)\otimes E, \Dirac_n\right), \left( C^\ast(\Z(\alpha)\times \Z,\sigma), \ell^2(\Z(\alpha)\times \Z)\otimes E, \Dirac_\infty\right)\right) = 0 \text.
  \end{equation*}
\end{corollary}

\begin{proof}
A straightforward computation shows that $|\cdot|$ is proper with the bounded doubling property.

By \cite[Proposition 3.7]{FarsiPacker22} applied to the proper unbounded lengths $|\cdot|$ and $\mathds{L}_Z$, we conclude that $\mathds{L}\coloneqq (\zeta,z)\in G_\infty \mapsto \mathds{L}_Z(\zeta) + \mathds{F}(\zeta) + |m|$ has the bounded doubling property.

Since $\mathds{L}_Z$ is continuous on $\T$, it induces the usual topology on $\T$ (as a subset of $\C$). Therefore, the topology of the Hausdorff distance $\Haus{\mathds{L}_H}$ is the Vietoris topology for the usual topology of $\T$, and thus the same as the topology induced by $\Haus{\T}$, when $\T$ is endowed with the restriction of the usual metric on $\C$. It then follows that:
\begin{equation*}
  \lim_{n\rightarrow\infty}\Haus{\mathds{L}_H}\left(\dualZmod{\alpha_n},\Z(\alpha)\right) = 0\text.
\end{equation*}

As all the other assumptions are now met, we conclude that our corollary holds, by Theorem (\ref{main-Abelian-group-thm}).
\end{proof}

The map
\begin{equation*}
  \varpi : z \in \Z \mapsto (z \mod \alpha_n)_{n\in\N} \in \faktor{\Z}{\alpha}
\end{equation*}
is an injective *-morphism of group with dense range. Now, we define the following automorphism of $\Z(\alpha)$:
\begin{equation*}
  \tau : u \in \Z(\alpha) \mapsto u+\varpi(1) \text.
\end{equation*}
The C*-crossed-product $C(\Z(\alpha))\rtimes_\tau \Z$ is the Bunce-Deddens algebra associated to the ``supernatural'' number
\begin{equation*}
  \mathbf{n}\coloneqq \left(p^{|\{n\in\N:\frac{\alpha_{n+1}}{\alpha_n}=p\}|}\right)_{p \text{ prime }}\text.
\end{equation*}
It is also *-isomorphic to $C^\ast(\Z(\alpha)\times Z,\sigma)$, as defined above, when $\sigma$ is the $2$-cocycle defined by setting, for all $(\zeta,z), (\eta,y) \in G_\infty$:
\begin{equation*}
  \sigma((\zeta,z),(\eta,y)) \coloneqq \eta^z  \text.
\end{equation*}

Indeed, this isomorphism can be obtained by using \cite{Packer89}. We begin with the observation that Bunce-Deddens algebras \cite{BunceDeddens75} are C*-crossed products \cite{Riedel82,Exel94}. Now, let us briefly explain the construction of this isomorphism. Since the natural inclusion $j : \Z(\alpha)\rightarrow\T$ is a character of $\Z(\alpha)$, it is given by the pairing with an element in $\faktor{\Z}{\alpha}$; this element is precisely our $\varpi(1)$ defined above. In our case, we note that $\lambda_{(1,1)} \lambda_{\zeta,0} \lambda_{(1,1)}^\ast = \zeta^{-1} \lambda_{\zeta,0}$ for all $\zeta\in\Z(\alpha)$. If $f \in C_c\left(\faktor{\Z}{\alpha}\right)$, we denote its Fourier transform by $\widehat{f}$; specifically
\begin{equation*}
  \widehat{f}: \zeta\in\Z(\alpha)\mapsto \sum_{z \in \faktor{\Z}{\alpha}} f(z) \zeta^{-z} \text.
\end{equation*}
A straightforward computation shows that $\widehat{\tau(f)}(\zeta) = \zeta^{-1} \widehat{f}(\zeta)$. Thus, we conclude that $\lambda_{(1,1)}\lambda(\widehat{f})\lambda_{(1,1)}^\ast = \lambda\left(\widehat{\tau(f)}\right)$. A similar computation invoking the inverse Fourier transform can be done by using the canonical generators of the C*-crossed product $C\left(\faktor{\Z}{\alpha}\right)\rtimes_\tau\Z$. By universality of the C*-crossed-product and the twisted group C*-algebra (here, since our groups are Abelian, these algebras agree with their image by their left regular representations), we conclude the description of our isomorphism.

\medskip

Thus, we have constructed metric spectral triples over Bunce-Deddens algebra for bounded supernatural numbers, and these triples are limits of sequences of metric spectral triples for the spectral propinquity.

In particular, $C^\ast(\Z(\alpha)\times\Z,\sigma)$ is seen to be the inductive limit \emph{ and }the limit for the propinquity, with the quantum metrics described here, of the C*-algebras $C^\ast\Bigg(\dualZmod{\alpha_n}\allowbreak\times \Z,\sigma\Bigg)$ as $n\in\N$ approaches $\infty$. Notably, $C^\ast\left(\dualZmod{\alpha_n}\times \Z,\sigma\right)$ is actually *-isomorphic to the C*-algebra of continuous sections of a vector bundle over the circle $\T$ with fibers the algebras of square $\alpha_n$-matrices. This situation is of course reminiscent of the fact that in particular, Bunce-Deddens algebras are A$\T$ algebras. However, starting from the usual description of Bunce-Deddens algebras as A$\T$ algebras led to difficulties in \cite{Latremoliere20b}, where the quantum metrics on the Bunce-Deddens algebra do not arise from a spectral triple, and the convergence is only proven in the sense of Rieffel's quantum Gromov-Hausdorff distance. Thus, for Bunce-Deddens algebras associated with supernatural numbers consisting of only finitely many prime numbers, we have now constructed metric spectral triples which actually capture their inductive limit structure within our geometric framework. We hope that Theorems (\ref{main-group-thm}) and (\ref{main-Abelian-group-thm}) will prove useful in constructing other examples of metric spectral triples over twisted group C*-algebras for interesting inductive limits of groups.


\providecommand{\bysame}{\leavevmode\hbox to3em{\hrulefill}\thinspace}
\providecommand{\MR}{\relax\ifhmode\unskip\space\fi MR }
\providecommand{\MRhref}[2]{%
  \href{http://www.ams.org/mathscinet-getitem?mr=#1}{#2}
}
\providecommand{\href}[2]{#2}

\vfill
\end{document}